\documentclass[11pt,a4paper]{article}
\usepackage{amsmath,amsfonts,latexsym,amssymb,amsthm,array,verbatim,xcolor,bbm}
\usepackage{float}
\usepackage{graphicx}
\usepackage[text={16cm,24cm},top=2cm]{geometry} 

\newtheorem{theorem}{Theorem}[section]

\newtheorem{lemma}[theorem]{Lemma}
\newtheorem{proposition}[theorem]{Proposition}

\newtheorem{remark}[theorem]{Remark}

\newcommand{\N}{{\mathbb{N}}}  
\newcommand{\Z}{{\mathbb{Z}}}  
\newcommand{\Q}{{\mathbb{Q}}} 
\newcommand{\C}{{\mathbb{C}}}  

\numberwithin{equation}{section}

\author{Misha Feigin, Daniele Valeri, Johan Wright}

\AtEndDocument{\bigskip{\footnotesize%
  (M. Feigin, J. Wright) \textsc{School of Mathematics \& Statistics, University of Glasgow, Glasgow, UK} \par  
  \textit{E-mail addresses} \, \texttt{misha.feigin@glasgow.ac.uk, } 
  \texttt{j.wright.2@research.gla.ac.uk} \par
  \addvspace{\medskipamount}
  (D. Valeri) \textsc{Dipartimento di Matematica \& INFN, Sapienza Universit\`a di Roma, Roma, Italy} \par
  \textit{E-mail address} \,\texttt{daniele.valeri@uniroma1.it }
}}

\title{Flat coordinates of algebraic Frobenius manifolds in
small dimensions}
\date{}

\begin{document}
\maketitle
\begin{abstract}
Orbit spaces of the reflection representation of finite irreducible Coxeter groups provide polynomial Frobenius manifolds. Flat coordinates of the Frobenius metric $\eta$ are 
Saito polynomials which are distinguished basic invariants of the Coxeter group.

Algebraic Frobenius manifolds are typically related to quasi-Coxeter conjugacy classes in finite Coxeter groups. We find explicit relations between flat coordinates of the Frobenius metric $\eta$ and flat coordinates of the intersection form $g$ for most known examples of algebraic Frobenius manifolds up to dimension 4. In all the cases, flat coordinates of the metric $\eta$ appear to be algebraic functions on the orbit space of the Coxeter group.
\end{abstract}

\section{Introduction}
A Frobenius manifold with a polynomial prepotential is known as a polynomial Frobenius manifold. Coxeter orbit spaces, constructed from finite irreducible Coxeter groups, can be given the structure of a semisimple polynomial Frobenius manifold \cite{DubrovinNotes}. Dubrovin conjectured that these classify all irreducible semisimple polynomial Frobenius manifolds, which Hertling proved with the added assumption that the Euler vector field has positive degrees \cite{Claus}. A Frobenius manifold with an algebraic prepotential is known as an algebraic Frobenius manifold. This is a
natural case to consider after classifying the polynomial Frobenius manifolds.

The first non-rational algebraic Frobenius manifolds were found by Dubrovin and Mazzocco in 2000, which they derived from the Coxeter group $H_3$ in relation to Painlev\'e VI equation \cite{DubrovinH3}. Explicit prepotentials of these Frobenius manifolds were given more recently by Kato, Mano and Sekiguchi \cite{H3forReal} (see also Remark 6.1 in that paper).

The local monodromy group of a semisimple Frobenius manifold is generated by finitely many reflections \cite{DubrovinConjecture}. It comes together with a particular set of generating reflections $R_1, \dots, R_n,$ where $n$ is the dimension of the Frobenius manifold. In the case of an algebraic Frobenius manifold there is a finite orbit of the braid group $\mathfrak{B}_n$ acting on $n$-tuples of reflections in the monodromy group, this action is known as the Hurwitz action \cite{DubrovinH3}. The local monodromy group is then necessarily a finite group \cite{Michel}, and the product of reflections $R_i$ gives a quasi-Coxeter element $w$ in this group. An equivalent property of $w$ is that it does not belong to any proper reflection subgroup of the Coxeter group (see \cite{Theo}). 

It is expected that irreducible semisimple algebraic Frobenius manifolds are closely related to the quasi-Coxeter conjugacy classes of finite irreducible Coxeter groups, where polynomial Frobenius manifolds correspond to the conjugacy class of a Coxeter element \cite{DinarOriginal}.
We recall some findings of algebraic Frobenius manifolds below together with their links to quasi-Coxeter elements. It seems not clear though whether these constructions give the same quasi-Coxeter conjugacy class as described above following \cite{Theo}.

Pavlyk constructed bi-Hamiltonian structures of hydrodynamic type by considering dispersionless limit of generalised Drinfeld--Sokolov hierarchies
associated to a regular element of a Heisenberg subalgebra ${\mathcal H}_w$ of an affine Lie algebra $\widehat{\mathfrak{g}}$  
\cite{Pavlyk}. To make dispersionless limit finite one has to restrict analysis to  a suitable submanifold of the phase space. 
In this construction
the Heisenberg subalgebra ${\mathcal H}_w$ is associated with a regular quasi-Coxeter element $w$ of the Weyl group of the finite-dimensional Lie algebra $\mathfrak{g}$ (in general, non-equivalent Heisenberg subalgebras are in one-to-one correspondence with conjugacy classes of the Weyl group  \cite{KacPeterson}).  Dubrovin had previously shown that bi-Hamiltonian structures of hydrodynamic type have a correspondence with Frobenius manifolds \cite{DubrovinBiham}. Pavlyk claimed that his construction produces algebraic Frobenius manifolds and he gave an explicit expression for the prepotential in the case of the conjugacy class $D_4(a_1)$ \cite{Pavlyk} (in the notation for conjugacy classes of Weyl groups from Carter \cite{Carter}).

Dinar also gave a construction of algebraic Frobenius manifolds \cite{DinarLatest}. Starting with a regular quasi-Coxeter element $w$ in a Weyl group there is a distinguished nilpotent element $e$ in the associated simple Lie algebra $\mathfrak{g}$ \cite{Feher}, \cite{Springer}. Dinar constructed a bi-Hamiltonian structure of hydrodynamic type on a subvariety of the Slodowy slice $\mathcal{S}_e \subseteq \mathfrak{g}^*$ using Dirac reduction and gave an explicit expression for the prepotential in the case of nilpotent orbit $F_4(a_2)$ \cite{DinarOriginal} (in the notation for nilpotent orbits from \cite{Collingwood}). He also derived prepotentials for $D_4(a_1)$ \cite{DinarSubregular} and $E_8(a_1)$ \cite{DinarE8}, the latter of which was simplified in a joint work with Sekiguchi \cite{DinarSekiguchiE8}. The eigenvalues of the quasi-Coxeter element $w$ have the form $e^{\frac{2 \pi i}{|w|}\eta_j},$ where $|w|$ denotes the order of $w$ and $0 \leq \eta_j \leq |w|-1.$ The degrees $d_j$ of the corresponding Frobenius manifold are $d_j=\frac{\eta_j+1}{|w|}.$

Two algebraic prepotentials related to Weyl groups $E_6$ and $E_7,$ and six algebraic prepotentials related to the Coxeter group $H_4$ were found by Sekiguchi \cite{Sekiguchi} who used degrees of the latter Frobenius manifolds conjectured by Douvropoulos (see further details in \cite{Theo}). These prepotentials are denoted by $H_4(k),$ where $k=1, 2, 3, 4, 7, 9.$

The degrees of these Frobenius manifolds are determined as follows \cite{Theo}. For a regular quasi-Coxeter element $w$ with eigenvalues $e^{\frac{2 \pi i}{|w|}\eta_j},$
there exists a regular element $w_0 \in W$ such that $w$ is conjugate to $w_0^l$ for some $l \in \N, \, |w|=|w_0|$ and the eigenvalues of $w_0$ have the form $e^{\frac{2 \pi i}{|w|}(d_j^W-1)},$ where $d_j^W$ are the fundamental degrees of $W,$ assuming $l$ is the smallest such positive integer. Then, the degrees $d_j$ of the Frobenius manifolds $H_4(k)$ are
\begin{equation} \label{Equation1}
d_j=\frac{(\eta_j+l) \, (\mathrm{mod} \, |w|)}{|w|}, 
\end{equation}
where the remainder $(\eta_j+l) \, (\mathrm{mod} \, |w|)$ is between 1 and $|w|.$ The algebraic degree of these Frobenius manifolds also have a combinatorial interpretation \cite{Theo}.

\smallskip

Any Frobenius manifold has two compatible flat metrics $\eta$ and $g$, where metric $g$ is usually referred to as the intersection form.
It is a complicated problem in general to express one flat coordinate system in terms of the other. For polynomial Frobenius manifolds expressing flat coordinates of $\eta$ via that of $g$ gives a distinguished set of basic invariants of a Coxeter group, known as Saito polynomials \cite{Saito}. These polynomials play an important role in the representation theory of Cherednik algebras \cite{Cherednik}. 

Let us explain the relation between the two sets of flat coordinates for two-dimensional algebraic Frobenius manifolds. Prepotentials for two-dimensional (semisimple) algebraic Frobenius manifolds have the following form \cite{DubrovinNotes}:
\begin{equation} \label{Equation2}
F=\frac{1}{2}t_1^2t_2+\frac{k(2k)^k}{k^2-1}t_2^{k+1},
\end{equation}
where $k \in \Q \setminus \{-1, 0, 1\}.$ The degrees of the Frobenius manifold are $d_1=1$ and $d_2=\frac{2}{k},$ and the charge is $d=\frac{k-2}{k}.$ Let $x_1, \, x_2$ be flat coordinates of the intersection form. Then we have the relations
\begin{equation} \label{tx2dimEqs}
t_1=(x_1+ix_2)^k+(x_1-ix_2)^k, \hspace{10mm} t_2=\frac{x_1^2+x_2^2}{2k},
\end{equation}
which can be checked similarly to the polynomial case $k \in \N$ considered in \cite{DubrovinNotes}.

Now, let $w$ be a quasi-Coxeter element in the dihedral group $I_2(m).$ It must be the product of two reflections that generate $I_2(m).$ Hence $w=c^l,$ where $c$ is a Coxeter element of $I_2(m)$ and $(m, \, l)=1.$ The eigenvalues of $w$ are $e^{\pm \frac{2 \pi i}{m}l}.$ We can assume $l \leq \frac{m}{2}$ as the corresponding elements $w=c^l$ give representatives for all the quasi-Coxeter conjugacy classes. Then the smallest positive integer $r$ such that $w$ is conjugate to $c^r$ is $r=l.$ Thus, the degrees of the Frobenius manifold, using prescription (\ref{Equation1}) and following \cite{Theo}, are
\begin{equation*}
d_1=1, \hspace{10mm} d_2=\frac{2l}{m},
\end{equation*}
since $\eta_{1}=m-l, \, \eta_2=l$ and $|w|=m.$ From the general form (\ref{Equation2}) of a prepotential of an algebraic two-dimensional Frobenius manifold, we see that $k=\frac{m}{l}$ and thus
\begin{equation*}
F=\frac{1}{2}t_1^2t_2+\frac{ml}{m^2-l^2}\left(\frac{2m}{l}\right)^{\frac{m}{l}}t_2^{1+\frac{m}{l}},
\end{equation*}
and this has charge $d=\frac{m-2l}{m}.$ Note that when $l=1$ we get the polynomial two-dimensional Frobenius manifolds.

The Coxeter group $I_2(m)$ has basic invariants
\begin{equation} \label{basicInvariants2dim}
y_1=(x_1+ix_2)^m+(x_1-ix_2)^m, \hspace{10mm} y_2=\frac{x_1^2+x_2^2}{2m}.
\end{equation}
We can express these basic invariants in terms of the flat coordinates of the metric in the following form:
\begin{equation} \label{2dimAlgebraicRelations}
y_1=\left(\frac{t_1 + \sqrt{t_1^2-4\left(\frac{2m}{l}\right)^{\frac{m}{l}}t_2^{\frac{m}{l}}}}{2}\right)^l+ \left(\frac{t_1 - \sqrt{t_1^2-4\left(\frac{2m}{l}\right)^{\frac{m}{l}}t_2^{\frac{m}{l}}}}{2}\right)^l, \hspace{10mm} y_2=\frac{t_2}{l}.
\end{equation}
Formulas (\ref{2dimAlgebraicRelations}) may be thought of as inverse relations to formulas (\ref{tx2dimEqs}), where we replace flat coordinates $x_1, \, x_2$ with basic invariants given by (\ref{basicInvariants2dim}).

Note that in the above analysis we relate an algebraic Frobenius manifold (\ref{Equation2}) with a conjugacy class of a quasi-Coxeter element in a dihedral group provided that $k \geq 2.$ Two-dimensional Frobenius manifolds with $0<k<2$ have positive degrees but a relation to quasi-Coxeter elements seems unclear in this range. Note that the charge $d<0$ in this case, whereas $d \geq 0$ when $k \geq 2.$ The general conjecture on the relation of algebraic Frobenius manifolds with quasi-Coxeter elements in \cite{DinarOriginal}  assumes that the degrees are positive. A possible way to exclude the examples (\ref{Equation2}) with $0<k<2$ is to impose an additional assumption to the conjecture that the charge $d \geq 0.$ For $k<0$ the Frobenius manifolds with prepotential (\ref{Equation2}) have $d_2<0.$

\smallskip

In this work we establish relations between the two sets of flat coordinates for all but one of the known non-rational algebraic Frobenius manifolds of dimensions 3 and 4. Thus, we deal with the two Dubrovin--Mazzocco examples $(H_3)'$ and $(H_3)'',$ Pavlyk's example for $D_4(a_1)$ and Dinar's example for $F_4(a_2).$ We also cover most of the examples from \cite{Sekiguchi} related to $H_4.$ The only known algebraic Frobenius manifold in dimensions 3 or 4 which we do not deal with is the example $H_4(9)$ from \cite{Sekiguchi}. In this case the prepotential is a rational function, rather than a polynomial function, of the flat coordinates and an additional variable $Z,$ which is algebraic in the flat coordinates. In all the cases we consider the flat coordinates of the metric $\eta$ appear to be functions on a finite cover of the orbit space of the corresponding Coxeter group, while coordinates on the orbit space are basic invariants in flat coordinates of the intersection form $g$. Formulas (\ref{basicInvariants2dim}) and (\ref{2dimAlgebraicRelations}) demonstrate this in dimension 2.

We note that the inversion symmetry \cite{DubrovinNotes} of a polynomial Frobenius manifold gives a Frobenius manifold with a rational prepotential, and, more generally, the inversion of an algebraic Frobenius manifold gives a Frobenius manifold with an algebraic prepotential. For the polynomial cases and the algebraic cases listed above their inversions have a negative degree. Relations between the two sets of flat coordinates of the resulting Frobenius manifolds can be deduced directly from the relations between the original two sets of the flat coordinates by considering how the inversion changes the flat coordinates of the intersection form following \cite{Ian} and how the inversion changes the flat coordinates of the metric following \cite{DubrovinNotes}.

\smallskip

The structure of the paper is as follows. In Section \ref{sec2} we present some preliminary results on Frobenius manifolds, Coxeter invariant polynomials, the Laplace operator and we explain the method we use in order to relate the two flat coordinate systems. We assume that flat coordinates of $\eta$ are algebraic in the basic invariants of the corresponding Coxeter group which turns out to be the case as we manage to find flat coordinates in such a form. The key tool to use is the Laplace operator in the flat coordinates of the intersection form $g$. This operator can also be rewritten in the flat coordinates of the Frobenius manifold metric $\eta$, see Proposition~\ref{Prop12}. This allows us to  investigate harmonic functions in the flat coordinates.
On the other hand we find harmonic basic invariants (see Proposition \ref{LemmaGood}) which we equate to harmonic functions of suitable degree in the flat coordinates. 
We still have some coefficients to be found in these relations. 
In order to find them
we compute the intersection form $g$ in two different ways.  
After the general method is explained in subsection \ref{sec2.4} we deal with all the examples in the subsequent Sections. 

In all the examples we express explicitly the basic invariants of the flat coordinates of the intersection form $g$ in terms of the flat coordinates of the Frobenius manifold metric $\eta,$ generalising formulas (\ref{basicInvariants2dim}) and (\ref{2dimAlgebraicRelations}) from the two-dimensional case. The explicit formulas for the flat coordinates of the metric $\eta$ in terms of the basic invariants of the flat coordinates of the intersection form $g$ are given for the examples related to $H_3, D_4$, $H_4(1)$ and $H_4(2)$.  These formulas are also obtained but are too long to include for the example related to $F_4$ and for $H_4(3)$. Calculations are done by Mathematica.

In the final Section \ref{AlmostDualitySection}, we present some observations on the dual prepotentials of algebraic Frobenius manifolds. In subsection \ref{2dimSubsection} we find the dual prepotentials for Frobenius manifolds with prepotentials of the form (\ref{Equation2}) with $k=\frac{1}{l}.$ For $l \geq 2$ we express the dual prepotential using a hypergeometric function and we apply the inversion symmetry to find the dual prepotentials for $k=-\frac{1}{l}.$ In subsection \ref{7h3d4}, for $(H_3)''$ and $D_4(a_1)$ we analyse singularities of the third order derivatives of their dual prepotentials on the Coxeter mirrors.
While determining dual prepotentials for algebraic Frobenius manifolds was a motivation for the present work results in Section \ref{AlmostDualitySection} demonstrate that these prepotentials are considerably more involved comparing to the polynomial Frobenius manifolds.

\section{General approach}
\label{sec2}

\subsection{Notations}

For any two vectors $a=(a_1, \dots, a_n), \, b=(b_1, \dots, b_n) \in \C^n,$ we define
\begin{equation*}
(a, \, b)=\sum_{\alpha=1}^n a_\alpha b_\alpha \in \C.
\end{equation*}
Let $M$ be a smooth manifold with coordinate system $z_1, \dots, z_n.$ If $f \in C^\infty(M)$ is homogeneous in the $z$ coordinates, and has degree $k$, then we may write
\begin{equation*}
\mathrm{deg} \, f(z)=k.
\end{equation*}
For an $(r, s)$-tensor field $T$ on $M$ expressed in the $z$ coordinates we denote
\begin{equation*}
T_{;i}(z):=\frac{\partial T(z)}{\partial z_i}.
\end{equation*}
For a vector field $X=X^\lambda(z)\partial_{z_\lambda}$ on $M,$ the Lie derivative $\mathcal{L}_XT$ of an $(r, s)$-tensor field $T$ along $X$ is an $(r, s)$-tensor field which is defined in the $z$ coordinates by the following formula:
\begin{equation} \label{LieDef}
(\mathcal{L}_XT)(z)=X^\lambda(z)T_{;\lambda}(z)-\sum_{\alpha=1}^r X^{a_\alpha}_{;\lambda}(z)T^{a_1 \dots \lambda \dots a_r}_{b_1 \dots b_s}(z)+\sum_{\beta=1}^s X^\lambda_{;b_\beta}(z)T^{a_1 \dots a_s}_{b_1 \dots \lambda \dots b_s}(z).
\end{equation}
Note that, throughout, we are assuming summation over repeated upper and lower indices.

\subsection{Laplace operator on Frobenius manifolds}

Let $M$ be an $n$-dimensional Frobenius manifold $M$ with prepotential $F(t_1, \dots, t_n)$ \cite{DubrovinNotes}. The third order derivatives of the prepotential are used to define the symmetric $(0, 3)$-tensor $c$ as
\begin{equation} \label{3rdDerivs}
c_{ijk}(t)=\frac{\partial^3 F}{\partial t_i \partial t_j \partial t_k}.
\end{equation}
We define the metric $\eta,$ which is constant in the $t$ coordinates, as
\begin{equation*}
\eta_{ij}(t)=\frac{\partial^3 F}{\partial t_i \partial t_j \partial t_1}.
\end{equation*}
Let us denote  $\eta^{ij}=\left(\eta^{-1}\right)^{ij}$ to be the inverse of the metric. We assume that
\begin{equation*}
\eta_{ij}(t)=\eta^{ij}(t)=\delta_{i+j, \, n+1}.
\end{equation*}
The prepotential $F$ satisfies the WDVV equations
\begin{equation*}
c_{ij\lambda}(t)\eta^{\lambda\mu}(t)c_{\mu kl}(t)=c_{kj\lambda}(t)\eta^{\lambda\mu}(t)c_{\mu il}(t).
\end{equation*}
We assume that $F$ is quasihomogeneous:
\begin{equation*}
\mathcal{L}_E F(t)=(3-d)F(t),
\end{equation*}
where the Euler vector field $E$ has the form
\begin{equation} \label{EulerDiag}
E(t)=\sum_{\alpha=1}^n d_\alpha t_\alpha\partial_{t_\alpha},
\end{equation}
with $d_1=1$ and the charge $d \neq 1.$ Let us use the shorthand notations
\begin{equation} \label{shorthand}
c^i_{jk}=\eta^{i\lambda}c_{\lambda jk}, \hspace{10mm} c^{ij}_k=\eta^{j\lambda}c^i_{\lambda k},
\end{equation}
Then $c^i_{jk}$ are structure constants of a commutative Frobenius algebra defined on the tangent space $T_tM$ and $e=\partial_{t_1}$ is its unity. It follows that
\begin{equation} \label{Necessary}
(\mathcal{L}_E c)^{ij}_k=(d-1)c^{ij}_k.
\end{equation}
The intersection form $g$ is defined in the $t$ coordinates by the formula
\begin{equation} \label{gDef}
g^{ij}(t)=E^\lambda(t)c^{ij}_\lambda(t),
\end{equation}
which is known to be a flat metric on a dense open subset of $M.$ Let $x_1, \dots, x_n$ be flat coordinates for $g$ such that
\begin{equation} \label{mug}
g^{ij}(x)=\delta^{ij}.
\end{equation}
In these coordinates,
\begin{equation*}
E(x)=\frac{1-d}{2}\sum_{\alpha=1}^n x_\alpha\partial_{x_\alpha}.
\end{equation*}
Let us denote $g_{ij}=\left(g^{-1}\right)_{ij}$ to be inverse of the intersection form. We know that
\begin{equation*}
g_{ij}=c_{ij\lambda}(E^{-1})^\lambda,
\end{equation*}
where $E^{-1}$ is the multiplicative inverse of the Euler vector field $E.$ We define the tensor field
\begin{equation} \label{DualStructureConstants}
\overset{*}{c}{\vphantom{c}}^i_{jk}:=g_{j\lambda}c^{i\lambda}_k,
\end{equation}
and we see that
\begin{equation*}
\overset{*}{c}{\vphantom{c}}^i_{jk}=g_{j \lambda}c^{i \lambda}_k=c_{j \lambda \mu}(E^{-1})^\mu c^{i \lambda}_k=c_{k \lambda \mu}(E^{-1})^\mu c^{i \lambda}_j=g_{k \lambda}c^{i \lambda}_j=\overset{*}{c}{\vphantom{c}}^i_{kj}.
\end{equation*}
Thus
\begin{equation} \label{gRelations}
g^{i\lambda}g_{k\mu}c^{j\mu}_\lambda=g^{i\lambda}\overset{*}{c}{\vphantom{c}}^j_{k\lambda}=g^{i\lambda}\overset{*}{c}{\vphantom{c}}^j_{\lambda k}=g^{i \lambda}g_{\lambda \mu}c^{j \mu}_k=c^{ji}_k.
\end{equation}
Define $\Delta$ to be the Laplace operator in the $x$ coordinates and $\nabla$ to be the gradient operator in the $x$ coordinates, so for a function $f \in C^\infty(M)$ we have
\begin{equation*}
\Delta(f)=\sum_{\alpha=1}^n \frac{\partial^2 f}{\partial x_\alpha^2}, \hspace{10mm} \nabla(f)=\left(\frac{\partial f}{\partial x_1}, \frac{\partial f}{\partial x_2}, \dots, \frac{\partial f}{\partial x_n}\right).
\end{equation*}
Let $z_1, \dots, z_n$ be any coordinate system on $M.$ Then
\begin{equation} \label{gProp}
g^{ij}(z)=\left(\nabla(z_i), \, \nabla(z_j)\right)=\sum_{\alpha=1}^n\frac{\partial z_i}{\partial x_\alpha}\frac{\partial z_j}{\partial x_\alpha}.
\end{equation}
We also have
\begin{equation} \label{switch}
g^{i\lambda}(z)g_{\lambda j;k}(z)=(g^{i\lambda}g_{\lambda j})_{;k}(z)-g^{i\lambda}_{;k}(z)g_{\lambda j}(z)=\delta^i_{j;k}-g^{i\lambda}_{;k}(z)g_{\lambda j}(z)=-g^{i\lambda}_{;k}(z)g_{\lambda j}(z).
\end{equation}
Let $^g\Gamma^i_{jk}(z)$ be the Christoffel symbols for the metric $g$ in the $z$ coordinates. Then, in the coordinate system $x_1, \dots, x_n,$ the Christoffel symbols satisfy the following coordinate transformation law:
\begin{equation} \label{Christoffel}
^g\Gamma^\lambda_{\mu\nu}(z)\frac{\partial z_\mu}{\partial x_j}\frac{\partial z_\nu}{\partial x_k}\frac{\partial x_i}{\partial z_\lambda}+\frac{\partial^2 z_\lambda}{\partial x_j \partial x_k}\frac{\partial x_i}{\partial z_\lambda}={}^g\Gamma^i_{jk}(x)=0.
\end{equation}
The following proposition can be extracted from \cite{DubrovinNotes} (see formula (G.6) and Lemma 3.4). We include a complete proof below.
\begin{proposition} \label{Prop12}
For a function $f \in C^\infty(M),$ we have
\begin{equation} \label{FirstStatement}
\Delta(f)=g^{\nu\mu}(t)\frac{\partial^2 f}{\partial t_\nu \partial t_\mu}+\Delta(t_\nu)\frac{\partial f}{\partial t_\nu}.
\end{equation}
Furthermore,
\begin{equation*}
\Delta(t_i)=\left(\frac{d-1}{2}+d_i\right)c^{i\lambda}_\lambda(t),
\end{equation*}
where we sum over the index $\lambda$ and $i=1, \dots, n$ is fixed.
\end{proposition}
\begin{proof}
We have
\begin{equation*}
\Delta(f)=\sum_{\alpha=1}^n \frac{\partial}{\partial x_\alpha}\left(\frac{\partial f}{\partial t_\nu}\frac{\partial t_\nu}{\partial x_\alpha}\right)=\sum_{\alpha=1}^n \left(\frac{\partial^2 f}{\partial t_\nu \partial t_\mu}\frac{\partial t_\nu}{\partial x_\alpha}\frac{\partial t_\mu}{\partial x_\alpha}+\frac{\partial f}{\partial t_\nu}\frac{\partial^2 t_\nu}{\partial x_\alpha^2}\right),
\end{equation*}
which gives the equality (\ref{FirstStatement}) by formula (\ref{gProp}). Now we find $\Delta(t_j).$ By relation (\ref{Christoffel}) we have that
\begin{equation*}
\Delta(t_i)=\sum_{\alpha=1}^n \frac{\partial^2 t_\lambda}{\partial x_\alpha^2}\delta^i_\lambda=\sum_{\alpha=1}^n \frac{\partial^2 t_\lambda}{\partial x_\alpha^2}\frac{\partial x_\mu}{\partial t_\lambda}\frac{\partial t_i}{\partial x_\mu}=\sum_{\alpha=1}^n -^g\Gamma^\nu_{\sigma\omega}(t)\frac{\partial t_\sigma}{\partial x_\alpha}\frac{\partial t_\omega}{\partial x_\alpha}\frac{\partial x_\mu}{\partial t_\nu}\frac{\partial t_i}{\partial x_\mu}=-^g\Gamma^i_{\sigma\omega}(t)\sum_{\alpha=1}^n\frac{\partial t_\sigma}{\partial x_\alpha}\frac{\partial t_\omega}{\partial x_\alpha}.
\end{equation*}
Using equations (\ref{gProp}) and (\ref{switch}) we get that
\begin{align}
\Delta(t_i)=& \, -g^{\sigma\omega}(t)^g\Gamma^i_{\sigma\omega}(t)=-\frac{1}{2}g^{\sigma\omega}(t)g^{i\lambda}(t)\left(g_{\lambda\omega;\sigma}(t)+g_{\sigma\lambda;\omega}(t)-g_{\sigma\omega;\lambda}(t)\right) \nonumber \\
=& \, \frac{1}{2}\left(g^{\sigma\omega}(t)g^{i\lambda}_{;\sigma}(t)g_{\lambda\omega}(t)+g^{\sigma\omega}(t)g^{i\lambda}_{;\omega}(t)g_{\sigma\lambda}(t)-g^{\sigma\omega}_{;\lambda}(t)g^{i\lambda}(t)g_{\sigma\omega}(t)\right) \nonumber \\
=& \, g^{i\lambda}_{;\lambda}(t)-\frac{1}{2}g^{\sigma\omega}_{;\lambda}(t)g^{i\lambda}(t)g_{\sigma\omega}(t). \label{Star}
\end{align}
By relation (\ref{gDef}) we can rearrange equation (\ref{Star}) as
\begin{align*}
\Delta(t_i)=& \, (E^\mu(t)c^{i\lambda}_\mu(t))_{;\lambda}-\frac{1}{2}(E^\mu(t)c^{\sigma\omega}_\mu(t))_{;\lambda}g^{i\lambda}(t)g_{\sigma\omega}(t) \\
=& \, E^\mu_{;\lambda}(t)c^{i\lambda}_\mu(t)+E^\mu(t)c^{i\lambda}_{\mu;\lambda}(t)-\frac{1}{2}\left(E^\mu_{;\lambda}(t)c^{\sigma\omega}_\mu(t)+E^\mu(t)c^{\sigma\omega}_{\mu;\lambda}(t)\right)g^{i\lambda}(t)g_{\sigma\omega}(t).
\end{align*}
From relations (\ref{shorthand}) we see that $c^{ij}_{k;l}(t)=c^{ij}_{l;k}(t),$ since $\eta$ is constant in the $t$ coordinates, hence
\begin{equation*}
\Delta(t_i)=E^\mu_{;\lambda}(t)c^{i\lambda}_\mu(t)+E^\mu(t)c^{i\lambda}_{\lambda;\mu}(t)-\frac{1}{2}\left(E^\mu_{;\lambda}(t)c^{\sigma\omega}_\mu(t)+E^\mu(t)c^{\sigma\omega}_{\lambda;\mu}(t)\right)g^{i\lambda}(t)g_{\sigma\omega}(t).
\end{equation*}
By relation (\ref{LieDef}) the Lie derivative of the tensor field $c^{ij}_k$ has the form
\begin{equation*}
(\mathcal{L}_Ec)^{ij}_k(t)=E^\lambda(t)c^{ij}_{k;\lambda}(t)-E^i_{;\lambda}(t)c^{\lambda j}_k(t)-E^j_{;\lambda}(t)c^{i\lambda}_k(t)+E^\lambda_{;k}(t)c^{ij}_\lambda(t).
\end{equation*}
Therefore
\begin{align*}
\Delta(t_i)=& \, (\mathcal{L}_Ec)^{i\lambda}_\lambda(t)+E^i_{;\mu}(t)c^{\mu\lambda}_\lambda(t)+E^\lambda_{;\mu}(t)c^{i\mu}_\lambda(t) \\
&-\frac{1}{2}\left((\mathcal{L}_Ec)^{\sigma\omega}_\lambda(t)+E^\sigma_{;\mu}(t)c^{\mu\omega}_\lambda(t)+E^\omega_{;\mu}(t)c^{\sigma\mu}_\lambda(t)\right)g^{i\lambda}(t)g_{\sigma\omega}(t).
\end{align*}
By relations (\ref{Necessary}) and (\ref{gRelations}) we have that
\begin{align*}
\Delta(t_i)=& \, (d-1)c^{i\lambda}_\lambda(t)+E^i_{;\mu}(t)c^{\mu\lambda}_\lambda(t)+E^\lambda_{;\mu}(t)c^{i\mu}_\lambda(t) \\
&-\frac{1}{2}\left((d-1)c^{\sigma\omega}_\lambda(t)+E^\sigma_{;\mu}(t)c^{\mu\omega}_\lambda(t)+E^\omega_{;\mu}(t)c^{\sigma\mu}_\lambda(t)\right)g^{i\lambda}(t)g_{\sigma\omega}(t) \\
=& \, (d-1)c^{i\lambda}_\lambda(t)+E^i_{;\mu}(t)c^{\mu\lambda}_\lambda(t)+E^\lambda_{;\mu}(t)c^{i\mu}_\lambda(t) \\
&-\frac{d-1}{2}c^{i\lambda}_\lambda(t)-\frac{1}{2}E^\sigma_{;\mu}(t)c^{\mu i}_\sigma(t)-\frac{1}{2}E^\omega_{;\mu}(t)c^{i\mu}_\omega(t) \\
=& \, \frac{d-1}{2}c^{i\lambda}_\lambda(t)+E^i_{;\mu}(t)c^{\mu\lambda}_\lambda(t).
\end{align*}
The statement follows by formula (\ref{EulerDiag}).
\end{proof}

\subsection{Coxeter-invariant coordinates}
Let $W$ be a finite irreducible Coxeter group of rank $n$ acting on its complexified reflection representation $V \cong \C^n$ by orthogonal transformations with respect to $( \, \cdot \, , \cdot \, ).$ Consider an orthonormal basis $e_1, \dots, e_n,$ and the coordinates $x_1, \dots, x_n$ defined as
\begin{equation*}
x_i(v)=(v, \, e_i),
\end{equation*}
for all $v \in V$ and all $i=1, \dots, n.$ Let $y_1, \dots, y_n$ be a set of homogeneous generators of the algebra $\C[x_1, \dots, x_n]^W$. It is well-known that such a set always exists \cite{Humphreys} and we have an algebra isomorphism
\begin{equation*}
\C[y_1, \dots, y_n] \cong \C[x_1, \dots, x_n]^W.
\end{equation*}
These generators are called basic invariants. The degrees $d^W_i$ of basic invariants $y_i$ do not depend on the choice of basic invariants \cite{Humphreys}. We assume that $d^W_1 \geq d^W_2 \geq \dots \geq d^W_{n-1} > d^W_n=2.$

\begin{lemma} \label{LaplaceLemma}
Let $p, q \in \C[x_1, \dots, x_n]^W.$ Then $\Delta(p), \Delta(q) \in \C[x_1, \dots, x_n]^W$ and $\left(\nabla(p), \, \nabla(q)\right) \in \C[x_1, \dots, x_n]^W.$
\end{lemma}
\begin{proof}
The first claim follows from the invariance of $\Delta$ under orthogonal transformations. We have
\begin{equation*}
\left(\nabla(p), \, \nabla(q)\right)=\frac{1}{2}\Bigl(\Delta(pq)-\Delta(p)q-p\Delta(q)\Bigr),
\end{equation*}
which implies the second statement.
\end{proof}

\noindent We will use the following statement.

\begin{proposition} \label{LemmaGood}
There exists a set of basic invariants $Y_1, \dots, Y_n \in \C[x_1, \dots, x_n]^W$ such that
\begin{equation*}
\Delta(Y_n)=1, \hspace{5mm} \Delta(Y_j)=0,
\end{equation*}
for $j=1, \dots, n-1.$
\end{proposition}
\begin{proof}
Let $y_1, \dots, y_n \in \C[x_1, \dots, x_n]^W$ be a set of basic invariants. Define $Y_n$ as
\begin{equation*}
Y_n:=\frac{1}{2n}\sum_{i=1}^n x_i^2,
\end{equation*}
so that $\Delta(Y_n)=1.$ Now, it is well-known that
\begin{equation} \label{decompH}
\C[x_1, \dots, x_n]=Y_n \C[x_1, \dots, x_n] \oplus H,
\end{equation}
where $H=\mathrm{Ker}(\Delta)$ is the vector space of harmonic polynomials. Consider the vector spaces $V_k$ of homogeneous $W$-invariant polynomials of degree $k$ and the linear maps
\begin{equation*}
\Delta: \mathrm{Span}\{y_j, Y_nV_{\mathrm{deg} \, y_j(x)-2}\} \rightarrow V_{\mathrm{deg} \, y_j(x)-2},
\end{equation*}
for $j=1, \dots, n-1.$ Since the dimension of the domain is larger than the dimension of the range, there must be a nontrivial kernel that is not contained in $Y_nV_{\mathrm{deg} \, y_j(x)-2}$ by the direct sum decomposition (\ref{decompH}). Let $Y_j$ be a nonzero element of this kernel. The polynomials $Y_j, \, 1 \leq j \leq n,$ are homogeneous and each basic invariant $y_i$ can be expressed as a polynomial in $Y_j$, thus $Y_j$ generate $\C[x_1, \dots, x_n]^W$ and we have that $\Delta(Y_j)=0$ for all $j \leq n-1.$
\end{proof}

Suppose we can write the prepotential $F$ of a Frobenius manifold $M$ as a polynomial $F(t, Z)$ in the $t$ coordinates and $Z,$ where $Z$ satisfies an equation of the form
\begin{equation} \label{Zequation}
P(t; Z)=\sum_{k=0}^N a_k(t)Z^k=0,
\end{equation}
where $a_k \in \C[t_2, \dots, t_n],$ such Frobenius manifolds are called \textit{algebraic}. We say that an algebraic Frobenius manifold $M$ is \textit{associated to the Coxeter group $W$} if there exist basic invariants $y_1, \dots, y_n$ in the flat coordinates $x_1, \dots, x_n$ of the intersection form $g$ which are simultaneously polynomial in $t_1, \dots, t_n$ and $Z.$ All of the examples of Frobenius manifolds which we consider below are associated to Coxeter groups.

We will sometimes need to consider the $t$ coordinates and $Z$ as independent variables (see e.g. Proposition \ref{UnityProp} below). In such cases, for a rational function $f$ of $n+1$ variables, we will write $f^F(t, Z)$ instead of $f(t, Z).$

\subsection{Relating flat coordinates with basic invariants}
\label{sec2.4}

In this subsection we will explain how to relate flat coordinates $t_i$ with flat coordinates $x_j$ of the intersection form $g,$ or rather with basic invariants $y_j$ of a Coxeter group. It is known \cite{DubrovinNotes} that for $d \neq 1$ we have
\begin{equation*} 
t_n=\frac{1-d}{4}\sum_{i=1}^n x_i^2.
\end{equation*}
Since $E$ is diagonal, we have $\mathrm{deg} \, t_i(x)=\frac{2d_i}{d_n}$ for all $i.$ The general method for finding basic invariants as polynomials $y_i(t, Z)$ below will go through the following steps:

\vspace{2mm}

\textbf{1)}  Set $y_n=\sum_{i=1}^nx_i^2=\frac{4}{1-d}t_n.$ Choose $y_1, \dots, y_{n-1}$ so that $y_1, \dots, y_n$ form a set of basic invariants for a finite irreducible Coxeter group $W.$

\vspace{2mm}

\textbf{2)} Let $Y_1, \dots, Y_n$ be a set of basic invariants such that $\Delta(Y_n)=1$ and $\Delta(Y_j)=0$ for $j=1, \dots, n-1,$ which exist by Proposition \ref{LemmaGood}. Each $Y_i$ can be expressed as a polynomial in $y_1, \dots, y_n.$ In particular, $Y_n=\frac{1}{2n}y_n.$

\vspace{2mm}

\textbf{3)} Let $V_j$ be the vector space of polynomials in $t_1, \dots, t_n$ and $Z$ which are homogenous in the $x$ coordinates of degree $d_j^W.$ Find the harmonic elements of $V_j$ using Proposition \ref{Prop12}, for $j=1, \dots, n-1.$

\vspace{2mm}

\textbf{4)} Equate $Y_j=Y_j(y)$ with a general harmonic element of $V_j.$ Rearrange these equations to find each $y_j$ as a polynomial in $t_1, \dots, t_n$ and $Z$ up to some coefficients to be found. This can be done successively for $j=n-1, n-2, \dots, 1.$

\vspace{2mm}

\textbf{5)} Find the intersection form $g^{ij}$ in the $y$ coordinates by the formula $g^{ij}(y)=\left(\nabla(y_i), \, \nabla(y_j)\right),$ and express the entries as polynomials in the $y$ coordinates, which can be done by Lemma \ref{LaplaceLemma}. Substitute the expressions for $y_i(t, Z)$ into these entries, so we have $g^{ij}(y(t)).$

\vspace{2mm}

\textbf{6)} Calculate the components $g^{ij}(y(t))$ of the intersection form $g$ in the $y$ coordinates by performing a change of coordinates $y=y(t)$ on the intersection form $g^{\lambda\mu}(t)$ given by formula (\ref{gDef}):
\begin{equation*}
g^{ij}(y(t))=g^{\lambda\mu}(t)\frac{\partial y_i}{\partial t_\lambda}\frac{\partial y_j}{\partial t_\mu}.
\end{equation*}
Here, the derivatives $\frac{\partial y_i}{\partial t_\lambda}$ are found via their expressions in the $t$ coordinates and $Z$ which still contain some coefficients to be found.

\vspace{2mm}

\textbf{7)} We equate the two expressions for $g^{ij}(y(t))$ from steps 5) and 6), and find the values for the remaining coefficients, which is possible in all the examples we consider. Thus we get basic invariants $y_j$ expressed as polynomials in $t_i$ and $Z.$ Note that the polynomials we find may not be unique if the Coxeter graph of $W$ has non-trivial symmetries.

\vspace{2mm}

One may alternatively try to skip steps 2) and 4), but this increases the difficulty of the calculations needed to equate the two expressions for $g^{ij}(y(t)).$

\begin{proposition} \label{UnityProp}
Let $e=e^i(y)\partial_{y_i}$ be the unity vector field of an algebraic Frobenius manifold associated to W with prepotential $F(t, Z).$ Then $e^i(y) \in \C[t, Z]$ for each $i=1, \dots, n.$
\end{proposition}
\begin{proof}
We know that $e=\partial_{t_1}.$ Hence
\begin{equation*}
e^i(y)=e^\alpha(t)\frac{\partial y_i}{\partial t_\alpha}=\frac{\partial y_i}{\partial t_1}=\frac{\partial y_i^F}{\partial t_1}+\frac{\partial y_i^F}{\partial Z}\frac{\partial Z}{\partial t_1} \in \C[t, Z]
\end{equation*}
since $\frac{\partial Z}{\partial t_1}=0$ by relation (\ref{Zequation}).
\end{proof}

\begin{proposition} \label{DiscrimProp}
Let $g$ be the intersection form of an algebraic Frobenius manifold associated to $W$ with root system $R_W.$ Then
\begin{equation*}
\mathrm{det}(g^{ij}(t))=\frac{c\prod\limits_{\alpha \in R_W}(\alpha, \, x)}{(\mathrm{det} \, J)^2},
\end{equation*}
where $J=\left(\frac{\partial y_i}{\partial t_j}\right)_{i, j=1}^n$ is the Jacobi matrix and $c \in \C.$
\end{proposition}
\begin{proof}
It follows from \cite{Humphreys} that 
\begin{equation*}
\mathrm{det}(g^{\lambda\mu}(y))=c\prod_{\alpha \in R_W}(\alpha, \, x)
\end{equation*}
for some $c \in \C.$ From relation (\ref{gProp}), we see that
\begin{equation*}
\mathrm{det}(g^{ij}(t))=\mathrm{det}\left(g^{\lambda\mu}(y)\frac{\partial t_i}{\partial y_\lambda}\frac{\partial t_j}{\partial y_\mu}\right)=\mathrm{det}(g^{\lambda\mu}(y))\mathrm{det}\left(J^{-1}\right)^2=\frac{c\prod\limits_{\alpha \in R_W}(\alpha, \, x)}{(\mathrm{det} \, J)^2}.
\end{equation*}
\end{proof}

\section{Algebraic Frobenius manifolds related to $H_3$}
There are two non-polynomial algebraic Frobenius manifolds which we can be associated to $H_3,$ both found by Dubrovin and Mazzocco \cite{DubrovinH3}. Prepotentials of these three dimensional Frobenius manifolds were given explicitly by Kato, Mano and Sekiguchi \cite{H3forReal} (see also Remark 6.1 in \cite{H3forReal}). Let $R_{H_3}$ be the following root system for $H_3$:
\begin{equation*}
R_{H_3}=\left\{\pm e_i \mid 1 \leq i \leq 3 \right\} \cup \left\{ \frac{1}{2}\left( \pm e_{\sigma(1)}\pm\varphi e_{\sigma(2)}\pm\overline{\varphi} e_{\sigma(3)}\right) \Big\vert \, \sigma \in \mathfrak{A}_3 \right\},
\end{equation*}
where
\begin{equation*}
\varphi=\frac{1+\sqrt{5}}{2}, \hspace{10mm} \overline{\varphi}=\frac{1-\sqrt{5}}{2},
\end{equation*}
and $\mathfrak{A}_3$ is the alternating group on 3 elements. Let us introduce the following basic invariants for $H_3$ (cf. \cite{Saito}):
\begin{align}
y_1=& \, 95\epsilon_2\epsilon_3-32\epsilon_1^2\epsilon_3-5\epsilon_1\epsilon_2^2+2\epsilon_1^3\epsilon_2+3\sqrt{5}\delta\epsilon_2, \label{H3y1x} \\
y_2=& \, \sqrt{5}\delta+\epsilon_1\epsilon_2-11\epsilon_3, \label{H3y2x} \\
y_3=& \, \epsilon_1, \label{H3y3x}
\end{align}
where
\begin{align}
\epsilon_1=& \, x_1^2+x_2^2+x_3^2, \label{H3epsilon1} \\
\epsilon_2=& \, x_1^2x_2^2+x_1^2x_3^2+x_2^2x_3^2, \label{H3epsilon2} \\
\epsilon_3=& \, x_1^2x_2^2x_3^2, \label{H3epsilon3} \\
\delta=& \, (x_1^2-x_2^2)(x_1^2-x_3^2)(x_2^2-x_3^2). \label{H3delta}
\end{align}
The basic invariants $y_1, y_2, y_3$ have degrees $10, 6, 2,$ respectively.

\begin{lemma} \label{H3yIntersectionForm}
(cf. \cite{Saito}) The intersection form $g^{ij}(y)$ takes the form
\begingroup
\renewcommand*{\arraystretch}{1.5}
\begin{equation*}
g^{ij}(y)=\left(\begin{matrix}
30y_2^3+36y_2^2y_3^3+8y_1y_3^4 & 28y_2^2y_3+8y_2y_3^4 & 20y_1 \\
28y_2^2y_3+8y_2y_3^4 & 8y_1+8y_2y_3^2 & 12y_2 \\
20y_1 & 12y_2 & 4y_3 \\
\end{matrix}\right).
\end{equation*}
\endgroup
\end{lemma}

\noindent Consider another set of basic invariants for $H_3$ given by
\begin{align}
Y_1=& \, y_1-\frac{9}{17}y_2y_3^2-\frac{10}{187}y_3^5, \label{H3LaplaceY1} \\
Y_2=& \, y_2-\frac{2}{21}y_3^3, \label{H3LaplaceY2} \\
Y_3=& \, \frac{1}{6}y_3. \label{H3LaplaceY3}
\end{align}
The following statement can be checked directly.
\begin{lemma}
We have $\Delta(Y_3)=1$ and $\Delta(Y_1)=\Delta(Y_2)=0.$
\end{lemma}

\subsection{$(H_3)'$ example}
The prepotential for $(H_3)'$ is
\begin{equation*}
F(t)=\frac{1}{2}\left(t_1t_2^2+t_1^2t_3\right)-\frac{1}{18}t_3^4Z-\frac{7}{72}t_3^3Z^4-\frac{17}{105}t_3^2Z^7-\frac{2}{9}t_3Z^{10}-\frac{64}{585}Z^{13},
\end{equation*}
where
\begin{equation} \label{ZEqH31}
P(t_2, t_3, Z):=Z^4+t_3Z+t_2=0.
\end{equation}
The Euler vector field is
\begin{equation*}
E(t)=t_1\partial_{t_1}+\frac{4}{5}t_2\partial_{t_2}+\frac{3}{5}t_3\partial_{t_3},
\end{equation*}
the unity vector field is $e(t)=\partial_{t_1},$ and the charge is $d=\frac{2}{5}.$ The intersection form (\ref{gDef}) is then given by
\begingroup
\renewcommand*{\arraystretch}{1.5}
\begin{equation} \label{H31intersectionForm}
g^{ij}(t)=\left(\begin{matrix}
\frac{1}{60}(16t_2Z^3+19t_2t_3-9t_3^2Z) & \frac{1}{5}(2t_2Z^2+t_3Z^3+t_3^2) & t_1 \\
\frac{1}{5}(2t_2Z^2+t_3^2+t_3 Z^3) & t_1+\frac{Z}{10}(8t_2+3t_3Z) & \frac{4}{5}t_2 \\
t_1 & \frac{4}{5}t_2 & \frac{3}{5}t_3 \\
\end{matrix}\right).
\end{equation}
\endgroup
We have that $\mathrm{deg} \, t_1(x)=\frac{10}{3}, \, \mathrm{deg} \, t_2(x)=\frac{8}{3}, \, \mathrm{deg} \, t_3(x)=2$ and $\mathrm{deg} \, Z(x)=\frac{2}{3}.$ 
\begin{proposition} \label{PropH31}
Let $V_1=\{p \in \C[t_1, t_2, t_3, Z] \mid \, \mathrm{deg} \, p(x)=10 \}$ and let $V_2=\{p \in \C[t_1, t_2, t_3, Z] \mid \, \mathrm{deg} \, p(x)=6 \}.$ The harmonic elements of $V_1$ are proportional to
\begin{align*}
&2244000t_1^3-628320t_1t_2^2Z^2-1168530t_1t_2t_3^2-583440t_1t_2t_3Z^3+151470t_1t_3^3Z \nonumber \\
+&768944t_2^3t_3+406912t_2^3Z^3-311872t_2^2t_3^2Z+43087t_2t_3^3Z^2+32000t_3^5+37103t_3^4Z^3,
\end{align*}
and the harmonic elements of $V_2$ are proportional to
\begin{equation*}
1260t_1t_2-224t_2^2Z-154t_2t_3Z^2-80t_3^3-35t_3^2Z^3.
\end{equation*}
\end{proposition}
\begin{proof}
Using Proposition \ref{Prop12} we can directly calculate
\begin{align}
\Delta(t_1)=& \, \frac{7}{20}Z^2, \label{Delta1H31} \\
\Delta(t_2)=& \, -\frac{1}{2}Z, \label{Delta2H31} \\
\Delta(t_3)=& \, \frac{9}{10}. \label{Delta3H31}
\end{align}
A general element of $V_1$ is of the form
\begin{align}
a_1t_1^3+&a_2t_1^2t_2Z+a_3t_1^2t_3Z^2+a_4t_1t_2^2Z^2+a_5t_1t_2t_3^2+a_6t_1t_2t_3Z^3+a_7t_1t_3^3Z \nonumber \\
+&a_8t_2^3t_3+a_9t_2^3Z^3+a_{10}t_2^2t_3^2Z+a_{11}t_2t_3^3Z^2+a_{12}t_3^5+a_{13}t_3^4Z^3, \label{DeadTrees}
\end{align}
where $a_i \in \C.$ By calculating the Laplacian of this general element (\ref{DeadTrees}) using Proposition \ref{Prop12} and formulas (\ref{Delta1H31})--(\ref{Delta3H31}) we find that the only harmonic elements of $V_1$ are as claimed. A general element of $V_2$ has the form
\begin{equation} \label{Blocks}
b_1t_1t_2+b_2t_1t_3Z+b_3t_2^2Z+b_4t_2t_3Z^2+b_5t_3^3+b_6t_3^2Z^3,
\end{equation}
where $b_i \in \C.$ By calculating the Laplacian of this general element (\ref{Blocks}) using Propostion \ref{Prop12} and formulas (\ref{Delta1H31})--(\ref{Delta3H31}) we find that the only harmonic elements of $V_2$ are as claimed.
\end{proof}

\begin{theorem} \label{theoremYH31}
We have the following relations
\begin{align}
y_1=& \, \frac{128000}{19683}\left(12000t_1^3-3360t_1t_2^2Z^2-3390t_1t_2t_3^2-3120t_1t_2t_3Z^3+810t_1t_3^3Z \right. \nonumber \\
&\left. +4112t_2^3t_3+2176t_2^3Z^3-2176t_2^2t_3^2Z-119t_2t_3^3Z+200t_3^5+119t_3^4Z^3 \right), \label{y1EqH31} \\
y_2=& \, \frac{3200}{729}\left(180t_1t_2-32t_2^2Z-22t_2t_3Z^2-5t_3^3-5t_3^2Z^3\right), \label{y2EqH31} \\
y_3=& \, \frac{20}{3}t_3. \label{y3EqH31}
\end{align}
\end{theorem}
\begin{proof}
Note that $Y_3=\frac{1}{6}y_3=\frac{10}{9}t_3.$ We now equate $Y_1$ and $Y_2$ given by relations (\ref{H3LaplaceY1})--(\ref{H3LaplaceY3}) with general harmonic elements of $V_1$ and $V_2,$ respectively, given by Proposition \ref{PropH31}. We then rearrange these equations to find $y_i$ in terms of $t_j$ and $Z.$ We find
\begin{align}
y_1=& \, \frac{25600000}{18711}t_3^5+\frac{a}{583440}(2244000t_1^3-628320t_1t_2^2Z^2-1168530t_1t_2t_3^2 \nonumber \\
&-583440t_1t_2t_3Z^3+151470t_1t_3^3Z+768944t_2^3t_3+406912t_2^3Z^3 \nonumber \\
&-311872t_2^2t_3^2Z+43087t_2t_3^3Z^2+32000t_3^5+37103t_3^4Z^3) \nonumber \\
&-\frac{80b}{119}(1260t_1t_2-224t_2^2Z-154t_2t_3Z^2-80t_3^3-35t_3^2Z^3), \label{y1FakeH31} \\
y_2=& \, \frac{16000}{567}t_3^3-\frac{b}{35}(1260t_1t_2-224t_2^2Z-154t_2t_3Z^2-80t_3^3-35t_3^2Z^3), \label{y2FakeH31} \\
y_3=& \, \frac{20}{3}t_3, \label{y3FakeH31}
\end{align}
where $a, b \in \C.$ In order to find $a$ and $b$ we perform steps 5--7 from Section \ref{sec2.4}. That is, we transform the intersection form (\ref{H31intersectionForm}) into $y$ coordinates by applying formulas (\ref{y1FakeH31})--(\ref{y3FakeH31}) and compare it with the expression given by Lemma \ref{H3yIntersectionForm}. We find that $a=\frac{133120000}{6561}$ and $b=-\frac{16000}{729},$ which implies the statement.
\end{proof}

\begin{proposition} \label{DerivsProp}
The derivatives $\frac{\partial y_i}{\partial t_j} \in \C[t_1, t_2, t_3, Z].$
\end{proposition}
\begin{proof}
We have $P(t, Z)=0$ by relation (\ref{ZEqH31}). Hence
\begin{equation*}
0=\frac{\partial P}{\partial t_j}=\frac{\partial P^F}{\partial t_j}+\frac{\partial P^F}{\partial Z}\frac{\partial Z}{\partial t_j}.
\end{equation*}
Therefore
\begin{equation*}
\frac{\partial Z}{\partial t_j}=-\frac{\frac{\partial P^F}{\partial t_j}}{\frac{\partial P^F}{\partial Z}}.
\end{equation*}
We thus have that
\begin{equation} \label{yH31derivs}
\frac{\partial y_i}{\partial t_j}=\frac{\partial y_i^F}{\partial t_j}-\frac{\partial y_i^F}{\partial Z}\frac{\frac{\partial P^F}{\partial t_j}}{\frac{\partial P^F}{\partial Z}}.
\end{equation}
The first term is polynomial in $t_1, t_2, t_3$ and $Z.$ The polynomial $P^F$ is irreducible over $\C[t_1, t_2, t_3]$ and thus $\frac{\partial P^F}{\partial Z}$ is invertible in the field $\C(t_1, t_2, t_3)[Z]/(P^F),$ where $\C(t_1, t_2, t_3)$ is the field of rational functions in $t_1, t_2$ and $t_3.$ Hence the second term in equality (\ref{yH31derivs}) can be represented as an element of the ring $\C(t_1, t_2, t_3)[Z],$ when we reduce it modulo $P^F$ as a polynomial in $Z.$ It can be checked that it is a polynomial in $t_1, t_2$ and $t_3.$
\end{proof}

\begin{proposition} \label{H31discrim}
We have that
\begin{equation*}
\mathrm{det}(g^{ij}(t))=\frac{c\prod\limits_{\alpha \in R_{H_3}}(\alpha, \, x)}{Q(t, Z)^2},
\end{equation*}
where $c=-3^{26}\!\cdot\!5$ and
\begin{align*}
Q(t, Z)=2^3\!\cdot\!5^7&\left(24000t_1^3-10208t_2^3t_3+540t_1t_2t_3^2+(2484t_2^2t_3^2+540t_1t_3^3-9600t_1^2t_2)Z \right. \\
+&\left.(3360t_1t_2^2-3600t_1^2t_3-189t_2t_3^3)Z^2+(720t_1t_2t_3-4544t_2^3-81t_3^4)Z^3\right).
\end{align*}
\end{proposition}
\noindent By Proposition \ref{DiscrimProp}, we need only find $\mathrm{det}\left(\frac{\partial y_i}{\partial t_j}\right).$ It can be calculated by Theorem \ref{theoremYH31}, which leads to Proposition \ref{H31discrim}.

In the next statement we express flat coordinates $t_i$ via basic invariants $y_j$ and $Z,$ which is an inversion of the formulas from Theorem \ref{theoremYH31}.
\begin{theorem}
We have the following relations:
\begin{align}
t_1=& \, -\frac{1}{28800Z(20Z^3+3y_3)}\left(102400Z^9+20160y_3Z^6+1080y_3^2Z^3+729y_2+54y_3^3\right), \label{t1EqH31} \\
t_2=& \, -\frac{Z}{20}\left(20Z^3+3y_3\right), \label{t2EqH31} \\
t_3=& \, \frac{3}{20}y_3, \label{t3EqH31}
\end{align}
where $Z$ satisfies the equation
\begin{align}
&2^{29} 5^{11} Z^{27}+2^{27} 3^3 5^{10} y_3Z^{24}+2^{22} 3^4 5^8 151 \, y_3^2Z^{21}+2^{18} 3^4 5^7\left(2^2 7^2 19 \, y_3^3-3^3 y_2\right)Z^{18} \nonumber \\
&+2^{13} 3^6 5^5\left(60089 \, y_3^4-2^2 3^3 11 \, y_2y_3\right)Z^{15}+2^{10} 3^7 5^3\left(5^2 2\!\cdot\!11\!\cdot\!19\!\cdot\!41 \, y_3^5-3^3 263 \, y_2y_3^2 \right. \nonumber \\
&+\left. 2^2 3^7 y_1\right)Z^{12}+2^9 3^7 5^2\left(2^3 3^6 y_2^2+3^9 2 \, y_1y_3
+3^3 19\!\cdot\!41 \, y_2y_3^3+5^2 2\!\cdot\!4987 \, y_3^6\right)Z^9 \nonumber \\
&+2^6 3^9 5\left(3^6 7 \, y_2^2y_3+2^2 3^8 y_1y_3^2
+2^5 3^3 23 \, y_2y_3^4+2^2 5^3 53 \, y_3^7\right)Z^6 \nonumber \\
&+2^3 3^{10}\left(3^6 5 \, y_2^2y_3^2+2^3 3^7 y_1y_3^3+2^2 3^3 131 \, y_2y_3^5-2^2 5^2 7 \, y_3^8\right)Z^3 \nonumber \\
&+3^9\left(3^9y_2^3+3^7 2 \, y_2^2y_3^3+2^2 3^4 y_2y_3^6+2^3 y_3^9\right)=0,\label{DeltaPolyH31}
\end{align}
and $y_i$ are given by relations (\ref{H3y1x})--(\ref{H3delta}).
\end{theorem}
\begin{proof}
Formula (\ref{t3EqH31}) follows immediately from Theorem \ref{theoremYH31}, and formula (\ref{t2EqH31}) follows from relation (\ref{ZEqH31}). Substituting the relations (\ref{t2EqH31}) and (\ref{t3EqH31}) into formula (\ref{y2EqH31}) we get the expression (\ref{t1EqH31}). Finally, substituting relations (\ref{t1EqH31})--(\ref{t3EqH31}) into formula (\ref{y1EqH31}) we get the formula (\ref{DeltaPolyH31}).
\end{proof}

\begin{proposition}
The unity vector field $e=\partial_{t_1}$ in the $y$ coordinates has the form
\begin{equation*}
e(y)=\frac{64000}{81}t_2 \partial_{y_2}+\frac{1280000}{6561}\left(1200t_1^2-112t_2^2Z^2-113t_2t_3^2-104t_2t_3Z^3+27t_3^3\right)\partial_{y_1}.
\end{equation*}
\end{proposition}
\begin{proof}
We have that
\begin{equation*}
e=\partial_{t_1}=\frac{\partial y_\alpha}{\partial t_1}\partial_{y_\alpha},
\end{equation*}
which gives the statement by applying the relations from Theorem \ref{theoremYH31}.
\end{proof}

\subsection{$(H_3)''$ example}
The prepotential for $(H_3)''$ is
\begin{equation*}
F(t)=\frac{1}{2}\left(t_2^2t_1+t_3t_1^2\right)+\frac{4063}{1701}t_3^7+\frac{19}{135}t_3^5Z^2-\frac{73}{27}t_3^3Z^4+\frac{11}{9}t_3Z^6-\frac{16}{35}Z^7,
\end{equation*}
where
\begin{equation} \label{ZEqH32}
P(t_2, t_3, Z):=Z^2+t_2-t_3^2=0.
\end{equation}
The Euler vector field is
\begin{equation*}
E(t)=t_1\partial_{t_1}+\frac{2}{3}t_2\partial_{t_2}+\frac{1}{3}t_3\partial_{t_3},
\end{equation*}
the unity vector field is $e(t)=\partial_{t_1},$ and the charge is $d=\frac{2}{3}.$ The intersection form (\ref{gDef}) is then given by
\begingroup
\scriptsize
\renewcommand*{\arraystretch}{1.5}
\begin{equation} \label{H32intersectionForm}
g^{ij}(t)=\left(\begin{matrix}
\frac{4}{243}(585t_2^2t_3+3240t_2t_3^3+4456t_3^5-324Z(t_2^2-7t_2t_3^2+6t_3^4)) & -\frac{4}{27}(33t_2^2+4t_2t_3(18Z-13t_3)-72t_3^3(Z+t_3)) & t_1 \\
-\frac{4}{27}(33t_2^2+4t_2t_3(18Z-13t_3)-72t_3^3(Z+t_3) & t_1-\frac{22}{3}t_2t_3+\frac{52}{27}t_3^3+4Z(t_2-t_3^2) & \frac{2}{3}t_2 \\
t_1 & \frac{2}{3}t_2 & \frac{1}{3}t_3 \\
\end{matrix}\right).
\end{equation}
\normalsize
\endgroup
We have that $\mathrm{deg} \, t_1(x)=6, \, \mathrm{deg} \, t_2(x)=4, \, \mathrm{deg} \, t_3(x)=2$ and $\mathrm{deg} \, Z(x)=2.$ 
\begin{proposition} \label{PropH32}
Let $V_1=\{p \in \C[t_1, t_2, t_3, Z] \mid \, \mathrm{deg} \, p(x)=10 \}$ and let $V_2=\{p \in \C[t_1, t_2, t_3, Z] \mid \, \mathrm{deg} \, p(x)=6 \}.$ The harmonic elements of $V_1$ are proportional to
\begin{align*}
25245t_1t_2&+22275t_1t_3^2-16830t_2^2t_3-20196t_2^2Z+21890t_2t_3^3 \nonumber \\
&+40392t_2t_3^2Z-104196t_3^5-20196t_3^4Z,
\end{align*}
and the harmonic elements of $V_2$ are proportional to
\begin{equation*}
189t_1+630t_2t_3+400t_3^3.
\end{equation*}
\end{proposition}
\begin{proof}
Using Proposition \ref{Prop12} we can directly calculate
\begin{align}
\Delta(t_1)=& \, -\frac{5}{27}(33t_2-26t_3^2+54t_3Z), \label{Delta1H32} \\
\Delta(t_2)=& \, \frac{1}{3}(9Z-11t_3), \label{Delta2H32} \\
\Delta(t_3)=& \, \frac{1}{2}. \label{Delta3H32}
\end{align}
A general element of $V_1$ is of the form
\begin{align}
a_1t_1t_2+a_2t_1t_3^2+a_3t_1t_3Z+a_4t_2^2t_3+a_5t_2^2Z+a_6t_2t_3^3+a_7t_2t_3^2Z+a_8t_3^5+a_9t_3^4Z, \label{DeadTrees2}
\end{align}
where $a_i \in \C.$ By calculating the Laplacian of this general element (\ref{DeadTrees2}) using Proposition \ref{Prop12} and formulas (\ref{Delta1H32})--(\ref{Delta3H32}) we find that the only harmonic elements of $V_1$ are as claimed. A general element of $V_2$ has the form
\begin{equation} \label{Block2}
b_1t_1+b_2t_2t_3+b_3t_2Z+b_4t_3^3+b_5t_2^2Z,
\end{equation}
where $b_i \in \C.$ By calculating the Laplacian of this general element (\ref{Block2}) using Propostion \ref{Prop12} and formulas (\ref{Delta1H32})--(\ref{Delta3H32}) we find that the only harmonic elements of $V_2$ are as claimed.
\end{proof}

\begin{theorem} \label{theoremYH32}
We have the following relations
\begin{align}
y_1=& \, \frac{288}{25}\left(135t_1t_2+405t_1t_3^2-90t_2^2t_3-108t_2^2Z+1070t_2t_3^3 \right. \nonumber \\
&\left. \hspace{21.7mm} +216t_2t_3^2Z+2292t_3^5-108t_3^4Z\right) \label{y1EqH32} \\
y_2=& \, \frac{8}{5}\left(27t_1+90t_2t_3+160t_3^3\right), \label{y2EqH32} \\
y_3=& \, 12t_3. \label{y3EqH32}
\end{align}
\end{theorem}
\begin{proof}
Note that $Y_3=\frac{1}{6}y_3=2t_3.$ We now equate $Y_1$ and $Y_2$ given by relations (\ref{H3LaplaceY1})--(\ref{H3LaplaceY3}) with general harmonic elements of $V_1$ and $V_2,$ respectively, given by Proposition \ref{PropH32}. We then rearrange these equations to find $y_i$ in terms of $t_j$ and $Z.$ We find
\begin{align}
y_1=& \, \frac{1990656}{77}t_3^5+\frac{a}{40392}(25245t_1t_2+22275t_1t_3^2-16830t_2^2t_3-20196t_2^2Z \nonumber \\
&+21890t_2t_3^3+40392t_2t_3^2Z-104196t_3^5-20196t_3^4Z) \nonumber \\
&+\frac{81b}{425}t_3^2(189t_1+630t_2t_3+400t_3^3), \label{y1FakeH32} \\
y_2=& \, \frac{1152}{7}t_3^3+\frac{b}{400}(189t_1+630t_2t_3+400t_3^3), \label{y2FakeH32} \\
y_3=& \, 12t_3, \label{y3FakeH32}
\end{align}
where $a, b \in \C.$ In order to find $a$ and $b$ we perform steps 5--7 from Section \ref{sec2.4}. That is, we transform the intersection form (\ref{H32intersectionForm}) into $y$ coordinates by applying formulas (\ref{y1FakeH32})--(\ref{y3FakeH32}) and compare it with the expression given by Lemma \ref{H3yIntersectionForm}. We find that $a=\frac{62208}{25}$ and $b=\frac{640}{7},$ which implies the statement.
\end{proof}

\begin{proposition}
The derivatives $\frac{\partial y_i}{\partial t_j} \in \C[t_1, t_2, t_3, Z].$
\end{proposition}
\noindent Proof is similar to the one for Proposition \ref{DerivsProp}.

\begin{proposition} \label{H32discrim}
We have that
\begin{equation*}
\mathrm{det}(g^{ij}(t))=\frac{c\prod\limits_{\alpha \in R_{H_3}}(\alpha, \, x)}{Q(t, Z)^2},
\end{equation*}
where $c=-2^{14}\!\cdot\!5^5$ and
\begin{equation*}
Q(t, Z)=3^6\left(56t_3^3+126t_2t_3-27t_1+54(t_2-t_3^2)Z\right).
\end{equation*}
\end{proposition}
\noindent By Proposition \ref{DiscrimProp}, we need only find $\mathrm{det}\left(\frac{\partial y_i}{\partial t_j}\right).$ It can be calculated Theorem \ref{theoremYH32}, which leads to Proposition \ref{H32discrim}.

In the next statement we express flat coordinates $t_i$ via basic invariants $y_j$ and $Z,$ which is an inversion of the formulas from Theorem \ref{theoremYH32}.
\begin{theorem} \label{H32inverse}
We have the following relations:
\begin{align}
t_1=& \, \frac{5}{23328}\left(108y_2-25y_3^3+1296y_3Z^2\right), \label{t1EqH32} \\
t_2=& \, \frac{1}{144}y_3^2-Z^2, \label{t2EqH32} \\
t_3=& \, \frac{1}{12}y_3, \label{t3EqH32}
\end{align}
where $Z$ satisfies the equation
\begin{equation} \label{DeltaPolyH32}
31104Z^5+12960Z^4y_3+(900y_2-360y_3^3)Z^2+(25y_1-25y_2y_3^2+2y_3^5)=0,
\end{equation}
and $y_i$ are given by relations (\ref{H3y1x})--(\ref{H3delta}).
\end{theorem}
\begin{proof}
Formula (\ref{t3EqH32}) follows immediately from Theorem \ref{theoremYH32}, and formula (\ref{t2EqH32}) follows from relation (\ref{ZEqH32}). Substituting the relations (\ref{t2EqH32}) and (\ref{t3EqH32}) into formula (\ref{y2EqH32}) we get the expression (\ref{t1EqH32}). Finally, substituting relations (\ref{t1EqH32})--(\ref{t3EqH32}) into formula (\ref{y1EqH32}) we get the formula (\ref{DeltaPolyH32}).
\end{proof}

\begin{proposition}
The unity vector field $e=\partial_{t_1}$ in the $y$ coordinates has the form
\begin{equation*}
e(y)=\frac{216}{5}\partial_{y_2}+\frac{7776}{5}\left(t_2+3t_3^2\right)\partial_{y_1}.
\end{equation*}
\end{proposition}
\begin{proof}
We have that
\begin{equation*}
e=\partial_{t_1}=\frac{\partial y_\alpha}{\partial t_1}\partial_{y_\alpha},
\end{equation*}
which gives the statement by applying the relations from Theorem \ref{theoremYH32}.
\end{proof}

\section{Algebraic Frobenius manifold related to $D_4$}
The $D_4(a_1)$ Frobenius manifold has been described by Pavlyk \cite{Pavlyk} and Dinar \cite{DinarOriginal}, with a prepotential given explicitly by Pavlyk. It is a four dimensional Frobenius manifold which can be associated to the Coxeter group $D_4,$ it is denoted with the conjugacy class $a_1$ in the Coxeter group $D_4$ \cite{Carter}. The prepotential for $D_4(a_1)$ is
\begin{equation*}
F(t)=\frac{19 \, t_4^5}{2^63^45}+\frac{7 \, t_4^3t_3^2}{2^53^3}-\frac{t_4^3t_2}{2\!\cdot\!3^3}+\frac{t_4t_3^4}{2^63}+\frac{t_4t_3^2t_2}{6}+\frac{t_4t_2^2}{6}+\frac{t_4t_1^2}{2}+t_2t_3t_1-\frac{Z^5}{2^33^45},
\end{equation*}
where
\begin{equation} \label{ZEqD4}
P(t_2, t_3, t_4, Z):=Z^2-(t_4^2+3t_3^2+24t_2)=0.
\end{equation}
The Euler vector field is
\begin{equation*}
E(t)=t_1\partial_{t_1}+t_2\partial_{t_2}+\frac{1}{2}t_3\partial_{t_3}+\frac{1}{2}t_4\partial_{t_4},
\end{equation*}
the unity vector field is $e(t)=\partial_{t_1},$ and the charge is $d=\frac{1}{2}.$ We note that slightly different prepotentials and coordinates are used in Pavlyk \cite{Pavlyk} and Dinar \cite{DinarOriginal}. The intersection form (\ref{gDef}) is then given by
\begin{align}
g^{11}(t)=& \, \frac{1}{864}\left(t_4(19t_4^2+63t_3^2-144t_2)-2Z(4t_4^2+3t_3^2+24t_2)\right), \label{D4intersectionForm1} \\
g^{12}(t)=& \, \frac{1}{96}t_3\left(t_4(7t_4-2Z)+3t_3^2+48t_2\right), \\
g^{22}(t)=& \, \frac{1}{288}\left(t_4(7t_4^2+27t_3^2+144t_2)-2Z(t_4^2+12t_3^2+24t_2)\right), \\
g^{13}(t)=& \, \frac{1}{18}\left(6t_2+3t_3^2-t_4(t_4+Z)\right), \\
g^{23}(t)=& \, t_1+\frac{1}{6}t_3(2t_4-Z), \, g^{33}(t)=\frac{1}{6}(t_4-2Z), \\
g^{14}(t)=& \, t_1, \, g^{24}(t)=t_2, \, g^{34}(t)=\frac{1}{2}t_3, \, g^{44}(t)=\frac{1}{2}t_4. \label{D4intersectionForm2}
\end{align}
Let $R_{D_4}$ be the following root system for $D_4$:
\begin{equation*}
R_{D_4}=\left\{\pm e_i \pm e_j \mid 1 \leq i<j \leq 4 \right\}.
\end{equation*}
Let us introduce the following basic invariants for $D_4$ (cf. \cite{Saito}):
\begin{align}
y_1=& \, x_1^6+x_2^6+x_3^6+x_4^6, \label{D4y1x} \\
y_2=& \, x_1x_2x_3x_4, \label{D4y2x} \\
y_3=& \, x_1^4+x_2^4+x_3^4+x_4^4, \label{D4y3x} \\
y_4=& \, x_1^2+x_2^2+x_3^2+x_4^2. \label{D4y4x}
\end{align}
The basic invariants $y_1, y_2, y_3, y_4$ have degrees $6, 4, 4, 2,$ respectively.

\begin{lemma} \label{D4yIntersectionForm}
(cf. \cite{Saito}) The intersection form $g^{ij}(y)$ takes the form
\begingroup
\scriptsize
\renewcommand*{\arraystretch}{1.5}
\begin{equation*}
g^{ij}(y)=\left(\begin{matrix}
30y_1y_3-180y_2^2y_4+30y_1y_4^2-30y_3y_4^3+6y_4^5 & 6y_2y_3 & 32y_1y_4-96y_2^2+12y_3^2-24y_3y_4^2+4y_4^4 & 12y_1 \\
6y_2y_3 & \frac{1}{6}\left(2y_1-3y_3y_4+y_4^3\right) & 4y_2y_4 & 8y_2 \\
32y_1y_4-96y_2^2+12y_3^2-24y_3y_4^2+4y_4^4 & 4y_2y_4 & 16y_1 & 8y_3 \\
12y_1 & 8y_2 & y_3 & y_4
\end{matrix}\right).
\end{equation*}
\endgroup
\end{lemma}

\noindent Consider another set of basic invariants for $D_4$ given by
\begin{align}
Y_1=& \, y_1-\frac{5}{4}y_3y_4+\frac{5}{16}y_4^3, \label{D4LaplaceY1} \\
Y_2=& \, y_2, \label{D4LaplaceY2} \\
Y_3=& \, y_3-\frac{1}{2}y_4^2, \label{D4LaplaceY3} \\
Y_4=& \, \frac{1}{8}y_4. \label{D4LaplaceY4}
\end{align}
The following statement can be checked directly.
\begin{lemma}
We have $\Delta(Y_4)=1$ and $\Delta(Y_1)=\Delta(Y_2)=\Delta(Y_3)=0.$
\end{lemma}

\noindent We have that $\mathrm{deg} \, t_1(x)=4, \, \mathrm{deg} \, t_2(x)=4, \, \mathrm{deg} \, t_3(x)=2, \, \mathrm{deg} \, t_4(x)=2$ and $\mathrm{deg} \, Z(x)=2.$
\begin{proposition} \label{PropD4}
Let $V_1=\{p \in \C[t_1, t_2, t_3, Z] \mid \, \mathrm{deg} \, p(x)=6 \}$ and let $V_2=\{p \in \C[t_1, t_2, t_3, Z] \mid \, \mathrm{deg} \, p(x)=4 \}.$ The harmonic elements of $V_1$ are proportional to
\begin{equation*}
216t_1t_3+72t_2t_4+24t_2Z-9t_3^2t_4+3t_3^2Z+t_4^3+t_4^2Z,
\end{equation*}
and the harmonic elements of $V_2$ are of the form
\begin{equation*}
a(4t_1-t_3t_4)+b(t_4^2+3t_3^2-8t_2),
\end{equation*}
where $a, \, b \in \C$ are constants.
\end{proposition}
\begin{proof}
Using Proposition \ref{Prop12} we can directly calculate
\begin{align}
\Delta(t_1)=& \, \frac{t_3}{Z}(2Z-t_4), \label{Delta1D4} \\
\Delta(t_2)=& \, \frac{1}{4}\left(2t_4-Z-\frac{3t_3^2}{Z}\right),  \label{Delta2D4} \\
\Delta(t_3)=& \, -\frac{t_3}{Z}, \label{Delta3D4} \\
\Delta(t_4)=& \, 1. \label{Delta4D4}
\end{align}
A general element of $V_1$ is of the form
\begin{align}
a_1t_4t_2+&a_2t_4t_1+a_3t_4Z+a_4t_3t_2+a_5t_3t_1+a_6t_3Z+a_7t_2^3 \nonumber \\
+&a_8t_2^2t_1+a_9t_2^2Z+a_{10}t_2t_1^2+a_{11}t_2t_1Z+a_{12}t_1^3+a_{13}t_1^2Z, \label{DeadTrees3}
\end{align}
where $a_i \in \C.$ By calculating the Laplacian of this general element (\ref{DeadTrees3}) using Proposition \ref{Prop12} and formulas (\ref{Delta1D4})--(\ref{Delta4D4}) we find that the only harmonic elements of $V_1$ are as claimed. A general element of $V_2$ has the form
\begin{equation} \label{Blocks3}
b_1t_4+b_2t_3+b_3t_2^2+b_4t_2t_1+b_5t_2Z+b_6t_1^2+b_7t_1Z,
\end{equation}
where $b_i \in \C.$ By calculating the Laplacian of this general element (\ref{Blocks3}) using Propostion \ref{Prop12} and formulas (\ref{Delta1D4})--(\ref{Delta4D4}) we find that the only harmonic elements of $V_2$ are as claimed.
\end{proof}

\begin{theorem} \label{theoremYD4}
Define
\begin{align}
y_1=& \, -\frac{16}{9}\left(216t_1t_3-288t_2t_4+24t_2Z+126t_3^2t_4+3t_3^2Z-44t_4^3+t_4^2Z\right), \label{y1EqD4} \\
y_2=& \, 4\left(4t_1-t_3t_4\right), \label{y2EqD4} \\
y_3=& \, 8\left(3t_4^2-3t_3^2+8t_2\right), \label{y3EqD4} \\
y_4=& \, 8t_4. \label{y4EqD4}
\end{align}
Under the corresponding tensorial transformation the intersection form given by formulas (\ref{D4intersectionForm1})--(\ref{D4intersectionForm2}) takes the form given in Lemma \ref{D4yIntersectionForm}.
\end{theorem}
\begin{proof}
Note that $Y_4=\frac{1}{8}y_4=t_4.$ We now equate $Y_1$ with a general harmonic element of $V_1,$ and we equate $Y_2$ and $Y_3$ with general harmonic elements of $V_2,$ where $Y_1, Y_2$ and $Y_3$ are given by formulas (\ref{D4LaplaceY1})--(\ref{D4LaplaceY3}) and the harmonic elements of $V_1$ and $V_2$ are given by Proposition \ref{PropD4}. We then rearrange these equations to find $y_i$ in terms of $t_j$ and $Z.$ We find
\begin{align}
y_1=& \, 160t_4^3+\frac{a_1}{24}(216t_1t_3+72t_2t_4+24t_2Z-9t_3^2t_4+3t_3^2Z+t_4^3+t_4^2Z) \nonumber \\
&+10t_4\left(\frac{a_3}{4}(4t_1-t_3t_4)+b_3(t_4^2+3t_3^2-8t_2)\right), \label{y1FakeD4} \\
y_2=& \, \frac{a_2}{4}(4t_1-t_3t_4)+b_2(t_4^2+3t_3^2-8t_2^2), \label{y2FakeD4} \\
y_3=& \, 32t_4^2+\frac{a_3}{4}(4t_1-t_3t_4)+b_3(t_4^2+3t_3^2-8t_2), \label{y3FakeD4} \\
y_4=& \, 8t_4, \label{y4FakeD4}
\end{align}
where $a_i, b_j \in \C.$ In order to find $a_i$ and $b_j$ we perform steps 5--7 from Section \ref{sec2.4}. That is, we transform the intersection form (\ref{D4intersectionForm1})--(\ref{D4intersectionForm2}) into $y$ coordinates by applying formulas (\ref{y1FakeD4})--(\ref{y4FakeD4}) and compare it with the expression given by Lemma \ref{D4yIntersectionForm}. A particular solution is given by
\begin{equation*}
a_1=-\frac{128}{3}, \hspace{5mm} a_2=16, \hspace{5mm} a_3=0, \hspace{5mm} b_2=0, \hspace{5mm} b_3=-8,
\end{equation*}
which implies the statement.
\end{proof}

\begin{remark}
There are in fact five other ways to choose $y_i$ in Theorem \ref{theoremYD4} as polynomials of $t_j$ and $Z.$ This non-uniqueness is due to the $S_3$ symmetry of the Coxeter graph of $D_4.$
\end{remark}

\begin{proposition}
The derivatives $\frac{\partial y_i}{\partial t_j} \in \C[t_1, t_2, t_3, t_4, Z].$
\end{proposition}
\noindent Proof is similar to the one for Proposition \ref{DerivsProp}.

\begin{proposition} \label{D4discrim}
We have that
\begin{equation*}
\mathrm{det}(g^{ij}(t))=\frac{c\prod\limits_{\alpha \in R_{D_4}}(\alpha, \, x)}{Q(t, Z)^2},
\end{equation*}
where $c=9$ and
\begin{equation*}
Q(t, Z)=2^{14}\left(12t_1+5t_3t_4+2t_3Z\right).
\end{equation*}
\end{proposition}
\noindent By Proposition \ref{DiscrimProp}, we need only find $\mathrm{det}\left(\frac{\partial y_i}{\partial t_j}\right).$ It can be calculated by Theorem \ref{theoremYD4}, which leads to Proposition \ref{D4discrim}.

In the next statement we express flat coordinates $t_i$ via basic invariants $y_j$ and $Z,$ which is an inversion of formulas from Theorem \ref{theoremYD4}.
\begin{theorem}
We have the following relations:
\begin{align}
t_1=& \, -\frac{1}{13824y_2}\left(32y_4Z^3+24y_4^2Z^2+18y_1y_4-864y_2^2-27y_3y_4^2+7y_4^4\right), \label{t1EqD4} \\
t_2=& \, \frac{1}{512}\left(16Z^2+2y_3-y_4^2\right), \label{t2EqD4} \\
t_3=& \, -\frac{1}{432y_2}\left(32Z^3+24y_4Z^2+18y_1-27y_3y_4+7y_4^3\right), \label{t3EqD4} \\
t_4=& \, \frac{1}{8}y_4, \label{t4EqD4}
\end{align}
where $Z$ satisfies the equation
\begin{align}
&2^{10}Z^6+2^9 3 \, y_4Z^5+2^6 3^2 y_4^2Z^4+2^6\left(7 \, y_4^3-3^3 y_3y_4+3^2 2 \, y_1\right)Z^3+2^4 3\left(7 \, y_4^4 \right. \nonumber \\
&-\left. 3^3 y_3y_4^2+3^2 2 \, y_1y_4-2^2 3^4 y_2^2\right)Z^2+7^2 y_4^6-3^3 2\!\cdot\!7 \, y_3y_4^4+2^2 3^2 7 \, y_1y_4^3 \nonumber \\
&+3^6 y_3^2y_4^2-2^3 3^5 y_2^2y_4^2-2^2 3^5 y_1y_3y_4+2^3 3^6 y_2^2y_3+2^2 3^4 y_1^2=0, \label{DeltaPolyD4}
\end{align}
and $y_i$ are given by relations (\ref{D4y1x})--(\ref{D4y4x}).
\end{theorem}
\begin{proof}
Formula (\ref{t4EqD4}) follows immediately from Theorem \ref{theoremYD4}. Using relations (\ref{ZEqD4}) and (\ref{y2EqD4}) we see that
\begin{align}
t_1=& \, \frac{1}{32}\left(2y_2+t_3y_4\right), \label{D4intermediate1} \\
t_2=& \, \frac{1}{24}\left(Z^2-3t_3^2-\frac{1}{64}y_4^2\right). \label{D4intermediate2}
\end{align}
Substituting the relations (\ref{t4EqD4}) and (\ref{D4intermediate2}) into formula (\ref{y3EqD4}) and rearranging, we get
\begin{equation}
t_3^2=\frac{1}{75}\left(Z^2+\frac{71}{64}y_4^2-3y_3\right). \label{D4intermediate3}
\end{equation}
We can then substitute relations (\ref{t4EqD4}), (\ref{D4intermediate1}) and (\ref{D4intermediate2}) into formula (\ref{y1EqD4}) and reduce modulo $t_3^2$ using relation (\ref{D4intermediate3}) to find a linear equation in $t_3$ which we can rearrange to find relation (\ref{t3EqD4}). Substituting relations (\ref{t3EqD4}) and (\ref{t4EqD4}) into formulas (\ref{D4intermediate1}) and (\ref{D4intermediate2}) gives us relation (\ref{t1EqD4}) and the following:
\begin{align}
t_2=-\frac{1}{1492992 y_2^2}&(1024Z^6+1152y_1Z^3+324y_1^2-62208y_2^2Z^2+1536y_4Z^5+864y_1y_4Z^2 \nonumber \\
&-1728y_3y_4Z^3-972y_1y_3y_4+576y_4^2Z^4+972y_2^2y_4^2-1296y_3y_4^2Z^2 \nonumber \\
&+729y_3^2y_4^2+448y_4^3Z^3+252y_1y_4^3+336y_4^4Z^2-378y_3y_4^4+49y_4^6). \label{D4intermediate4}
\end{align}
We then substitute relations (\ref{t1EqD4}), (\ref{t3EqD4}), (\ref{t4EqD4}) and (\ref{D4intermediate4}) into formula (\ref{y1EqD4}) and we get the formula (\ref{DeltaPolyD4}). Finally, reducing relation (\ref{D4intermediate4}) modulo the polynomial (\ref{DeltaPolyD4}) in $Z$ gives us relation (\ref{t2EqD4}).
\end{proof}

\begin{proposition} \label{PropD4unity}
The unity vector field $e=\partial_{t_1}$ in the $y$ coordinates has the form
\begin{equation*}
e(y)=16 \left(\partial_{y_2}-24t_3\partial_{y_1}\right).
\end{equation*}
\end{proposition}
\begin{proof}
We have that
\begin{equation*}
e=\partial_{t_1}=\frac{\partial y_\alpha}{\partial t_1}\partial_{y_\alpha},
\end{equation*}
which gives the statement by applying the relations from Theorem \ref{theoremYD4}.
\end{proof}

\section{Algebraic Frobenius manifold related to $F_4$}
The $F_4(a_2)$ Frobenius manifold was described by Dinar with a prepotential given explicitly \cite{DinarOriginal} (there seem to be some typos for the prepotential in \cite{DinarOriginal}, we include a corrected version below which was communicated to us by Dinar). It is a four dimensional Frobenius manifold which can be associated to the Coxeter group $F_4,$ and is denoted by the conjugacy class $a_2$ in the Coxeter group $F_4$ \cite{Carter}. The prepotential for $F_4(a_2)$ is
\begin{align*}
F(t)=& \, \frac{2^5 3^7 67\!\cdot\!521749}{5^9 7}t_4^7+\frac{2^4 3^{10} 13\!\cdot\!693097}{5^9 7}t_4^6t_3+\frac{2^2 3^8 13^2 23^2 7\!\cdot\!97}{5^9}t_4^5t_3^2+\frac{3^7 13^3 18224639}{2^6 5^8 7}t_4^4t_3^3 \\
&+\frac{3^8 13^4 7243667}{2^{11} 5^8 7}t_4^3t_3^4+\frac{3^8 13^5 8754721}{2^{14} 5^9 7}t_4^2t_3^5+\frac{3^7 13^6 19\!\cdot\!435503}{2^{18} 5^9 7}t_4t_3^6+\frac{3^8 13^7 41\!\cdot\!7129}{2^{22} 5^9 7}t_3^7 \\
&+\frac{2^4 3^4 71\!\cdot\!4259}{5^5 7\!\cdot\!13}t_4^4t_2-\frac{2^3 3^4 23\!\cdot\!47}{5^5 7}t_4^3t_3t_2-\frac{3^5 13\!\cdot\!103\!\cdot\!293}{2^3 5^5 7}t_4^2t_3^2t_2-\frac{3^4 13^2 79\!\cdot\!467}{2^6 5^5 7}t_4t_3^3t_2 \\
&-\frac{3^4 13^3 157\!\cdot\!383}{2^{12} 5^5 7}t_3^4t_2+\frac{2^4}{13^2}t_4t_2^2-\frac{7}{13}t_3t_2^2+\left(\frac{3^{14} 139}{5^5 2\!\cdot\!7}t_4^5+\frac{3^{13} 13\!\cdot\!19}{5^4 2\!\cdot\!7}t_4^4t_3+\frac{3^{15} 13^3}{2^{11} 5^4}t_4^2t_3^3 \right. \\
&+\frac{3^{12} 13^3 101}{2^9 5^4 7}t_4^3t_3^2+\frac{3^{13} 13^4 31}{2^{16} 5^4 7}t_4t_3^4+\frac{3^{12} 13^5 41}{2^{20} 5^5 7}t_3^5-\frac{3^{11}}{2^2 5\!\cdot\!7\!\cdot\!13}t_4^2t_2-\frac{3^{10}}{2^5 5\!\cdot\!7}t_4t_3t_2 \\
&-\left. \frac{3^9 13}{2^{10} 5\!\cdot\!7}t_3^2t_2\right)Z^2+\left(\frac{2^2 3^9 89\!\cdot\!11701}{5^7 7}t_4^6+\frac{3^{11} 13\!\cdot\!68473}{5^7 7}t_4^5t_3+\frac{3^{10} 13^3 1949}{2^3 5^6 7}t_4^4t_3^2 \right. \\
&+\left. \frac{3^9 13^3 131357}{2^8 5^6 7}t_4^3t_3^3+\frac{3^{12} 13^4 2729}{2^{12} 5^6 7}t_4^2t_3^4+\frac{3^{10} 13^5 15937}{2^{15} 5^7 7}t_4t_3^5+\frac{3^9 13^6 29\!\cdot\!43}{2^{20} 5^7}t_3^6+\frac{3^3 5^2}{13^2 7}t_2^2 \right. \\
&\left. -\frac{2^2 3^6 139}{5^2 7\!\cdot\!13}t_4^3t_2-\frac{3^8 23}{5^2 2\!\cdot\!7}t_4^2t_3t_2-\frac{3^7 13\!\cdot\!73}{2^5 5^2 7}t_4t_3^2t_2-\frac{3^6 13^2 41}{2^9 5^2 7}t_3^3t_2\right)Z+t_3t_2t_1+\frac{t_4t_1^2}{2},
\end{align*}
where
\begin{align*}
P(t_2, t_3, t_4, Z):=& \, Z^3-\frac{2^3 3^4 13}{5^4}\left(\frac{2^3 3}{13} \, t_4^2+t_4t_3+\frac{13}{2^5 3} t_3^2\right)Z+\frac{2^2 3\!\cdot\!13}{5^6}\left(\frac{2^6 5^4}{3^3 13^2}t_2 \right. \\
&-\left. \frac{2^7 139}{13}t_4^3-2^4 3^2 23 \, t_4^2t_3-3\!\cdot\!13\!\cdot\!73 \, t_4t_3^2-\frac{13^2 41}{2^4}t_3^3\right)=0.
\end{align*}
The Euler vector field is
\begin{equation*}
E(t)=t_1\partial_{t_1}+t_2\partial_{t_2}+\frac{1}{3}t_3\partial_{t_3}+\frac{1}{3}t_4\partial_{t_4},
\end{equation*}
the unity vector field is $e(t)=\partial_{t_1},$ and the charge is $d=\frac{2}{3}.$ The intersection form (\ref{gDef}) is then given by
\begin{align}
g^{11}(t)=& \, -\frac{3^4}{2^{13} 5^8 13}\left(259200000000 Z^2 t_2 + 861120000000 Z t_2 t_3 + 
 1833744640000 t_2 t_3^2 \right. \nonumber \\
&- 94787461372500 Z^2 t_3^3 - 
 229118338413900 Z t_3^4 - 184708373655429 t_3^5 \nonumber \\
&+ 
 1067520000000 Z t_2 t_4 + 6818938880000 t_2 t_3 t_4 - 
 593226777780000 Z^2 t_3^2 t_4 \nonumber \\
&- 1466753056797600 Z t_3^3 t_4 - 
 1556308486273320 t_3^4 t_4 - 1870069760000 t_2 t_4^2 \nonumber \\
&- 
 996804506880000 Z^2 t_3 t_4^2 - 4153975366694400 Z t_3^2 t_4^2 - 
 6641896760778240 t_3^3 t_4^2 \nonumber \\
&- 546311900160000 Z^2 t_4^3 - 
 6910723746201600 Z t_3 t_4^3 - 16691391832227840 t_3^2 t_4^3 \nonumber \\
&-\left. 
 5389922879078400 Z t_4^4 - 17222218351902720 t_3 t_4^4 -
 6432998677807104 t_4^5\right), \label{F4intersectionForm1} \\
g^{12}(t)=& \, -\frac{3^4}{2^{17} 5^8}\left(86400000000t_2Z^2+911040000000t_2t_3Z+1440474880000t_2t_3^2 \right. \nonumber \\
&-44608926082500t_3^3Z^2-143790305946300t_3^4Z-157672431777393t_3^5 \nonumber \\
&+1059840000000t_2t_4Z+4513832960000t_2t_3t_4-330765116760000t_3^2t_4Z^2 \nonumber \\
&-1081698384499200t_3^3t_4Z-1043122465279440t_3^4t_4+4196270080000t_2t_4^2 \nonumber \\
&-700228488960000t_3t_4^2Z^2-2682937817164800t_3^2t_4^2Z-3140123415886080t_3^3t_4^2 \nonumber \\
&-382258206720000t_4^3Z^2-3411310364467200t_3t_4^3Z-5919243052769280t_3^2t_4^3 \nonumber \\
&\left. -2126376537292800t_4^4Z-7197815592714240t_3t_4^4-3176813316538368t_4^5\right), \\
g^{22}(t)=& \, -\frac{3^4 13}{2^{21} 5^8}\left(28800000000t_2Z^2+511680000000t_2t_3Z+1958344960000t_2t_3^2 \right. \nonumber \\
&-19207343902500t_3^3Z^2-78756703307100t_3^4Z-114276677239881t_3^5 \nonumber \\
&+1121280000000t_2t_4Z+3545784320000t_2t_3t_4-158303428920000t_3^2t_4Z^2 \nonumber \\
&-707890736966400t_3^3t_4Z-936110404686480t_3^4t_4+2777743360000t_2t_4^2 \nonumber \\
&-383440936320000t_3t_4^2Z^2-2025454752921600t_3^2t_4^2Z-2337422808015360t_3^3t_4^2 \nonumber \\
&-244681482240000t_4^3Z^2-2221827115622400t_3t_4^3Z-3112967576309760t_3^2t_4^3 \nonumber \\
&\left. -1046938278297600t_4^4Z-2618007725998080t_3t_4^4-1417602668691456t_4^5\right), \\
g^{13}(t)=& \, \frac{1}{2^6 5^4 13^2}\left(1280000t_2-924007500t_3Z^2-3897258300t_3^2Z-2181574863t_3^3 \right. \nonumber \\
&-3411720000t_4Z^2-9067593600t_3t_4Z-4518684144t_3^2t_4-5620492800t_4^2Z \nonumber \\
&\left. +9861336576t_3t_4^2+50200031232t_4^3\right), \\
g^{23}(t)=& \, \frac{1}{2^{10} 5^4 13}\left(8320000t_1-8960000t_2-308002500t_3Z^2-2188871100t_3^2Z \right. \nonumber \\
&-3888894321t_3^3-1137240000t_4Z^2-9593251200t_3t_4Z-8055045648t_3^2t_4 \nonumber \\
&\left. -5580057600t_4^2Z-5561457408t_3t_4^2+4045676544t_4^3\right),
\end{align}
\begin{align}
g^{33}(t)=& \, \frac{2}{13^2 3}\left(150Z-91t_3+16t_4\right), \\
g^{14}(t)=& \, t_1, \hspace{15mm} g^{24}(t)=t_2, \hspace{15mm} g^{34}(t)=\frac{1}{3}t_3, \hspace{15mm} g^{44}(t)=\frac{1}{3}t_4. \label{F4intersectionForm2}
\end{align}
Let $R_{F_4}$ be the following root system for $F_4$:
\begin{equation*}
R_{F_4}=\left\{\pm e_i \mid 1 \leq i \leq 4 \right\} \cup \left\{\pm e_i \pm e_j \mid 1 \leq i<j \leq 4 \right\} \cup \left\{\frac{1}{2}\left(\pm e_1 \pm e_2 \pm e_3 \pm e_4\right)\right\}.
\end{equation*}
Let us introduce the following basic invariants for $F_4$ (cf. \cite{Saito}):
\begin{align*}
y_1=& \, 288\epsilon_2\epsilon_4-108\epsilon_1^2\epsilon_4-8\epsilon_2^3+3\epsilon_1^2\epsilon_2^2, \\
y_2=& \, 12\epsilon_4-3\epsilon_1\epsilon_3+\epsilon_2^2, \\
y_3=& \, 6\epsilon_3-\epsilon_1\epsilon_2, \\
y_4=& \, \epsilon_1,
\end{align*}
where
\begin{align*}
\epsilon_1=& \, x_1^2+x_2^2+x_3^2+x_4^2, \\
\epsilon_2=& \, x_1^2x_2^2+x_1^2x_3^2+x_1^2x_4^2+x_2^2x_3^2+x_2^2x_4^2+x_3^2x_4^2, \\
\epsilon_3=& \, x_1^2x_2^2x_3^2+x_1^2x_2^2x_4^2+x_1^2x_3^2x_4^2+x_2^2x_3^2x_4^2, \\
\epsilon_4=& \, x_1^2x_2^2x_3^2x_4^2.
\end{align*}
The basic invariants $y_1, y_2, y_3, y_4$ have degrees $12, 8, 6, 2,$ respectively.

\begin{lemma} \label{F4yIntersectionForm}
(cf. \cite{Saito}) The entries of the intersection form $g^{ij}(y)$ are
\begin{align*}
g^{11}(y)=& \, 1152y_2^2y_3-144y_1y_2y_4+1152y_2y_3^2y_4-144y_1y_3y_4^2+288y_3^3y_4^2, \\
g^{12}(y)=& \, -96y_2^2y_4-48y_2y_3y_4^2, \hspace{6mm} g^{22}(y)=-8y_2y_3-y_1y_4+3y_3^2y_4, \\
g^{13}(y)=& \, 192y_2^2+120y_2y_3y_4-12y_1y_4^2+12y_3^2y_4^2, \\
g^{23}(y)=& \, 2y_1-6y_3^2-8y_2y_4^2, \hspace{12mm} g^{33}(y)=20y_2y_4-4y_3y_4^2, \\
g^{14}(y)=& \, 24y_1, \hspace{5mm} g^{24}(y)=16y_2, \hspace{5mm} g^{34}(y)=12y_3, \hspace{5mm} g^{44}(y)=4y_4.
\end{align*}
\end{lemma}

\noindent Consider another set of basic invariants for $F_4$ given by
\begin{align*}
Y_1=& \, y_1+\frac{1}{3080}y_4\left(2520y_2y_4+1708y_3y_4^2+61y_4^5\right), \\
Y_2=& \, y_2+\frac{1}{160}y_4\left(40y_3+3y_4^3\right), \\
Y_3=& \, y_3-\frac{1}{8}y_4^3, \\
Y_4=& \, \frac{1}{8}y_4,
\end{align*}
The following statement can be checked directly.
\begin{lemma}
We have $\Delta(Y_4)=1$ and $\Delta(Y_1)=\Delta(Y_2)=\Delta(Y_3)=0.$
\end{lemma}

\noindent We have that $\mathrm{deg} \, t_1(x)=6, \, \mathrm{deg} \, t_2(x)=6, \, \mathrm{deg} \, t_3(x)=2, \, \mathrm{deg} \, t_4(x)=2$ and $\mathrm{deg} \, Z(x)=2.$

\begin{theorem} \label{theoremYF4}
Define
\begin{align*}
y_1=& \, \frac{1}{2^{15}3^95^613^2}\left(2^{17}5^813^2t_1^2-2^{22}5^813 \, t_1t_2+2^{22}5^817 \, t_2^2+2^{14}3^65^813^2t_2t_3Z^2 \right. \\
&-2^{12}3^65^613^341 \, t_2t_3^2Z+2^{13}3^35^413^57\!\cdot\!29 \, t_1t_3^3-2^{12}3^35^513^4491 \, t_2t_3^3 \\
&-2^23^95^413^641 \, t_3^4Z^2+2^23^95^211^213^8t_3^5Z+3^613^81202837 \, t_3^6+2^{18}3^75^813 \, t_1t_3^3 \\
&-2^{17}3^65^613^223 \, t_2t_3t_4Z+2^{17}3^45^413^459 \, t_1t_3^2t_4+2^{16}3^45^513^367 \, t_2t_3^2t_4 \\
&-2^83^{10}5^413^541 \, t_3^3t_4Z^2+2^63^95^313^643^2t_3^4t_4Z+2^53^713^8111347 \, t_3^5t_4 \\
&+2^{20}3^65^613\!\cdot\!31 \, t_2t_4^2Z-2^{19}3^45^413^3367 \, t_1t_3t_4^2+2^{20}3^45^513^411 \, t_2t_3t_4^2 \\
&-2^{11}3^{10}5^413^4269 \, t_3^2t_4^2Z^2+2^{11}3^95^313^5757 \, t_3^3t_4^2Z+2^811^25\!\cdot\!23633\!\cdot\!6892993 \, t_3^4t_4^2 \\
&+2^{22}3^35^413^243\!\cdot\!61 \, t_1t_4^3-2^{25}3^35^511\!\cdot\!13\!\cdot\!701 \, t_2t_4^3-2^{15}3^95^413^31039 \, t_3t_4^3Z^2 \\
&-2^{13}3^95^313^49431 \, t_3^2t_4^3Z+2^{13}3^613^55\!\cdot\!1939033 \, t_3^3t_4^3-2^{18}3^{10}5^413^2557 \, t_4^4Z^2 \\
&-2^{22}3^95^313^3587 \, t_3t_4^4Z-2^{16}3^713^45\!\cdot\!23\!\cdot\!206351 \, t_3^2t_4^4-2^{21}3^95^213^317\!\cdot\!257 \, t_4^5Z \\
&\left. -2^{22}3^713^319\!\cdot\!71\!\cdot\!2383 \, t_3t_4^5+2^53\!\cdot\!43\!\cdot\!103\!\cdot\!149\!\cdot\!2791\!\cdot\!1285517 \, t_4^6\right), \\
y_2=& \, \frac{1}{2^{12}3^55^413}\left(-2^{12}5^63 \, t_2Z+2^{13}5^413^2t_1t_3-2^{13}5^47\!\cdot\!13 \, t_2t_3+2^23^55^413^3t_3^2Z^2 \right. \\
&+2^23^45^213^441 \, t_3^3Z+3^313^511\!\cdot\!1171 \, t_3^4+2^{17}5^413 \, t_1t_4-2^{17}5^6t_2t_4 \\
&+2^73^65^413^2t_3t_4Z^2+2^63^55^213^373 \, t_3^2t_4Z+2^63^313^417\!\cdot\!79 \, t_3^3t_4+2^{10}3^75^413 \, t_4^2Z^2 \\
&+2^{10}3^65^213^223 \, t_3t_4^2Z-2^93^413^347\!\cdot\!593 \, t_3^2t_4^2+2^{13}3^45^213\!\cdot\!139 \, t_4^3Z \\
&\left. -2^{15}3^313^223\!\cdot\!1303 \, t_3t_4^3-2^{16}3^313\!\cdot\!62539 \, t_4^4\right), \\
y_3=& \, \frac{1}{2^43^45\!\cdot\!13}\left(-2^55\!\cdot\!13 \, t_1+2^93\!\cdot\!5 \, t_2+3^313^4t_3^3+2^43^513^3t_3^2t_4+2^63^413^211 \, t_3t_4^2 \right. \\
&\left. -2^93^313\!\cdot\!79 \, t_4^3\right), \\
y_4=& \, 12t_4.
\end{align*}
Under the corresponding tensorial transformation the intersection form given by formulas (\ref{F4intersectionForm1})--(\ref{F4intersectionForm2}) takes the form given in Lemma \ref{F4yIntersectionForm}.
\end{theorem}

\begin{remark}
There is in fact one other way to choose $y_i$ in Theorem \ref{theoremYF4} as polynomials of $t_j$ and $Z.$ This non-uniqueness is due to the $\Z_2$ symmetry of the Coxeter graph of $F_4.$
\end{remark}

\begin{proposition}
The derivatives $\frac{\partial y_i}{\partial t_j} \in \C[t_1, t_2, t_3, t_4, Z].$
\end{proposition}
\noindent Proof is similar to the one for Proposition \ref{DerivsProp}.

\begin{proposition} \label{F4discrim}
We have that
\begin{equation*}
\mathrm{det}(g^{ij}(t))=\frac{c\prod\limits_{\alpha \in R_{F_4}}(\alpha, \, x)}{Q(t, Z)^2},
\end{equation*}
where $c=2^{36}3^{42}5^813^4$ and
\begin{align*}
Q(t, Z)=& \, -2^{16}5^613^2t_1^2+2^{17}5^67\!\cdot\!13 \, t_1t_2+2^{21}5^6t_2^2-2^83^65^613^3Z^2t_1t_3+2^{12}3^75^613^2Z^2t_2t_3 \\
&-2^83^55^413^441 \, Zt_1t_3^2-2^{14}3^55^413^4Zt_2t_3^2+2^63^35^213^517\!\cdot\!797 \, t_1t_3^3 \\
&-2^{10}3^35^213^5919 \, t_2t_3^3-2^23^95^413^647 \, Z^2t_3^4+2^23^85^213^717\!\cdot\!41 \, Zt_3^5-3^613^85\!\cdot\!89\!\cdot\!97 \, t_3^6 \\
&-2^{12}3^75^613^2Z^2t_1t_4+2^{14}3^65^611\!\cdot\!13 \, Z^2t_2t_4-2^{13}3^55^413^373 \, Zt_1t_3t_4 \\
&+2^{13}3^55^413^2229 \, Zt_2t_3t_4+2^{10}3^45^213^5919 \, t_1t_3^2t_4-2^{12}3^45^213^372889 \, t_2t_3^2t_4 \\
&-2^33^95^413^519\!\cdot\!149 \, Z^2t_3^3t_4+2^33^85^213^641^217 \, Zt_3^4t_4+3^713^72\!\cdot\!17\!\cdot\!23\!\cdot\!37\!\cdot\!59 \, t_3^5t_4 \\
&-2^{16}3^65^413^223 \, Zt_1t_4^2+2^{17}3^65^413\!\cdot\!227 \, Zt_2t_4^2+2^{14}3^45^213^319\!\cdot\!1039 \, t_1t_3t_4^2 \\
&-2^{17}3^45^213^264871 \, t_2t_3t_4^2-2^73^{10}5^613^483 \, Z^2t_3^2t_4^2+2^93^95^213^57\!\cdot\!11\!\cdot\!41 \, Zt_3^3t_4^2 \\
&+2^53^713^65\!\cdot\!41\!\cdot\!37649 \, t_3^4t_4^2+2^{20}3^35^213^317\!\cdot\!47 \, t_1t_4^3-2^{20}3^35^213\!\cdot\!188701 \, t_2t_4^3 \\
&-2^{11}3^95^413^33571 \, Z^2t_3t_4^3+2^{13}3^85^213^411\!\cdot\!97 \, Zt_3^2t_4^3+2^{17}3^613^55\!\cdot\!52057 \, t_3^3t_4^3 \\
&-2^{17}3^95^413^211 \, Z^2t_4^4-2^{14}3^85^213^311\!\cdot\!37\!\cdot\!139 \, Zt_3t_4^4+2^{13}3^713^45\!\cdot\!17\!\cdot\!499\!\cdot\!659 \, t_3^2t_4^4 \\
&-2^{18}3^{11}5^313^2197 \, Zt_4^5+2^{17}3^713^35\!\cdot\!11\!\cdot\!247439 \, t_3t_4^5+2^{22}3^67^213^219\!\cdot\!41\!\cdot\!61 \, t_4^6.
\end{align*}
\end{proposition}
\noindent By Proposition \ref{DiscrimProp}, we need only find $\mathrm{det}\left(\frac{\partial y_i}{\partial t_j}\right).$ It can be calculated by Theorem \ref{theoremYF4}, which leads to Proposition \ref{F4discrim}.

Note that we do not include relations for $t_i, Z$ and $e$ as functions of the basic invariants $y_j$ for this example as they are too long to present here. 

\section{Algebraic Frobenius manifolds related to $H_4$}
There are 6 known non-polynomial algebraic Frobenius manifolds which can be associated to $H_4,$ they are each four-dimensional and their prepotentials have been listed by Sekiguchi \cite{Sekiguchi}. Let $R_{H_4}$ be the following root system for $H_4$:
\begin{equation*}
R_{H_4}=\left\{\pm e_i \mid 1 \leq i \leq 4 \right\} \cup \left\{ \frac{1}{2}\left(\pm e_1 \pm e_2 \pm e_3 \pm e_4\right) \right\} \cup \left\{ \frac{1}{2}\left( \pm e_{\sigma(2)}\pm\varphi e_{\sigma(3)}\pm\overline{\varphi} e_{\sigma(4)}\right) \Big\vert \, \sigma \in \mathfrak{A}_4 \right\},
\end{equation*}
where
\begin{equation*}
\varphi=\frac{1+\sqrt{5}}{2}, \hspace{10mm} \overline{\varphi}=\frac{1-\sqrt{5}}{2},
\end{equation*}
and $\mathfrak{A}_4$ is the alternating group on 4 elements. Let us introduce the following basic invariants for $H_4$ (cf. \cite{Saito}):
\begin{align}
y_1=& \, \frac{32}{3}x_1^{24}h_2^3-40x_1^{22}\left(2h_2^4+3h_2h_6\right)+x_1^{20}\left(360h_{10}+\frac{1344}{5}h_2^5+672h_2^2h_6\right) \nonumber \\
&+x_1^{18}\left(1080h_6^2-1608h_2^3h_6-\frac{1328}{3}h_2^6-2880h_{10}h_2\right) \nonumber \\
&+x_1^{16}\left(10024h_{10}h_2^2+272h_2^7+1248h_2^4h_6-5628h_2h_6^2\right) \nonumber \\
&+x_1^{14}\left(18588h_2^2h_6^2+272h_2^8-7620h_{10}h_6-16856h_{10}h_2^3\right) \nonumber \\
&+x_1^{12}\left(14216h_{10}h_2^4+23508h_{10}h_6h_2-\frac{1328}{3}h_2^9-1248h_6h_2^6-27396h_6^2h_2^3 \right. \nonumber \\
&\biggl. -5796h_6^3\biggr)+x_1^{10}\left(3240h_{10}^2-7160h_{10}h_2^5-25332h_{10}h_6h_2^2+\frac{1344}{5}h_2^{10} \right. \nonumber \\
&\biggl. +1608h_2^7h_6+19968h_2^4h_6^2+7350h_2h_6^3\biggr)+x_1^8\left(2144h_{10}h_2^6-3232h_{10}^2h_2 \right. \nonumber \\
&\left. +10908h_{10}h_2^3h_6-906h_{10}h_6^2-80h_2^{11}-672h_2^8h_6-6924h_2^5h_6^2-1956h_2^2h_6^3\right) \nonumber \\
&+x_1^6\left(1168h_{10}^2h_2^2-344h_{10}h_2^7-2172h_{10}h_2^4h_6-1908h_{10}h_2h_6^2+\frac{32}{3}h_2^{12} \right. \nonumber \\
&\biggl. +120h_2^9h_6+1332h_2^6h_6^2+288h_2^3h_6^3+2394h_6^4\biggr)+x_1^4\left(348h_{10}^2h_6-152h_{10}^2h_2^3 \right. \nonumber \\
&\left. +16h_{10}h_2^8+60h_{10}h_6h_2^5+408h_{10}h_6^2h_2^2-84h_2^7h_6^2+84h_2^4h_6^3-909h_2h_6^4\right) \nonumber \\
&+x_1^2\left(8h_{10}^2h_2^4-42h_{10}h_6^2h_2^3-87h_{10}h_6^3-6h_2^5h_6^3+135h_2^2h_6^4\right) \nonumber \\
&+\frac{4}{3}h_{10}^3-3h_{10}h_2h_6^3+\frac{9}{5}h_6^5, \label{H4y1x} \\
y_2=& \, 4x_1^{16}h_2^2-10x_1^{14}\left(2h_2^3+3h_6\right)+x_1^{12}\left(44h_2^4+138h_2h_6\right) \nonumber \\
&+x_1^{10}\left(180h_{10}-44h_2^5-402h_2^2h_6\right)+x_1^8\left(44h_2^6-464h_{10}h_2+402h_2^3h_6+294h_6^2\right) \nonumber \\
&+x_1^6\left(296h_{10}h_2^2-20h_2^7-138h_2^4h_6-306h_2h_6^2\right)+x_1^4\left(4h_2^8-76h_{10}h_2^3-114h_{10}h_6 \right. \nonumber \\
&\left. +30h_2^5h_6+168h_2^2h_6^2\right)+x_1^2\left(4h_{10}h_2^4-21h_2^3h_6^2+\frac{57}{2}h_6^3\right)+h_{10}^2-\frac{3}{2}h_2h_6^3, \\
y_3=& \, -2x_1^{10}h_2+6x_1^8h_2^2+x_1^6\left(33h_6-14h_2^3\right)-x_1^4\left(33h_2h_6-6h_2^4\right) \nonumber \\
&+x_1^2\left(11h_{10}-2h_2^5\right)-h_{10}h_2+\frac{3}{2}h_6^2, \\
y_4=& \, x_1^2+h_2, \label{H4y4x}
\end{align}
where
\begin{align}
h_2=& \, \epsilon_1, \\
h_6=& \, \sqrt{5}\delta+\epsilon_1\epsilon_2-11\epsilon_3, \\
h_{10}=& \, 95\epsilon_2\epsilon_3-32\epsilon_1^2\epsilon_3-5\epsilon_1\epsilon_2^2+2\epsilon_1^3\epsilon_2+3\sqrt{5}\delta\epsilon_2,
\end{align}
and
\begin{align}
\epsilon_1=& \, x_2^2+x_3^2+x_4^2, \\
\epsilon_2=& \, x_2^2x_3^2+x_2^2x_4^2+x_3^2x_4^2, \\
\epsilon_3=& \, x_2^2x_3^2x_4^2, \\
\delta=& \, (x_2^2-x_3^2)(x_2^2-x_4^2)(x_3^2-x_4^2). \label{H4delta}
\end{align}
The basic invariants $y_1, y_2, y_3, y_4$ have degrees $30, 20, 12, 2,$ respectively.
\begin{lemma} \label{H4yIntersectionForm}
(cf. \cite{Saito}) The entries of the intersection form $g^{ij}(y)$ are
\begin{align*}
g^{11}(y)=& \, \frac{928}{3}y_2y_3^3y_4+240y_1y_3^2y_4^2+96y_2^2y_3y_4^3+160y_1y_2y_4^4, \\
g^{12}(y)=& \, -32y_3^4-112y_2y_3^2y_4^2-120y_1y_3y_4^3+48y_2^2y_4^4, \\
g^{22}(y)=& \, \frac{152}{3}y_3^3y_4-56y_2y_3y_4^3+20y_1y_4^4, \\
g^{13}(y)=& \, -80y_2^2-\frac{16}{3}y_3^3y_4^2-16y_2y_3y_4^4-40y_1y_4^5, \\
g^{23}(y)=& \, -30y_1+8y_3^2y_4^3-24y_2y_4^5, \hspace{6mm} g^{33}(y)=44y_2y_4-8y_3y_4^5, \\
g^{14}(y)=& \, 60y_1, \hspace{8mm} g^{24}(y)=40y_2, \hspace{8mm} g^{34}(y)=24y_3, \hspace{8mm} g^{44}(y)=5y_4.
\end{align*}
\end{lemma}

\noindent Consider another set of basic invariants for $H_4$ given by
\begin{align}
Y_1=& \, y_1-\frac{y_4^3}{30030}\left(4y_4^{12}+320y_3y_4^6+7051y_2y_4^2-715y_3^2\right), \label{H4LaplaceY1} \\
Y_2=& \, y_2+\frac{y_4^4}{748}\left(3y_4^6+110y_3\right), \label{H4LaplaceY2} \\
Y_3=& \, y_3+\frac{y_4^6}{14}, \label{H4LaplaceY3} \\
Y_4=& \, \frac{1}{8}y_4. \label{H4LaplaceY4}
\end{align}
The following statement can be checked directly.
\begin{lemma}
We have $\Delta(Y_4)=1$ and $\Delta(Y_1)=\Delta(Y_2)=\Delta(Y_3)=0.$
\end{lemma}

\noindent Note that for examples $H_4(3), H_4(4)$ and $H_4(7)$ we will omit the relations for $t_i$ and $Z$ as functions of the basic invariants $y_j,$ as they become too long. Likewise, we omit analogues of Propositions \ref{PropD4}, \ref{D4discrim} and \ref{PropD4unity}.

\subsection{$H_4(1)$ example}
The prepotential for $H_4(1)$ is
\begin{align*}
F(t)=& \, t_1t_2t_3+\frac{1}{2}t_1^2t_4+\frac{3356}{665}t_4^{21}+\frac{64}{5}t_3t_4^{16}+\frac{472}{11}t_3^2t_4^{11}+\frac{16}{3}t_3^3t_4^6+28t_3^4t_4-\frac{16}{15}t_2t_4^{15} \nonumber \\
&+8t_2t_3t_4^{10}+32t_2t_3^2t_4^5-\frac{8}{3}t_2t_3^3+\frac{19}{18}t_2^2t_4^9-t_2^2t_3t_4^4+\frac{1}{6}t_2^3t_4^3+\frac{1}{105}Z^7,
\end{align*}
where
\begin{equation} \label{ZEqH41}
P(t_2, t_3, t_4, Z):=Z^2-4t_4(t_4^5-3t_3)-t_2=0.
\end{equation}
The Euler vector field is
\begin{equation*}
E(t)=t_1\partial_{t_1}+\frac{3}{5}t_2\partial_{t_2}+\frac{1}{2}t_3\partial_{t_3}+\frac{1}{10}t_4\partial_{t_4},
\end{equation*}
the unity vector field is $e(t)=\partial_{t_1},$ and the charge is $d=\frac{9}{10}.$ The intersection form (\ref{gDef}) is then given by
\begin{align}
g^{11}(t)=& \, \frac{1}{10}\left(228t_2t_3^2Z+19t_2^3t_4-2736t_3^3t_4Z-228t_2^2t_3t_4^2+12160t_2t_3^2t_4^3+76t_2^2t_4^4Z \right. \nonumber \\
&+3040t_3^3t_4^4-2736t_2t_3t_4^5Z+22800t_3^2t_4^6Z+1444t_2^2t_4^7+13680t_2t_3t_4^8+89680t_3^2t_4^9 \nonumber \\
&\left. +1520t_2t_4^{10}Z-21888t_3t_4^{11}Z-4256t_2t_4^{13}+58368t_3t_4^{14}+4864t_4^{16}Z+40272t_4^{19}\right), \label{H41intersectionForm1} \\
g^{12}(t)=& \, -\frac{3}{5}\left(t_2^2Z-280t_3^3-54t_2t_3t_4Z+504t_3^2t_4^2Z+10t_2^2t_4^3-800t_2t_3t_4^4-240t_3^2t_4^5 \right. \nonumber \\
&\left. +68t_2t_4^6Z-936t_3t_4^7Z-200t_2t_4^9-2360t_3t_4^{10}+256t_4^{12}Z-512t_4^{15}\right), \\
g^{22}(t)=& \, -\frac{2}{5}\left(44t_2t_3-924t_3^2t_4-33t_2t_4^2Z+396t_3t_4^3Z-176t_2t_4^5-88t_3t_4^6-132t_4^8Z \right. \nonumber \\
&\left. -236t_4^{11}\right), \\
g^{13}(t)=& \, \frac{7}{10}\left(-2t_2t_3Z+24t_3^2t_4Z+3t_2^2t_4^2-16t_2t_3t_4^3+320t_3^2t_4^4+4t_2t_4^5Z-56t_3t_4^6Z \right. \nonumber \\
&\left. +38t_2t_4^8+160t_3t_4^9+16t_4^{11}Z-32t_4^{14}\right), \\
g^{23}(t)=& \, t_1-8t_3^2-t_2t_4Z+12t_3t_4^2Z-2t_2t_4^4+64t_3t_4^5-4t_4^7Z+8t_4^{10}, \\
g^{33}(t)=& \, \frac{1}{40}\left(3t_2Z-36t_3t_4Z+36t_2t_4^3-72t_3t_4^4+12t_4^6Z+76t_4^9\right), \\
g^{14}(t)=& \, t_1, \hspace{5mm} g^{24}(t)=\frac{3}{5}t_2, \hspace{5mm} g^{34}(t)=\frac{1}{2}t_3, \hspace{5mm} g^{44}(t)=\frac{1}{10}t_4. \label{H41intersectionForm2}
\end{align}
We have that $\mathrm{deg} \, t_1(x)=20, \, \mathrm{deg} \, t_2(x)=12, \, \mathrm{deg} \, t_3(x)=10, \, \mathrm{deg} \, t_4(x)=2$ and $\mathrm{deg} \, Z(x)=6.$

\begin{proposition} \label{PropH41}
Let $V_1=\{p \in \C[t_1, t_2, t_3, t_4, Z] \mid \, \mathrm{deg} \, p(x)=30 \},$ let $V_2=\{p \in \C[t_1, t_2, t_3, t_4, Z] \mid \, \mathrm{deg} \, p(x)=20 \}$ and let $V_3=\{p \in \C[t_1, t_2, t_3, t_4, Z] \mid \, \mathrm{deg} \, p(x)=12 \}.$ The harmonic elements of $V_1$ are proportional to
\begin{align*}
27027t_2^2Z&-1801800t_1t_3-6806800t_3^3-648648t_2t_3t_4Z+3891888t_3^2t_4^2Z+32175t_2^2t_4^3 \nonumber \\
&+13556400t_2t_3t_4^4-2335080t_1t_4^5+90256320t_3^2t_4^5+216216t_2t_4^6Z \nonumber \\
&-25594592t_3t_4^7Z-4591440t_2t_4^9+35834480t_3t_4^{10}+432432t_4^{12}Z-864864t_4^{15},
\end{align*}
the harmonic elements of $V_2$ are proportional to
\begin{equation*}
561t_1+7106t_3^2-627t_2t_4^4-37620t_3t_4^5-12350t_4^{10},
\end{equation*}
and the harmonic elements of $V_3$ are proportional to
\begin{equation*}
21t_2+308t_3t_4-220t_4^6.
\end{equation*}
\end{proposition}
\begin{proof}
Using Proposition \ref{Prop12} we can directly calculate
\begin{align}
\Delta(t_1)=& \, -\frac{19}{10}\left(t_2Z-30t_3t_4Z+8t_2t_4^3-320t_3t_4^4+40t_4^6Z-80t_4^9\right), \label{Delta1H41} \\
\Delta(t_2)=& \, -\frac{11}{5}\left(8t_3-9t_4^2Z-32t_4^5\right), \label{Delta2H41} \\
\Delta(t_3)=& \, -\frac{9}{20}t_4\left(3Z+4t_4^3\right), \label{Delta3H41} \\
\Delta(t_4)=& \, \frac{1}{5}. \label{Delta4H41}
\end{align}
A general element of $V_1$ is of the form
\begin{align}
a_1t_1t_3&+a_2t_1t_4^5+a_3t_1t_4^2Z+a_4t_2^2t_4^3+a_5t_2^2Z+a_6t_2t_3t_4^4+a_7t_2t_3t_4Z+a_8t_2t_4^9+a_9t_2t_4^6Z \nonumber \\
&+a_{10}t_3^3+a_{11}t_3^2t_4^5+a_{12}t_3^2t_4^2Z+a_{13}t_3t_4^{10}+a_{14}t_3t_4^7Z+a_{15}t_4^{15}+a_{16}t_4^{12}Z, \label{DeadTreesH41}
\end{align}
where $a_i \in \C.$ By calculating the Laplacian of this general element (\ref{DeadTreesH41}) using Proposition \ref{Prop12} and formulas (\ref{Delta1H41})--(\ref{Delta4H41}) we find that the only harmonic elements of $V_1$ are as claimed. A general element of $V_2$ has the form
\begin{equation} \label{BlocksH41}
b_1t_1+b_2t_2t_4^4+b_3t_2t_4Z+b_4t_3^2+b_5t_3t_4^5+b_6t_3t_4^2Z+b_7t_4^{10}+b_8t_4^7Z,
\end{equation}
where $b_i \in \C.$ By calculating the Laplacian of this general element (\ref{BlocksH41}) using Propostion \ref{Prop12} and formulas (\ref{Delta1H41})--(\ref{Delta4H41}) we find that the only harmonic elements of $V_2$ are as claimed. A general element of $V_3$ has the form
\begin{equation} \label{LavaH41}
c_1t_2+c_2t_3t_4+c_3t_4^6+c_4t_4^3Z,
\end{equation}
where $c_i \in \C.$ By calculating the Laplacian of this general element (\ref{LavaH41}) using Propostion \ref{Prop12} and formulas (\ref{Delta1H41})--(\ref{Delta4H41}) we find that the only harmonic elements of $V_3$ are as claimed.
\end{proof}

\begin{theorem} \label{theoremYH41}
We have the following relations
\begin{align}
y_1=& \, \frac{2^{30}5^9}{3}\left(27t_2^2Z-1800t_1t_3-6800t_3^3-648t_2t_3t_4Z \right. \nonumber \\
&+3888t_3^2t_4^2+12600t_2t_3t_4^4+6120t_1t_4^5+190320t_3^2t_4^5+216t_2t_4^6Z \nonumber \\
&\left. -2592t_3t_4^7Z-42840t_2t_4^9-953520t_3t_4^{10}+432t_4^{12}Z+1309136t_4^{15}\right), \label{y1EqH41} \\
y_2=& \, 2^{20}5^7\left(3t_1+38t_3^2-21t_2t_4^4-460t_3t_4^5+950t_4^{10}\right), \label{y2EqH41} \\
y_3=& \, 2^{11}5^4\left(3t_2+44t_3t_4-260t_4^6\right), \label{y3EqH41} \\
y_4=& \, 40t_4. \label{y4EqH41}
\end{align}
\end{theorem}
\begin{proof}
Note that $Y_4=\frac{1}{8}y_4=5t_4.$ We now equate $Y_1, \, Y_2$ and $Y_3$ given by relations (\ref{H4LaplaceY1})--(\ref{H4LaplaceY3}) with general harmonic elements of $V_1, \, V_2$ and $V_3,$ respectively, given by Proposition \ref{PropH41}. We then rearrange these equations to find $y_i$ in terms of $t_j$ and $Z.$ We find
\begin{align}
y_1=& \, \frac{2^{42}3^25^{14}13}{7^311}t_4^{15}+\frac{a}{2^33^37\!\cdot\!11\!\cdot\!13}\left(27027t_2^2Z-1801800t_1t_3 \right. \nonumber \\
&-6806800t_3^3-648648t_2t_3t_4Z+3891888t_3^2t_4^2Z+32175t_2^2t_4^3+13556400t_2t_3t_4^4 \nonumber \\
&-2335080t_1t_4^5+90256320t_3^2t_4^5+216216t_2t_4^6Z-2594592t_3t_4^7Z-4591440t_2t_4^9 \nonumber \\
&\left. +35834480t_3t_4^{10}+432432t_4^{12}Z-864864t_4^{15}\right)-\frac{131276800b}{67431}t_4^5\left(561t_1+7106t_3^2 \right. \nonumber \\
&\left.-627t_2t_4^4-37620t_3t_4^5-12350t_4^{10}\right)+\frac{2^{23}5^82251c}{302379}t_4^9\left(21t_2+308t_3t_4 \right. \nonumber \\
&\left. -220t_4^6\right)-\frac{80c^2}{2541}t_4^3\left(21t_2+308t_3t_4-220t_4^6\right)^2, \label{y1FakeH41} \\
y_2=& \, \frac{2^{29}5^{10}}{77}t_4^{10}+\frac{b}{12350}\left(-561t_1-7106t_3^2+627t_2t_4^4+37620t_3t_4^5 \right. \nonumber \\
&\left. +12350t_4^{10}\right)+\frac{320000c}{187}t_4^4\left(21t_2+308t_3t_4-220t_4^6\right), \label{y2FakeH41} \\
y_3=& \, -\frac{2^{17}5^6}{7}t_4^6+\frac{c}{220}\left(-21t_2-308t_3t_4+220t_4^6\right), \label{y3FakeH41} \\
y_4=& \, 40t_4, \label{y4FakeH41}
\end{align}
where $a, b, c \in \C.$ In order to find $a, b$ and $c$ we perform steps 5--7 from Section \ref{sec2.4}. That is, we transform the intersection form (\ref{H41intersectionForm1})--(\ref{H41intersectionForm2}) into $y$ coordinates by applying formulas (\ref{y1FakeH41})--(\ref{y4FakeH41}) and compare it with the expression given by Lemma \ref{H4yIntersectionForm}. We find that
\begin{equation*}
a=2^{33}3^25^9, \hspace{5mm} b=-\frac{2^{21}5^913\!\cdot\!19}{11\!\cdot\!17}, \hspace{5mm} c=-\frac{2^{13}5^511}{7},
\end{equation*}
which implies the statement.
\end{proof}

\begin{proposition}
The derivatives $\frac{\partial y_i}{\partial t_j} \in \C[t_1, t_2, t_3, t_4, Z].$
\end{proposition}
\noindent Proof is similar to the one for Proposition \ref{DerivsProp}.

\begin{proposition} \label{H41discrim}
We have that
\begin{equation*}
\mathrm{det}(g^{ij}(t))=\frac{c\prod\limits_{\alpha \in R_{H_4}}(\alpha, \, x)}{Q(t, Z)^2},
\end{equation*}
where $c=5$ and
\begin{equation*}
Q(t, Z)=5^{20}\left(5t_1-70t_3^2+5t_2t_4Z-60t_3t_4^2Z-35t_2t_4^4+140t_3t_4^5+20t_4^7Z+42t_4^{10}\right).
\end{equation*}
\end{proposition}
\noindent By Proposition \ref{DiscrimProp}, we need only find $\mathrm{det}\left(\frac{\partial y_i}{\partial t_j}\right).$ It can be calculated by Theorem \ref{theoremYH41}, which leads to Proposition \ref{H41discrim}.

In the next statement we express flat coordinates $t_i$ via basic invariants $y_j$ and $Z,$ which is an inversion of the formulas from Theorem \ref{theoremYH41}.
\begin{theorem}
We have the following relations:
\begin{align}
t_1=& \, \frac{1}{2^{27}5^{12}3y_4^2}\left(-2^{26}3^25^{12}19 \, Z^4+2^{16}5^83\!\cdot\!19 \, y_3Z^2-2^45^419 \, y_3^2 \right. \nonumber \\
&\left. +2^75^5y_2y_4^2+2^{13}5^63\!\cdot\!7\!\cdot\!19 \, y_4^6Z^2+2^23^45^27 \, y_3y_4^6+3^27\!\cdot\!47 \, y_4^{12}\right), \label{t1EqH41} \\
t_2=& \, \frac{1}{2^{15}5^7}\left(2^{13}5^611 \, Z^2+2^25^23 \, y_3+23y_4^6\right), \label{t2EqH41} \\
t_3=& \, \frac{1}{2^{15}5^6y_4}\left(-2^{14}5^63 \, Z^2+2^35^2y_3+17y_4^6\right), \label{t3EqH41} \\
t_4=& \, \frac{1}{40}y_4, \label{t4EqH41}
\end{align}
where $Z$ satisfies the equation
\begin{align}
&2^{34}3^35^{12}Z^6-2^{30}3^35^9y_4^3Z^5-2^{23}3^35^8y_3Z^4+2^{12}3^25^4y_3^2Z^2 \nonumber \\
&-2^{12}3^25^4y_2y_4^2Z^2-2y_3^3+6y_2y_3y_4^2+3y_1y_4^3=0,\label{DeltaPolyH41}
\end{align}
and $y_i$ are given by relations (\ref{H4y1x})--(\ref{H4delta}).
\end{theorem}
\begin{proof}
Formula (\ref{t4EqH41}) follows immediately from Theorem \ref{theoremYH41}. Using relations (\ref{ZEqH41}) and (\ref{y3EqH41}) we see that
\begin{align*}
t_2=& \, Z^2+\frac{3}{10}t_3y_4-\frac{y_4^6}{2^{16}5^6}, \\
t_3=& \, \frac{1}{2^{15}5^411 \, y_4}\left(-2^{16}5^53 \, t_2+160y_3+13y_4^6\right).
\end{align*}
We can solve this system of equations to find $t_2$ and $t_3$ which gives us formulas (\ref{t2EqH41}) and (\ref{t3EqH41}). Substituting formulas (\ref{t2EqH41})--(\ref{t4EqH41}) into relation (\ref{y2EqH41}) and solving for $t_1$ we get formula (\ref{t1EqH41}). Finally, substituting relations (\ref{t1EqH41})--(\ref{t4EqH41}) into formula (\ref{y1EqH41}) we get the formula (\ref{DeltaPolyH41}).
\end{proof}

\begin{proposition}
The unity vector field $e=\partial_{t_1}$ in the $y$ coordinates has the form
\begin{equation*}
e(y)=2^{20}5^73 \, \partial_{y_2}-2^{33}5^{10}3\left(5t_3-17t_4^5\right)\partial_{y_1}.
\end{equation*}
\end{proposition}
\begin{proof}
We have that
\begin{equation*}
e=\partial_{t_1}=\frac{\partial y_\alpha}{\partial t_1}\partial_{y_\alpha},
\end{equation*}
which gives the statement by applying the relations from Theorem \ref{theoremYH41}.
\end{proof}

\subsection{$H_4(2)$ example}
The prepotential for $H_4(2)$ is
\begin{align*}
F(t)=& \, -\frac{66084040}{73920}t_4^{16}+\frac{143564400}{73920}t_3^2t_4^{11}-\frac{40727610}{73920}t_3^4t_4^6-\frac{392931}{73920}t_3^6t_4 \\
&+t_1t_2t_3+\frac{1}{2}t_1^2t_4-\frac{3}{4}t_4^4\left(2288t_4^{10}-1620t_3^2t_4^5-27t_3^4\right)Z-760t_4^{12}Z^2 \\
&+\left(\frac{1744}{48}t_4^{10}-\frac{4860}{48}t_3^2t_4^5-\frac{81}{48}t_3^4\right)Z^3+140t_4^8Z^4+24t_4^6Z^5 \\
&-\frac{53}{6}t_4^4Z^6-\frac{10}{7}t_4^2Z^7+\frac{Z^8}{4},
\end{align*}
where
\begin{equation} \label{ZEqH42}
P(t_2, t_3, t_4, Z):=Z^3-12t_4^4Z-11t_4^6+\frac{27}{4}t_3^2t_4-t_2=0.
\end{equation}
The Euler vector field is
\begin{equation*}
E(t)=t_1\partial_{t_1}+\frac{4}{5}t_2\partial_{t_2}+\frac{1}{3}t_3\partial_{t_3}+\frac{2}{15}t_4\partial_{t_4},
\end{equation*}
the unity vector field is $e(t)=\partial_{t_1},$ and the charge is $d=\frac{13}{15}.$ The intersection form (\ref{gDef}) is then given by
\begin{align}
g^{11}(t)=& \, \frac{1}{6}\left(-32Zt_2^2+567Z^2t_3^4+1440Zt_2t_3^2t_4-112t_2^2t_4^2-10530Zt_3^4t_4^2-20160t_2t_3^2t_4^3 \right. \nonumber \\
&-2784Z^2t_2t_4^4+20979t_3^4t_4^4+27864Z^2t_3^2t_4^5+2208Zt_2t_4^6-23976Zt_3^2t_4^7 \nonumber \\
&\left. +40416t_2t_4^8+648000t_3^2t_4^9-7776Z^2t_4^{10}-186624Zt_4^{12}+182736t_4^{14}\right), \label{H42intersectionForm1} \\
g^{12}(t)=& \, \frac{3}{4}t_3\left(-20Z^2t_2+81t_3^4+360Z^2t_3^2t_4+300Zt_2t_4^2-2925Zt_3^2t_4^3-1840t_2t_4^4 \right. \nonumber \\
&\left. +1170t_3^2t_4^5+1980Z^2t_4^6-1980Zt_4^8+18720t_4^{10}\right), \\
g^{22}(t)=& \, -\frac{3}{10}\left(99t_2t_3^2+44Z^2t_2t_4-891t_3^4t_4-1287Z^2t_3^2t_4^2-220Zt_2t_4^3+5445Zt_3^2t_4^4 \right. \nonumber \\
&\left. +528t_2t_4^5+5841t_3^2t_4^6-396Z^2t_4^7+396Zt_4^9-3744t_4^{11}\right), \\
g^{13}(t)=& \, -9Z^2t_3^2-8Zt_2t_4+90Zt_3^2t_4^2-8t_2t_4^3-468t_3^2t_4^4-72Z^2t_4^5
+72Zt_4^7+792t_4^9, \\
g^{23}(t)=& \, \frac{1}{4}\left(4t_1-27t_3^3-60Z^2t_3t_4+240Zt_3t_4^3-240t_3t_4^5\right), \\
g^{33}(t)=& \, \frac{8}{27}\left(2Z^2-8t_4^2Z-19t_4^4\right), \\
g^{14}(t)=& \, t_1, \hspace{5mm} g^{24}(t)=\frac{4}{5}t_2, \hspace{5mm} g^{34}(t)=\frac{1}{3}t_3, \hspace{5mm} g^{44}(t)=\frac{2}{15}t_4. \label{H42intersectionForm2}
\end{align}
We have that $\mathrm{deg} \, t_1(x)=15, \, \mathrm{deg} \, t_2(x)=12, \, \mathrm{deg} \, t_3(x)=5, \, \mathrm{deg} \, t_4(x)=2$ and $\mathrm{deg} \, Z(x)=4.$

\begin{proposition} \label{PropH42}
Let $V_1=\{p \in \C[t_1, t_2, t_3, t_4, Z] \mid \, \mathrm{deg} \, p(x)=30 \},$ let $V_2=\{p \in \C[t_1, t_2, t_3, t_4, Z] \mid \, \mathrm{deg} \, p(x)=20 \}$ and let $V_3=\{p \in \C[t_1, t_2, t_3, t_4, Z] \mid \, \mathrm{deg} \, p(x)=12 \}.$ The harmonic elements of $V_1$ are proportional to
\begin{align*}
576&5760t_1^2-51891840Z^2t_2t_3^2+188107920t_1t_3^3+302837535 t_3^6+30750720Zt_2^2t_4 \\
&+350269920Z^2t_3^4t_4-155675520Zt_2t_3^2t_4^2+274743040t_2^2t_4^3-350269920Zt_3^4t_4^3 \\
&+3090653280t_2t_3^2t_4^4+337438464Z^2t_2t_4^5+558122400t_1t_3t_4^5+33046299000t_3^4t_4^5 \\
&-1810683072Z^2t_3^2t_4^6+96349440Zt_2t_4^7-1117385280Zt_3^2t_4^8+1662275328t_2t_4^9 \\
&+185379977376t_3^2t_4^{10}+1391157504Z^2t_4^{11}+11062884096Zt_4^{13}-21529753344t_4^{15},
\end{align*}
the harmonic elements of $V_2$ are proportional to
\begin{align*}
47872&Z^2t_2-269280t_1t_3-984555t_3^4-323136Z^2t_3^2t_4-239360Zt_2t_4^2+1615680Zt_3^2t_4^3 \\
&-3512256t_2t_4^4-25208172t_3^2t_4^5-430848Z^2t_4^6+430848Zt_4^8-35274672t_4^{10},
\end{align*}
and the harmonic elements of $V_3$ are proportional to
\begin{equation*}
112t_2+2079t_3^2t_4+5940t_4^6.
\end{equation*}
\end{proposition}
\begin{proof}
Using Proposition \ref{Prop12} we can directly calculate
\begin{align}
\Delta(t_1)=& \, \frac{28t_3}{4t_2-27t_3^2t_4-20t_4^6}\left(-4Z^2t_2+45Z^2t_3^2t_4+48Zt_2t_4^2
-360Zt_3^2t_4^3 \right. \nonumber \\
&\left. -208t_2t_4^4+1260t_3^2t_4^5+100Z^2t_4^6-400Zt_4^8+400t_4^{10}\right), \label{Delta1H42} \\
\Delta(t_2)=& \, -\frac{11}{10\left(4t_2-27t_3^2t_4-20t_4^6\right)}\left(108t_2t_3^2+80Z^2t_2t_4
-729t_3^4t_4 \right. \nonumber \\
&-1260Z^2t_3^2t_4^2-320Zt_2t_4^3+3600Zt_3^2t_4^4+320t_2t_4^5+3060t_3^2t_4^6 \nonumber \\
&\left. -400Z^2t_4^7+1600Zt_4^9-1600t_4^{11}\right), \label{Delta2H42} \\
\Delta(t_3)=& \, \frac{64t_3t_4\left(2t_4^2+Z\right)\left(4t_4^2-Z\right)}{3\left(4t_2-27t_3^2t_4-20t_4^6\right)}, \label{Delta3H42} \\
\Delta(t_4)=& \, \frac{4}{15}. \label{Delta4H42}
\end{align}
A general element of $V_1$ is of the form
\begin{align}
a_1t_1^2&+a_2t_1t_3^3+a_3t_1t_3t_4^5+a_4t_1t_3t_4^3Z+a_5t_1t_3t_4Z^2+a_6t_2^2t_4^3+a_7t_2^2t_4Z \nonumber \\
&+a_8t_2t_3^2t_4^4+a_9t_2t_3^2t_4^2Z+a_{10}t_2t_3^2Z^2+a_{11}t_2t_4^9+a_{12}t_2t_4^7Z+a_{13}t_2t_4^5Z^2 \nonumber \\
&+a_{14}t_3^6+a_{15}t_3^4t_4^5+a_{16}t_3^4t_4^3Z+a_{17}t_3^4t_4Z^2+a_{18}t_3^2t_4^{10}+a_{19}t_3^2t_4^8Z \nonumber \\
&+a_{20}t_3^2t_4^6Z^2+a_{21}t_4^{15}+a_{22}t_4^{13}Z+a_{23}t_4^{11}Z^2, \label{DeadTreesH42}
\end{align}
where $a_i \in \C.$ By calculating the Laplacian of this general element (\ref{DeadTreesH42}) using Proposition \ref{Prop12} and formulas (\ref{Delta1H42})--(\ref{Delta4H42}) we find that the only harmonic elements of $V_1$ are as claimed. A general element of $V_2$ has the form
\begin{align}
b_1t_1t_3&+b_2t_2t_4^4+b_3t_2t_4^2Z+b_4t_2Z^2+b_5t_3^4+b_6t_3^2t_4^5 \nonumber \\
&+b_7t_3^2t_4^3Z+b_8t_3^2t_4Z^2+b_9t_4^{10}+b_{10}t_4^8Z+b_{11}t_4^6Z^2,  \label{BlocksH42}
\end{align}
where $b_i \in \C.$ By calculating the Laplacian of this general element (\ref{BlocksH42}) using Propostion \ref{Prop12} and formulas (\ref{Delta1H42})--(\ref{Delta4H42}) we find that the only harmonic elements of $V_2$ are as claimed. A general element of $V_3$ has the form
\begin{equation} \label{LavaH42}
c_1t_2+c_2t_3^2t_4+c_3t_4^6+c_4t_4^4Z+c_5t_4^2Z^2,
\end{equation}
where $c_i \in \C.$ By calculating the Laplacian of this general element (\ref{LavaH42}) using Propostion \ref{Prop12} and formulas (\ref{Delta1H42})--(\ref{Delta4H42}) we find that the only harmonic elements of $V_3$ are as claimed.
\end{proof}

\begin{theorem} \label{theoremYH42}
We have the following relations
\begin{align}
y_1=& \, 2^83^25^9\left(5760t_1^2-51840Z^2t_2t_3^2+187920t_1t_3^3+302535t_3^6+30720Zt_2^2t_4 \right. \nonumber \\
&+349920Z^2t_3^4t_4-155520Zt_2t_3^2t_4^2+266240t_2^2t_4^3-349920Zt_3^4t_4^3+2782080t_2t_3^2t_4^4 \nonumber \\
&-393216Z^2t_2t_4^5+4665600t_1t_3t_4^5+45198000t_3^4t_4^5+3120768Z^2t_3^2t_4^6 \nonumber \\
&+3747840Zt_2t_4^7-25764480Zt_3^2t_4^8+75859968t_2t_4^9+952477056t_3^2t_4^{10} \nonumber \\
&\left. +7962624Z^2t_4^{11}+4478976Zt_4^{13}+4105230336t_4^{15}\right), \label{y1EqH42} \\
y_2=& \, -2^43^25^6\left(256Z^2t_2-1440t_1t_3-5265t_3^4-1728Z^2t_3^2t_4-1280Zt_2t_4^2 \right. \nonumber \\
&\left. +8640Zt_3^2t_4^3-31488t_2t_4^4-370656t_3^2t_4^5-2304Z^2t_4^6+2304Zt_4^8-2566656t_4^{10}\right), \label{y2EqH42} \\
y_3=& \, -2^35^43\left(16t_2+297t_3^2t_4+4320t_4^6\right), \label{y3EqH42} \\
y_4=& \, 30t_4. \label{y4EqH42}
\end{align}
\end{theorem}
\begin{proof}
Note that $Y_4=\frac{1}{8}y_4=\frac{15}{4}t_4.$ We now equate $Y_1, \, Y_2$ and $Y_3$ given by relations (\ref{H4LaplaceY1})--(\ref{H4LaplaceY3}) with general harmonic elements of $V_1, \, V_2$ and $V_3,$ respectively, given by Proposition \ref{PropH42}. We then rearrange these equations to find $y_i$ in terms of $t_j$ and $Z.$ We find
\begin{align}
y_1=& \, \frac{2^{12}3^{17}5^{14}13}{7^311}t_4^{15}+\frac{a}{2^83\!\cdot\!11\!\cdot\!59\!\cdot\!677}\left(5765760t_1^2-51891840Z^2t_2t_3^2 \right. \nonumber \\
&+188107920t_1t_3^3+302837535t_3^6+30750720Zt_2^2t_4+350269920Z^2t_3^4t_4 \nonumber \\
&-155675520Zt_2t_3^2t_4^2+274743040t_2^2t_4^3-350269920Zt_3^4t_4^3+3090653280t_2t_3^2t_4^4 \nonumber \\
&+337438464Z^2t_2t_4^5+558122400t_1t_3t_4^5+33046299000t_3^4t_4^5-1810683072Z^2t_3^2t_4^6 \nonumber \\
&+96349440Zt_2t_4^7-1117385280Zt_3^2t_4^8+1662275328t_2t_4^9+185379977376t_3^2t_4^{10} \nonumber \\
&\left. +1391157504Z^2t_4^{11}+11062884096Zt_4^{13}-21529753344t_4^{15}\right) \nonumber \\
&-\frac{3^25^4641b}{7\!\cdot\!13\!\cdot\!29\!\cdot\!8447}t_4^5\left(47872Z^2t_2-269280t_1t_3-984555t_3^4-323136Z^2t_3^2t_4 \right. \nonumber \\
&-239360Zt_2t_4^2+1615680Zt_3^2t_4^3-3512256t_2t_4^4-25208172t_3^2t_4^5-430848Z^2t_4^6 \nonumber \\
&\left. +430848Zt_4^8-35274672t_4^{10}\right)-\frac{2^53^55^82251c}{7^211^217}t_4^9\left(112t_2+2079t_3^2t_4+5940t_4^6\right) \nonumber \\
&-\frac{5c^2}{2^23^411^27}t_4^3\left(112t_2+2079t_3^2t_4+5940t_4^6\right)^2, \label{y1FakeH42} \\
y_2=& \, \frac{2^93^{10}5^{10}}{77}t_4^{10}+\frac{b}{2^43^229\!\cdot\!8447}\left(-47872Z^2t_2+269280t_1t_3+984555t_3^4 \right. \nonumber \\
&+323136Z^2t_3^2t_4+239360Zt_2t_4^2-1615680Zt_3^2t_4^3+3512256t_2t_4^4+25208172t_3^2t_4^5 \nonumber \\
&\left. +430848Z^2t_4^6-430848Zt_4^8+35274672t_4^{10}\right)-\frac{3750c}{187}t_4^4\left(112t_2+2079t_3^2t_4+5940t_4^6\right), \label{y2FakeH42} \\
y_3=& \, -\frac{2^53^65^6}{7}t_4^6+\frac{c}{5940}\left(112t_2+2079t_3^2t_4+5940t_4^6\right), \label{y3FakeH42} \\
y_4=& \, 30t_4, \label{y4FakeH42}
\end{align}
where $a, b, c \in \C.$ In order to find $a, b$ and $c$ we perform steps 5--7 from Section \ref{sec2.4}. That is, we transform the intersection form (\ref{H42intersectionForm1})--(\ref{H42intersectionForm2}) into $y$ coordinates by applying formulas (\ref{y1FakeH42})--(\ref{y4FakeH42}) and compare it with the expression given by Lemma \ref{H4yIntersectionForm}. We find that
\begin{equation*}
a=\frac{2^{16}3^35^959\!\cdot\!677}{91}, \hspace{5mm} b=\frac{2^83^45^629\!\cdot\!8447}{11\!\cdot\!17}, \hspace{5mm} c=-\frac{2^53^45^511}{7},
\end{equation*}
which implies the statement.
\end{proof}

\begin{proposition}
The derivatives $\frac{\partial y_i}{\partial t_j} \in \C[t_1, t_2, t_3, t_4, Z].$
\end{proposition}
\noindent Proof is similar to the one for Proposition \ref{DerivsProp}.

\begin{proposition} \label{H42discrim}
We have that
\begin{equation*}
\mathrm{det}(g^{ij}(t))=\frac{c\prod\limits_{\alpha \in R_{H_4}}(\alpha, \, x)}{Q(t, Z)^2},
\end{equation*}
where $c=2^{72}5$ and
\begin{align*}
Q(t, Z)=& \, 3^75^{20}\left(4t_1^2+36Z^2t_2t_3^2-72t_1t_3^3+324t_3^6+60Z^2t_1t_3t_4-783Z^2t_3^4t_4 \right. \\
&+1080Zt_2t_3^2t_4^2-240Zt_1t_3t_4^3-5130Zt_3^4t_4^3+4932t_2t_3^2t_4^4-1380t_1t_3t_4^5 \\
&\left. -20871t_3^4t_4^5+11016Z^2t_3^2t_4^6-11016Zt_3^2t_4^8-194076t_3^2t_4^{10}\right).
\end{align*}
\end{proposition}
\noindent By Proposition \ref{DiscrimProp}, we need only find $\mathrm{det}\left(\frac{\partial y_i}{\partial t_j}\right).$ It can be calculated by Theorem \ref{theoremYH42}, which leads to Proposition \ref{H42discrim}.

In the next statement we express flat coordinates $t_i$ via basic invariants $y_j$ and $Z,$ which is an inversion of the formulas from Theorem \ref{theoremYH42}.
\begin{theorem}
We have the following relations:
\begin{align}
t_1=& \, \frac{1}{2^83^{11}5^9\sqrt{3}y_4^{\frac{3}{2}}\sqrt{-2^73^65^6Z^3-3^55^2y_3+2^53^35^2y_4^4Z-518y_4^6}} \nonumber \\
&\left(-2^{14}3^{12}5^{12}13 \, Z^6-2^83^{11}5^813 \, Z^3y_3-3^{10}5^413 \, y_3^2+2^{12}3^{14}5^{11}Z^5y_4^2 \right. \nonumber \\
&+3^{12}5^5 \, y_2y_4^2-2^{10}3^95^8571 \, Z^4y_4^4+2^63^85^413 \, Zy_3y_4^4 \nonumber \\
&-2^83^65^611\!\cdot\!2269 \, Z^3y_4^6+2^33^55^217\!\cdot\!19\!\cdot\!23 \, y_3y_4^6+2^93^65^4199 \, Z^2y_4^8 \nonumber \\
&\left. +2^63^35^222259 \, Zy_4^{10}+2^47\!\cdot\!180569 \, y_4^{12}\right), \label{t1EqH42} \\
t_2=& \, \frac{1}{2^63^75^7}\left(2^63^65^611 \, Z^3-3^55^22 \, y_3-2^43^35^211 \, y_4^4Z-1201y_4^6\right), \label{t2EqH42} \\
t_3=& \, \frac{1}{3^45^32\sqrt{3}}\sqrt{\frac{-2^73^65^6Z^3-3^55^2y_3+2^53^35^2y_4^4Z-518y_4^6}{y_4}}, \label{t3EqH42} \\
t_4=& \, \frac{1}{30}y_4, \label{t4EqH42}
\end{align}
where $Z$ satisfies the equation
\begin{align}
&2^{28}3^{24}5^{24}Z^{12}+2^{23}3^{23}5^{20}Z^9y_3+2^{15}3^{23}5^{16}Z^6y_3^2+2^93^{21}5^{12}Z^3y_3^3 \nonumber \\
&+3^{20}5^8y_3^4+2^{27}3^{25}5^{22}Z^{11}y_4^2-2^{15}3^{23}5^{16}Z^6y_2y_4^2+2^{21}3^{24}5^{18}Z^8y_3y_4^2 \nonumber \\
&-2^93^{22}5^{12}Z^3y_2y_3y_4^2+2^{13}3^{23}5^{14}Z^5y_3^2y_4^2-3^{21}5^82 \, y_2y_3^2y_4^2 \nonumber \\
&-2^83^{22}5^{12}Z^3y_1y_4^3-3^{21}5^82 \, y_1y_3y_4^3+2^{24}3^{21}5^{20}353 \, Z^{10}y_4^4 \nonumber \\
&-2^{13}3^{23}5^{14}Z^5y_2y_4^4-3^{21}5^8y_2^2y_4^4+2^{20}3^{21}5^{16}13 \, Z^7y_3y_4^4 \nonumber \\
&-2^{11}3^{20}5^{12}23 \, Z^4y_3^2y_4^4-2^73^{18}5^8Zy_3^3y_4^4+2^{25}3^{18}5^{19}223 \, Z^9y_4^6 \nonumber \\
&+2^{11}3^{20}5^{12}23 \, Z^4y_2y_4^6+2^{18}3^{18}5^{14}53 \, Z^6y_3y_4^6+2^73^{19}5^8Zy_2y_3y_4^6 \nonumber \\
&-2^93^{17}5^{10}17\!\cdot\!41 \, Z^3y_3^2y_4^6+2^33^{15}5^67\!\cdot\!37 \, y_3^3y_4^6+2^63^{19}5^8Zy_1y_4^7 \nonumber \\
&+2^{20}3^{19}5^{17}7\!\cdot\!19 \, Z^8y_4^8+2^93^{17}5^{10}17\!\cdot\!41 \, Z^3y_2y_4^8-2^{15}3^{19}5^{12}59 \, Z^5y_3y_4^8 \nonumber \\
&-2^33^{16}5^67\!\cdot\!37 \, y_2y_3y_4^8+2^{10}3^{17}5^87 \, Z^2y_3^2y_4^8-2^23^{16}5^67\!\cdot\!37 \, y_1y_4^9 \nonumber \\
&-2^{19}3^{16}5^{15}2029 \, Z^7y_4^{10}-2^{10}3^{17}5^87 \, Z^2y_2y_4^{10}-2^{13}3^{15}5^{10}7\!\cdot\!191 \, Z^4y_3y_4^{10} \nonumber \\
&+2^73^{14}5^67^213 \, Zy_3^2y_4^{10}-2^{16}3^{13}5^{15}2213 \, Z^6y_4^{12}-2^73^{14}5^67^213 \, Zy_2y_4^{12} \nonumber \\
&+2^{16}3^{12}5^87^211 \, Z^3y_3y_4^{12}-2^83^{11}5^47^3y_3^2y_4^{12}+2^{16}3^{13}5^{11}7\!\cdot\!11\!\cdot\!43 \, Z^5y_4^{14} \nonumber \\
&+2^83^{11}5^47^3y_2y_4^{14}+2^{12}3^{12}5^67^3Z^2y_3y_4^{14}+2^{14}3^{10}5^97^21171 \, Z^4y_4^{16} \nonumber \\
&-2^{11}3^95^47^4Zy_3y_4^{16}-2^{16}3^65^77^3163 \, Z^3y_4^{18}+2^93^55^27^5y_3y_4^{18} \nonumber \\
&-2^{12}3^75^47^447 \, Z^2y_4^{20}+2^{12}3^35^27^541 \, Zy_4^{22}-2^{10}7^611 \, y_4^{24}=0, \label{DeltaPolyH42}
\end{align}
and $y_i$ are given by relations (\ref{H4y1x})--(\ref{H4delta}).
\end{theorem}
\begin{proof}
Formula (\ref{t4EqH42}) follows immediately from Theorem \ref{theoremYH42}. Using relations (\ref{ZEqH42}) and (\ref{y2EqH42}) we see that
\begin{align}
t_1=& \, -\frac{1}{2^93^{11}5^{11}t_3}\left(-2^{12}3^95^{10}Z^5+2^43^{13}5^{11}13t_3^4-3^75^4y_2+2^{10}3^75^9y_4^2Z^4+2^83^85^7y_4^4Z^3 \right. \nonumber \\
&\left. \hspace{25mm} +2^43^{10}5^7y_4^5t_3^2-2^93^35^5y_4^6Z^2-2^65^33\cdot7\cdot11y_4^8Z+2^27^259y_4^{10}\right), \label{firstEqH42} \\
t_2=& \, \frac{1}{2^63^65^6}\left(2^63^65^6Z^3+2^33^85^5t_3^2y_4-2^43^35^2y_4^4Z-11y_4^6\right). \label{secondEqH42}
\end{align}
We also have relation (\ref{y3EqH42}), which together with relations (\ref{firstEqH42}) and (\ref{secondEqH42}) gives formulas (\ref{t1EqH42})--(\ref{t3EqH42}). Finally, by substituting relations (\ref{t1EqH42})--(\ref{t4EqH42}) into formula (\ref{y1EqH42}) we get the formula (\ref{DeltaPolyH42}).
\end{proof}

\begin{proposition}
The unity vector field $e=\partial_{t_1}$ in the $y$ coordinates has the form
\begin{equation*}
e(y)=2^93^45^7t_3\partial_{y_2}-2^{12}3^45^{10}\left(16t_1+261t_3^3+6480t_3t_4^5\right)\partial_{y_1}.
\end{equation*}
\end{proposition}
\begin{proof}
We have that
\begin{equation*}
e=\partial_{t_1}=\frac{\partial y_\alpha}{\partial t_1}\partial_{y_\alpha},
\end{equation*}
which gives the statement by applying the relations from Theorem \ref{theoremYH42}.
\end{proof}

\subsection{$H_4(3)$ example}
The prepotential for $H_4(3)$ is
\begin{align*}
F(t)=& \, \frac{2016569088}{43793750}t_4^{13}+\frac{7929073152}{43793750}t_4^{10}t_3+\frac{11291664384}{43793750}t_4^7t_3^2-\frac{6228045824}{43793750}t_4^4t_3^3 \\
&-\frac{1582124544}{43793750}t_4t_3^4+t_1t_2t_3+\frac{1}{2}t_1^2t_4-\frac{256}{9375}t_4^3\left(t_4^3-t_3\right)\left(779t_4^6+20532t_4^3t_3-18480t_3^2\right) \\
&+\frac{32256}{3125}t_4^2\left(17t_4^3-10t_3\right)\left(t_4^3-t_3\right)^2-\frac{7168}{125}t_4\left(t_4^3-t_3\right)^3Z^3+\frac{96}{5}t_4\left(t_4^3-t_3\right)^2Z^6 \\
&-\frac{8}{2625}\left(1573t_4^9-27588t_4^6t_3+25536t_4^3t_3^2-2352t_3^3\right)Z^4+\frac{544}{175}\left(t_4^3-t_3\right)^2Z^7 \\
&-\frac{288}{125}t_4^2\left(17t_4^3-10t_3\right)\left(t_4^3-t_3\right)Z^5+\frac{9}{70}t_4^2\left(17t_4^3-10t_3\right)Z^8+\frac{50}{1911}Z^{13} \\
&-\frac{15}{7}t_4\left(t_4^3-t_3\right)Z^9-\frac{10}{21}\left(t_4^3-t_3\right)Z^{10}+\frac{125}{1568}t_4Z^{12},
\end{align*}
where
\begin{equation*}
P(t_2, t_3, t_4, Z):=Z^4-\frac{224}{25}\left(t_4^3-t_3\right)Z+\frac{48}{25}t_4^4+\frac{224}{25}t_4t_3-t_2=0.
\end{equation*}
The Euler vector field is
\begin{equation*}
E(t)=t_1\partial_{t_1}+\frac{2}{3}t_2\partial_{t_2}+\frac{1}{2}t_3\partial_{t_3}+\frac{1}{6}t_4\partial_{t_4},
\end{equation*}
the unity vector field is $e(t)=\partial_{t_1},$ and the charge is $d=\frac{5}{6}.$ The intersection form (\ref{gDef}) is then given by
\begin{align*}
g^{11}(t)=& \, -\frac{2}{765625}\left(4812500t_2^2t_3-10780000Zt_2t_3^2+36220800Z^2t_3^3-1203125Z^2t_2^2t_4 \right. \nonumber \\
&+10010000Z^3t_2t_3t_4+36005200t_2t_3^2t_4+305220608Zt_3^3t_4+1375000Zt_2^2t_4^2 \nonumber \\
&+40040000Z^2t_2t_3t_4^2-193177600Z^3t_3^2t_4^2-1491331072t_3^3t_4^2-24234375t_2^2t_4^3 \nonumber \\
&-43120000Zt_2t_3t_4^3-236297600Z^2t_3^2t_4^3-33110000Z^3t_2t_4^4-607006400t_2t_3t_4^4 \nonumber \\
&-1372388864Zt_3^2t_4^4+36190000Z^2t_2t_4^5+260198400Z^3t_3t_4^5+1541281280t_3^2t_4^5 \nonumber \\
&-13200000Zt_2t_4^6-395225600Z^2t_3t_4^6+119891200t_2t_4^7+2865967104Zt_3t_4^7 \nonumber \\
&+184307200Z^3t_4^8-5516836864t_3t_4^8-91660800Z^2t_4^9-1231515648Zt_4^{10} \nonumber \\
&\left. -5094508544t_4^{11}\right), \\
g^{12}(t)=& \, -\frac{8}{21875}\left(3125Zt_2^2-45500Z^2t_2t_3+156800Z^3t_3^2+592704t_3^3+28125t_2^2t_4 \right. \nonumber \\
&-119000Zt_2t_3t_4+619360Z^2t_3^2t_4+45000Z^3t_2t_4^2+119700t_2t_3t_4^2 \nonumber \\
&+2847488Zt_3^2t_4^2+75500Z^2t_2t_4^3-851200Z^3t_3t_4^3-7902720t_3^2t_4^3 \nonumber \\
&-66000Zt_2t_4^4-619360Z^2t_3t_4^4-673200t_2t_4^5-3351936Zt_3t_4^5 \nonumber \\
&\left. +204800Z^3t_4^6+4163712t_3t_4^6-326400Z^2t_4^7+2147328Zt_4^8-4627456t_4^9\right), \\
g^{22}(t)=& \, -\frac{32}{9375}\left(-1250Z^3t_2+6300t_2t_3-56448Zt_3^2-4375Z^2t_2t_4+30800Z^3t_3t_4 \right. \nonumber \\
&+91728t_3^2t_4-5250Zt_2t_4^2+56840Z^2t_3t_4^2+3325t_2t_4^3+159936Zt_3t_4^3 \nonumber \\
&\left. -17200Z^3t_4^4-369600t_3t_4^4-9240Z^2t_4^5-46368Zt_4^6+80672t_4^7\right),
\end{align*}
\begin{align*}
g^{13}(t)=& \, \frac{1}{9800}\left(3125t_2^2-22400Zt_2t_3+75264Z^2t_3^2-179200t_2t_3t_4+200704Zt_3^2t_4 \right. \nonumber \\
&+28000Z^2t_2t_4^2-125440Z^3t_3t_4^2-551936t_3^2t_4^2-19200Zt_2t_4^3-261632Z^2t_3t_4^3 \nonumber \\
&+225600t_2t_4^4+215040Zt_3t_4^4+125440Z^3t_4^5+2867200t_3t_4^5-118272Z^2t_4^6 \nonumber \\
&\left. +36864Zt_4^7-604160t_4^8\right), \\
g^{23}(t)=& \, \frac{1}{175}\left(175t_1-125Z^2t_2+560Z^3t_3-300Zt_2t_4+2128Z^2t_3t_4-1200t_2t_4^2 \right. \nonumber \\
&\left. +2688Zt_3t_4^2-560Z^3t_4^3-4928t_3t_4^3-768Z^2t_4^4+576Zt_4^5+6400t_4^6\right), \\
g^{33}(t)=& \, \frac{1}{1568}\left(250Zt_2-840Z^2t_3+625t_2t_4-2240Zt_3t_4-8960t_3t_4^2+840Z^2t_4^3 \right. \nonumber \\
&\left. -480Zt_4^4+4512t_4^5\right), \\
g^{14}(t)=& \, t_1, \hspace{5mm} g^{24}(t)=\frac{2}{3}t_2, \hspace{5mm} g^{34}(t)=\frac{1}{2}t_3, \hspace{5mm} g^{44}(t)=\frac{1}{6}t_4.
\end{align*}
We have that $\mathrm{deg} \, t_1(x)=12, \, \mathrm{deg} \, t_2(x)=8, \, \mathrm{deg} \, t_3(x)=6, \, \mathrm{deg} \, t_4(x)=2$ and $\mathrm{deg} \, Z(x)=2.$

\begin{theorem} \label{theoremYH43}
We have the following relations
\begin{align*}
y_1=& \, \frac{32768}{45}\left(273437500 Z t_1 t_2^2 + 166015625 Z^3 t_2^3 - 
 7717500000 t_1^2 t_3- 1684375000 Z^2 t_1 t_2 t_3 \right. \nonumber \\
&+ 
 2810937500 t_2^3 t_3 + 3430000000 Z^3 t_1 t_3^2+ 
 10143000000 Z t_2^2 t_3^2 + 44562560000 t_1 t_3^3  \nonumber \\
&+ 
 13088880000 Z^2 t_2 t_3^3 - 98467891200 Z^3 t_3^4 - 
 1066905133056 t_3^5 - 3691406250 t_1 t_2^2 t_4 \nonumber \\
&- 
 615234375 Z^2 t_2^3 t_4 - 4900000000 Z t_1 t_2 t_3 t_4 + 
 656250000 Z^3 t_2^2 t_3 t_4 \nonumber \\
&+ 15092000000 Z^2 t_1 t_3^2 t_4 + 
 17272500000 t_2^2 t_3^2 t_4 - 75075840000 Z t_2 t_3^3 t_4 \nonumber \\
&- 
 117276364800 Z^2 t_3^4 t_4 + 1289062500 Z t_2^3 t_4^2 - 
 85995000000 t_1 t_2 t_3 t_4^2 \nonumber \\
&- 
 4633125000 Z^2 t_2^2 t_3 t_4^2 + 
 21952000000 Z t_1 t_3^2 t_4^2 + 
 74499600000 Z^3 t_2 t_3^2 t_4^2 \nonumber \\
&+ 
 1219784832000 t_2 t_3^3 t_4^2 + 2152574484480 Z t_3^4 t_4^2 + 
 7717500000 t_1^2 t_4^3 \nonumber \\
&+ 1684375000 Z^2 t_1 t_2 t_4^3 - 
 7557031250 t_2^3 t_4^3 - 6860000000 Z^3 t_1 t_3 t_4^3 \nonumber \\
&- 
 17206000000 Z t_2^2 t_3 t_4^3 - 
 1196603520000 t_1 t_3^2 t_4^3 - 
 40481840000 Z^2 t_2 t_3^2 t_4^3 \nonumber \\
&+ 27536588800 Z^3 t_3^3 t_4^3 - 
 2854462464000 t_3^4 t_4^3 - 1050000000 Z t_1 t_2 t_4^4 \nonumber \\
&- 
 6075000000 Z^3 t_2^2 t_4^4 - 11858000000 Z^2 t_1 t_3 t_4^4 - 
 237331500000 t_2^2 t_3 t_4^4 \nonumber \\
&- 
 361141760000 Z t_2 t_3^2 t_4^4 + 
 1152093644800 Z^2 t_3^3 t_4^4 + 166320000000 t_1 t_2 t_4^5 \nonumber \\
&+ 
 24714375000 Z^2 t_2^2 t_4^5 + 9408000000 Z t_1 t_3 t_4^5 - 
 163279200000 Z^3 t_2 t_3 t_4^5 \nonumber \\
&- 
 4338329856000 t_2 t_3^2 t_4^5 - 5691700510720 Z t_3^3 t_4^5 + 
 3430000000 Z^3 t_1 t_4^6 \nonumber \\
&- 59062000000 Z t_2^2 t_4^6 + 
 4552719360000 t_1 t_3 t_4^6 + 290661840000 Z^2 t_2 t_3 t_4^6 \nonumber \\
&- 
 604041267200 Z^3 t_3^2 t_4^6 - 1144791531520 t_3^3 t_4^6 - 
 3234000000 Z^2 t_1 t_4^7 \nonumber \\
&+ 699646500000 t_2^2 t_4^7 + 
 1311466240000 Z t_2 t_3 t_4^7 - 4066807449600 Z^2 t_3^2 t_4^7 \nonumber \\
&+ 
 1008000000 Z t_1 t_4^8 + 147735600000 Z^3 t_2 t_4^8 + 
 15995339136000 t_2 t_3 t_4^8 \nonumber \\
&+ 7407908372480 Z t_3^2 t_4^8 - 
 3837646400000 t_1 t_4^9 - 481752880000 Z^2 t_2 t_4^9 \nonumber \\
&+ 
 1597027532800 Z^3 t_3 t_4^9 + 130060834283520 t_3^2 t_4^9 + 
 105855360000 Z t_2 t_4^{10} \nonumber \\
&+ 2986145792000 Z^2 t_3 t_4^{10} - 
 21627170112000 t_2 t_4^{11} - 10675438878720 Z t_3 t_4^{11} \nonumber \\
&- 
 1135868723200 Z^3 t_4^{12} - 398693684838400 t_3 t_4^{12} + 
 838213017600 Z^2 t_4^{13} \nonumber \\
&\left. + 2299551252480 Z t_4^{14} + 
 318244776083456 t_4^{15}\right),
\end{align*}
\begin{align*}
y_2=& \, \frac{256}{3}\left(656250 t_1 t_2 + 109375 Z^2 t_2^2 - 910000 Z^3 t_2 t_3 - 
 16503200 t_2 t_3^2 \right. \nonumber \\
&- 18966528 Z t_3^3 - 500000 Z t_2^2 t_4 + 
 22344000 t_1 t_3 t_4 + 1120000 Z^2 t_2 t_3 t_4 \nonumber \\
&+ 
 1881600 Z^3 t_3^2 t_4 + 66382848 t_3^3 t_4 + 
 3093750 t_2^2 t_4^2 + 8960000 Z t_2 t_3 t_4^2 \nonumber \\
&- 
 18816000 Z^2 t_3^2 t_4^2 + 910000 Z^3 t_2 t_4^3 + 
 56022400 t_2 t_3 t_4^3 + 16758784 Z t_3^2 t_4^3 \nonumber \\
&- 
 29484000 t_1 t_4^4 - 3500000 Z^2 t_2 t_4^4 + 
 6137600 Z^3 t_3 t_4^4 + 298923520 t_3^2 t_4^4 \nonumber \\
&+ 
 1920000 Z t_2 t_4^5 + 25446400 Z^2 t_3 t_4^5 - 
 164559200 t_2 t_4^6 - 74102784 Z t_3 t_4^6 \nonumber \\
&- 
 8019200 Z^3 t_4^7 - 3012660224 t_3 t_4^7 + 6316800 Z^2 t_4^8 + 
 17123328 Z t_4^9 \nonumber \\
&\left. + 3641568256 t_4^{10}\right), \\
y_3=& \, 64\left(875 t_1 + 8624 t_3^2 + 4125 t_2 t_4^2 + 66528 t_3 t_4^3 - 
 196992 t_4^6\right), \\
y_4=& \, 24t_4.
\end{align*}
\end{theorem}

\begin{proposition}
The derivatives $\frac{\partial y_i}{\partial t_j} \in \C[t_1, t_2, t_3, t_4, Z].$
\end{proposition}
\noindent Proof is similar to the one for Proposition \ref{DerivsProp}.

\subsection{$H_4(4)$ example}
The prepotential for $H_4(4)$ is
\begin{align*}
F(t)=& \, -\frac{25}{33}Z^{11}+\frac{1}{2}Z^{10} (-4 t_3 - 5 t_4)-\frac{125}{9}Z^9 (t_3 + t_4) (t_3 + 3 t_4)-5 Z^8 (8 t_3^3 + 57 t_3^2 t_4 \\
&+ 124 t_3 t_4^2 + 95 t_4^3)-\frac{35}{3} Z^6 (t_3 + t_4) (t_3 + 3 t_4) (12 t_3^3 + 123 t_3^2 t_4 + 
    336 t_3 t_4^2 + 305 t_4^3) \\
&- 10 Z^7 (5 t_3^4 + 64 t_3^3 t_4 + 236 t_3^2 t_4^2 + 
    344 t_3 t_4^3 + 175 t_4^4)+ 
 Z^5 (135 t_3^6 - 840 t_3^5 t_4 \\
&- 17130 t_3^4 t_4^2 - 
    80920 t_3^3 t_4^3 - 174765 t_3^2 t_4^4 - 181952 t_3 t_4^5 - 
    74420 t_4^6) \\
&+ 5 Z t_4 (t_3 + t_4)^2 (t_3 + 3 t_4)^2 (180 t_3^5 + 
    2115 t_3^4 t_4 + 11760 t_3^3 t_4^2 + 30270 t_3^2 t_4^3 \\
&+ 
    32316 t_3 t_4^4 + 9035 t_4^5)+\frac{5}{3}Z^3 (t_3 + t_4) (t_3 + 3 t_4) (-45 t_3^6 + 1260 t_3^5 t_4 + 
    13185 t_3^4 t_4^2 \\
&+ 42360 t_3^3 t_4^3 + 58125 t_3^2 t_4^4 + 
    37788 t_3 t_4^5 + 15815 t_4^6) - 100 Z^4 (-6 t_3^7 - 87 t_3^6 t_4 \\
&- 450 t_3^5 t_4^2 - 
    1015 t_3^4 t_4^3 - 650 t_3^3 t_4^4 + 1283 t_3^2 t_4^5 + 
    2450 t_3 t_4^6 + 1195 t_4^7) \\
&-\frac{5}{6}Z^2 (540 t_3^9 + 9315 t_3^8 t_4 + 73440 t_3^7 t_4^2 + 
    356940 t_3^6 t_4^3 + 1252440 t_3^5 t_4^4 \\
&+3321450 t_3^4 t_4^5 + 6239424 t_3^3 t_4^6 + 
    7356636 t_3^2 t_4^7 + 4563564 t_3 t_4^8 + 994315 t_4^9) \\
&+\frac{1}{198}(198 t_1 t_2 t_3 + 99 t_1^2 t_4 - 4455 t_3^{10} t_4 - 
  178200 t_3^9 t_4^2 - 2569050 t_3^8 t_4^3 \\
&- 21740400 t_3^7 t_4^4 - 120561210 t_3^6 t_4^5 - 
  458678880 t_3^5 t_4^6 - 1191552120 t_3^4 t_4^7 \\
&- 2004227280 t_3^3 t_4^8 - 1955070535 t_3^2 t_4^9 - 
  858257224 t_3 t_4^{10} - 45669270 t_4^{11}),
\end{align*}
where
\begin{align*}
P(t_2, t_3, t_4, Z):=& \, Z^5 - t_2 + 15 t_3^4 t_4 + 120 t_3^3 t_4^2 + 530 t_3^2 t_4^3 + 
 1160 t_3 t_4^4 + 843 t_4^5 \nonumber \\
&+ 
 10 Z^3 (t_3 + t_4) (t_3 + 3 t_4) - 
 15 Z (t_3 + t_4)^2 (t_3 + 3 t_4)^2 \nonumber \\
&+ 
 20 Z^2 t_4 (3 t_3^2 + 12 t_3 t_4 + 13 t_4^2)=0.
\end{align*}
The Euler vector field is
\begin{equation*}
E(t)=t_1\partial_{t_1}+t_2\partial_{t_2}+\frac{1}{5}t_3\partial_{t_3}+\frac{1}{5}t_4\partial_{t_4},
\end{equation*}
the unity vector field is $e(t)=\partial_{t_1},$ and the charge is $d=\frac{4}{5}.$ The intersection form (\ref{gDef}) is then given by
\begin{align*}
g^{11}(t)=& \, -15\left(-2 Z^4 t_2 - 24 Z^3 t_2 t_3 + 116 Z^2 t_2 t_3^2 + 
 192 Z t_2 t_3^3 - 956 t_2 t_3^4 - 3360 Z^4 t_3^5 \right. \nonumber \\
&+ 
 12560 Z^3 t_3^6 + 2880 Z^2 t_3^7 - 14826 Z t_3^8 + 108 t_3^9 - 
 78 Z^3 t_2 t_4 + 368 Z^2 t_2 t_3 t_4 \nonumber \\
&+ 
 1584 Z t_2 t_3^2 t_4 - 4096 t_2 t_3^3 t_4 - 
 48360 Z^4 t_3^4 t_4 + 106920 Z^3 t_3^5 t_4 + 
 104472 Z^2 t_3^6 t_4 \nonumber \\
&- 194592 Z t_3^7 t_4 + 17013 t_3^8 t_4 + 
 276 Z^2 t_2 t_4^2 + 2976 Z t_2 t_3 t_4^2 + 
 2824 t_2 t_3^2 t_4^2 \nonumber \\
&- 238560 Z^4 t_3^3 t_4^2 + 
 221106 Z^3 t_3^4 t_4^2 + 806976 Z^2 t_3^5 t_4^2 - 
 962208 Z t_3^6 t_4^2 \nonumber \\
&+ 199488 t_3^7 t_4^2 + 1048 Z t_2 t_4^3 + 
 27264 t_2 t_3 t_4^3 - 526312 Z^4 t_3^2 t_4^3 - 
 417856 Z^3 t_3^3 t_4^3 \nonumber \\
&+ 2293152 Z^2 t_3^4 t_4^3 - 
 1907424 Z t_3^5 t_4^3 + 1003828 t_3^6 t_4^3 + 
 24108 t_2 t_4^4 \nonumber \\
&- 527680 Z^4 t_3 t_4^4 - 
 2204820 Z^3 t_3^2 t_4^4 + 1151040 Z^2 t_3^3 t_4^4 + 
 655860 Z t_3^4 t_4^4 \nonumber \\
&+ 2033688 t_3^5 t_4^4 - 193800 Z^4 t_4^5 - 
 2871240 Z^3 t_3 t_4^5 - 6809880 Z^2 t_3^2 t_4^5 \nonumber \\
&+ 
 10362720 Z t_3^3 t_4^5 - 3485610 t_3^4 t_4^5 - 
 1183430 Z^3 t_4^6 - 13507200 Z^2 t_3 t_4^6 \nonumber \\
&+ 
 20695840 Z t_3^2 t_4^6 - 29726240 t_3^3 t_4^6 - 
 7758720 Z^2 t_4^7 + 19063840 Z t_3 t_4^7 \nonumber \\
&-\left. 
 68995740 t_3^2 t_4^7 + 7219830 Z t_4^8 - 73570020 t_3 t_4^8 - 
 30991507 t_4^9\right), \\
g^{12}(t)=& \, -5\left(-4 Z^4 t_2 - 18 Z^3 t_2 t_3 + 184 Z^2 t_2 t_3^2 + 
 144 Z t_2 t_3^3 - 1792 t_2 t_3^4 \right. \nonumber \\
&- 2520 Z^4 t_3^5+ 
 21760 Z^3 t_3^6 + 2160 Z^2 t_3^7 - 26880 Z t_3^8 + 81 t_3^9 \nonumber \\
& - 
 72 Z^3 t_2 t_4 + 664 Z^2 t_2 t_3 t_4+ 
 1968 Z t_2 t_3^2 t_4 - 11672 t_2 t_3^3 t_4 \nonumber \\
&- 
 49920 Z^4 t_3^4 t_4 + 228270 Z^3 t_3^5 t_4 + 
 148320 Z^2 t_3^6 t_4 - 398112 Z t_3^7 t_4  \nonumber \\
&+ 29796 t_3^8 t_4 + 
 552 Z^2 t_2 t_4^2 + 5352 Z t_2 t_3 t_4^2- 
 18832 t_2 t_3^2 t_4^2 - 288120 Z^4 t_3^3 t_4^2  \nonumber \\
&+ 
 813000 Z^3 t_3^4 t_4^2 + 1444824 Z^2 t_3^5 t_4^2- 
 2417184 Z t_3^6 t_4^2 + 422196 t_3^7 t_4^2  \nonumber \\
&+ 
 3536 Z t_2 t_4^3 + 5448 t_2 t_3 t_4^3 - 
 708320 Z^4 t_3^2 t_4^3+ 915172 Z^3 t_3^3 t_4^3  \nonumber \\
&+ 
 5680992 Z^2 t_3^4 t_4^3 - 7655496 Z t_3^5 t_4^3 + 
 2785376 t_3^6 t_4^3+ 20448 t_2 t_4^4 \nonumber \\
&- 776504 Z^4 t_3 t_4^4 - 
 881952 Z^3 t_3^2 t_4^4 + 10163184 Z^2 t_3^3 t_4^4 - 
 13244448 Z t_3^4 t_4^4  \nonumber \\
&+ 10650846 t_3^5 t_4^4 - 
 307680 Z^4 t_4^5 - 2577210 Z^3 t_3 t_4^5+ 
 6008160 Z^2 t_3^2 t_4^5 \nonumber \\
&- 11434800 Z t_3^3 t_4^5 + 
 23025000 t_3^4 t_4^5 - 1401160 Z^3 t_4^6- 
 4244760 Z^2 t_3 t_4^6  \nonumber \\
&- 1991200 Z t_3^2 t_4^6 + 
 20880740 t_3^3 t_4^6 - 5101920 Z^2 t_4^7+
 4450040 Z t_3 t_4^7 \nonumber \\
&-\left. 12170880 t_3^2 t_4^7 + 2911200 Z t_4^8 - 
 39000375 t_3 t_4^8 - 21754220 t_4^9\right), \\
g^{22}(t)=& \, 5\left(2 Z^4 t_2 - 92 Z^2 t_2 t_3^2 + 896 t_2 t_3^4 - 
 10880 Z^3 t_3^6 + 13440 Z t_3^8 + 18 Z^3 t_2 t_4 \right. \nonumber \\
&- 
 368 Z^2 t_2 t_3 t_4 - 552 Z t_2 t_3^2 t_4 + 
 7168 t_2 t_3^3 t_4 + 12360 Z^4 t_3^4 t_4 - 
 130560 Z^3 t_3^5 t_4 \nonumber \\
&- 59040 Z^2 t_3^6 t_4 + 
 215040 Z t_3^7 t_4 - 14169 t_3^8 t_4 - 348 Z^2 t_2 t_4^2 - 
 2208 Z t_2 t_3 t_4^2 \nonumber \\
&+ 17408 t_2 t_3^2 t_4^2 + 
 98880 Z^4 t_3^3 t_4^2 - 570750 Z^3 t_3^4 t_4^2 - 
 708480 Z^2 t_3^5 t_4^2 \nonumber \\
&+ 1432368 Z t_3^6 t_4^2 - 
 226704 t_3^7 t_4^2 - 1984 Z t_2 t_4^3 + 12288 t_2 t_3 t_4^3 + 
 284680 Z^4 t_3^2 t_4^3 \nonumber \\
&- 1084400 Z^3 t_3^3 t_4^3 - 
 3305976 Z^2 t_3^4 t_4^3 + 5146176 Z t_3^5 t_4^3 - 
 1665604 t_3^6 t_4^3 \nonumber \\
&- 1512 t_2 t_4^4 + 347680 Z^4 t_3 t_4^4 - 
 691908 Z^3 t_3^2 t_4^4 - 7555008 Z^2 t_3^3 t_4^4 \nonumber \\
&+ 
 10855464 Z t_3^4 t_4^4 - 7291824 t_3^5 t_4^4 + 
 150456 Z^4 t_4^5 + 337008 Z^3 t_3 t_4^5 \nonumber \\
&- 
 8461536 Z^2 t_3^2 t_4^5 + 13837632 Z t_3^3 t_4^5 - 
 19824414 t_3^4 t_4^5 + 416810 Z^3 t_4^6 \nonumber \\
&- 
 3634560 Z^2 t_3 t_4^6 + 10628480 Z t_3^2 t_4^6 - 
 31855600 t_3^3 t_4^6 + 143640 Z^2 t_4^7 \nonumber \\
&+\left. 4582400 Z t_3 t_4^7 - 
 26339220 t_3^2 t_4^7 + 771720 Z t_4^8 - 6866640 t_3 t_4^8 + 
 1889215 t_4^9\right),
\end{align*}
\begin{align*}
g^{13}(t)=& \, -5\left(t_2 - 4 Z^4 t_3 - 8 Z^3 t_3^2 + 24 Z^2 t_3^3 - 6 Z t_3^4 - 
 6 Z^4 t_4 - 64 Z^3 t_3 t_4 + 132 Z^2 t_3^2 t_4 \right. \nonumber \\
&- 
 96 Z t_3^3 t_4 + 60 t_3^4 t_4 - 104 Z^3 t_4^2 + 
 216 Z^2 t_3 t_4^2 - 636 Z t_3^2 t_4^2 + 720 t_3^3 t_4^2 + 
 108 Z^2 t_4^3 \nonumber \\
&-\left. 1856 Z t_3 t_4^3 + 2240 t_3^2 t_4^3 - 
 1686 Z t_4^4 + 1600 t_3 t_4^4 - 1444 t_4^5\right), \\
g^{23}(t)=& \, t_1 - 4 t_2 + 10 Z^4 t_3 - 60 Z^2 t_3^3 + 20 Z^4 t_4 + 
 80 Z^3 t_3 t_4 - 360 Z^2 t_3^2 t_4 + 120 Z t_3^3 t_4 \nonumber \\
&+ 
 160 Z^3 t_4^2 - 660 Z^2 t_3 t_4^2 + 720 Z t_3^2 t_4^2 - 
 600 t_3^3 t_4^2 - 360 Z^2 t_4^3 + 2120 Z t_3 t_4^3 \nonumber \\
&- 
 3600 t_3^2 t_4^3 + 2320 Z t_4^4 - 5600 t_3 t_4^4 - 1600 t_4^5, \\
g^{33}(t)=& \, -\frac{1}{5}\left(2Z+4t_3+5t_4\right), \\
g^{14}(t)=& \, t_1, \hspace{5mm} g^{24}(t)=t_2, \hspace{5mm} g^{34}(t)=\frac{1}{5}t_3, \hspace{5mm} g^{44}(t)=\frac{1}{5}t_4.
\end{align*}
We have that $\mathrm{deg} \, t_1(x)=10, \, \mathrm{deg} \, t_2(x)=10, \, \mathrm{deg} \, t_3(x)=2, \, \mathrm{deg} \, t_4(x)=2$ and $\mathrm{deg} \, Z(x)=2.$

\begin{theorem} \label{theoremYH44}
We have the following relations
\begin{align*}
y_1=& \, -\frac{2^{19}5^9}{3^{13}7}\left(42 t_1^2 t_2 - 168 t_1 t_2^2 + 168 t_2^3 - 
 210 Z^4 t_1 t_2 t_3 + 420 Z^4 t_2^2 t_3+ 
 1050 Z^3 t_2^2 t_3^2 \right. \nonumber \\
&+ 2100 Z^2 t_1 t_2 t_3^3 - 
 4200 Z^2 t_2^2 t_3^3 + 27300 Z t_2^2 t_3^4 + 
 89250 t_1 t_2 t_3^5 - 178500 t_2^2 t_3^5 \nonumber \\
&+ 
 228620 Z^4 t_2 t_3^6 + 114240 Z^3 t_1 t_3^7 - 
 228480 Z^3 t_2 t_3^7 + 2465680 Z^2 t_2 t_3^8 \nonumber \\
&- 
 141120 Z t_1 t_3^9 + 282240 Z t_2 t_3^9 - 5667634 t_2 t_3^{10} - 
 13482560 Z^3 t_3^{12} \nonumber \\
&+ 42611520 Z t_3^{14} - 420 Z^4 t_1 t_2 t_4 + 
 1050 Z^4 t_2^2 t_4 - 1890 Z^3 t_1 t_2 t_3 t_4 + 
 7980 Z^3 t_2^2 t_3 t_4 \nonumber \\
&+ 12600 Z^2 t_1 t_2 t_3^2 t_4 - 
 20790 Z^2 t_2^2 t_3^2 t_4 - 104580 Z t_1 t_2 t_3^3 t_4 + 
 427560 Z t_2^2 t_3^3 t_4 \nonumber \\
&- 164430 t_1^2 t_3^4 t_4 + 
 1550220 t_1 t_2 t_3^4 t_4 - 2334570 t_2^2 t_3^4 t_4 - 
 46620 Z^4 t_1 t_3^5 t_4 \nonumber \\
&+ 2836680 Z^4 t_2 t_3^5 t_4 + 
 1599360 Z^3 t_1 t_3^6 t_4 - 2240700 Z^3 t_2 t_3^6 t_4 - 
 15120 Z^2 t_1 t_3^7 t_4 \nonumber \\
&+ 39481120 Z^2 t_2 t_3^7 t_4 - 
 2540160 Z t_1 t_3^8 t_4 + 38725680 Z t_2 t_3^8 t_4 + 
 28513800 t_1 t_3^9 t_4 \nonumber \\
&- 170380280 t_2 t_3^9 t_4 + 
 16561860 Z^4 t_3^{10} t_4 - 323581440 Z^3 t_3^{11} t_4 - 
 64723680 Z^2 t_3^{12} t_4 \nonumber \\
&+ 1193122560 Z t_3^{13} t_4 - 
 29292690 t_3^{14} t_4 - 3780 Z^3 t_1 t_2 t_4^2 + 
 8820 Z^3 t_2^2 t_4^2 \nonumber \\
&+ 23100 Z^2 t_1 t_2 t_3 t_4^2 - 
 28560 Z^2 t_2^2 t_3 t_4^2 - 627480 Z t_1 t_2 t_3^2 t_4^2 + 
 1892100 Z t_2^2 t_3^2 t_4^2 \nonumber \\
&- 1315440 t_1^2 t_3^3 t_4^2 + 
 8585220 t_1 t_2 t_3^3 t_4^2 - 11043480 t_2^2 t_3^3 t_4^2 - 
 466200 Z^4 t_1 t_3^4 t_4^2 \nonumber \\
&+ 15817200 Z^4 t_2 t_3^4 t_4^2 + 
 9705150 Z^3 t_1 t_3^5 t_4^2 - 7914060 Z^3 t_2 t_3^5 t_4^2 - 
 211680 Z^2 t_1 t_3^6 t_4^2 \nonumber \\
&+ 289891980 Z^2 t_2 t_3^6 t_4^2 - 
 17575740 Z t_1 t_3^7 t_4^2 + 573477240 Z t_2 t_3^7 t_4^2 + 
 513248400 t_1 t_3^8 t_4^2 \nonumber \\
&- 1955803360 t_2 t_3^8 t_4^2 + 
 331237200 Z^4 t_3^9 t_4^2 - 3429078870 Z^3 t_3^{10} t_4^2 - 
 1553368320 Z^2 t_3^{11} t_4^2 \nonumber \\
&+ 14721448140 Z t_3^{12} t_4^2 - 
 820195320 t_3^{13} t_4^2 + 12600 Z^2 t_1 t_2 t_4^3 - 
 3290 Z^2 t_2^2 t_4^3 \nonumber \\
&- 1244460 Z t_1 t_2 t_3 t_4^3 + 
 3290280 Z t_2^2 t_3 t_4^3 - 3844260 t_1^2 t_3^2 t_4^3 - 
 450 Z^5 t_2 t_3^2 t_4^3 \nonumber \\
&+ 21037800 t_1 t_2 t_3^2 t_4^3 - 
 24410930 t_2^2 t_3^2 t_4^3 - 1773240 Z^4 t_1 t_3^3 t_4^3 + 
 49466480 Z^4 t_2 t_3^3 t_4^3 \nonumber \\
&+ 225 Z^8 t_3^4 t_4^3 + 
 33077100 Z^3 t_1 t_3^4 t_4^3 - 11938395 Z^3 t_2 t_3^4 t_4^3 - 
 3845520 Z^2 t_1 t_3^5 t_4^3 \nonumber \\
&+ 1272065200 Z^2 t_2 t_3^5 t_4^3 - 
 2700 Z^6 t_3^6 t_4^3 - 56395080 Z t_1 t_3^6 t_4^3 + 
 3895595560 Z t_2 t_3^6 t_4^3 \nonumber \\
&+ 4000050600 t_1 t_3^7 t_4^3 - 
 11987148880 t_2 t_3^7 t_4^3 + 2939165645 Z^4 t_3^8 t_4^3 \nonumber \\
&- 
 21122966200 Z^3 t_3^9 t_4^3 - 17379662880 Z^2 t_3^{10} t_4^3 + 
 105145262880 Z t_3^{11} t_4^3 \nonumber \\
&- 6999178480 t_3^{12} t_4^3 - 
 815640 Z t_1 t_2 t_4^4 + 2080680 Z t_2^2 t_4^4 - 
 4853520 t_1^2 t_3 t_4^4 \nonumber \\
&- 1800 Z^5 t_2 t_3 t_4^4 + 
 23828490 t_1 t_2 t_3 t_4^4 - 26013980 t_2^2 t_3 t_4^4 - 
 3180240 Z^4 t_1 t_3^2 t_4^4 \nonumber \\
&+ 91783860 Z^4 t_2 t_3^2 t_4^4 + 
 1800 Z^8 t_3^3 t_4^4 + 69522180 Z^3 t_1 t_3^3 t_4^4 - 
 11884320 Z^3 t_2 t_3^3 t_4^4 \nonumber \\
&+ 5400 Z^7 t_3^4 t_4^4 - 
 29988000 Z^2 t_1 t_3^4 t_4^4 + 
 3617141020 Z^2 t_2 t_3^4 t_4^4 - 32400 Z^6 t_3^5 t_4^4 \nonumber
\end{align*}
\begin{align*}
&- 
 38118780 Z t_1 t_3^5 t_4^4 + 15323659800 Z t_2 t_3^5 t_4^4 - 
 25650 Z^5 t_3^6 t_4^4 + 17678161200 t_1 t_3^6 t_4^4 \nonumber \\
&- 
 44692143980 t_2 t_3^6 t_4^4 + 15227879120 Z^4 t_3^7 t_4^4 - 
 83502915015 Z^3 t_3^8 t_4^4 \nonumber \\
&- 119765904000 Z^2 t_3^9 t_4^4 + 
 469572887280 Z t_3^{10} t_4^4 + 2620343040 t_3^{11} t_4^4 - 
 2164806 t_1^2 t_4^5 \nonumber \\
&- 1350 Z^5 t_2 t_4^5 + 
 10008684 t_1 t_2 t_4^5 - 11210172 t_2^2 t_4^5 - 
 2653140 Z^4 t_1 t_3 t_4^5 \nonumber \\
&+ 96693240 Z^4 t_2 t_3 t_4^5 + 
 4950 Z^8 t_3^2 t_4^5 + 92617560 Z^3 t_1 t_3^2 t_4^5 - 
 43361430 Z^3 t_2 t_3^2 t_4^5 \nonumber \\
&+ 43200 Z^7 t_3^3 t_4^5 - 
 113047200 Z^2 t_1 t_3^3 t_4^5 + 
 6749433120 Z^2 t_2 t_3^3 t_4^5 - 121500 Z^6 t_3^4 t_4^5 \nonumber \\
&+ 
 357293160 Z t_1 t_3^4 t_4^5 + 38092662600 Z t_2 t_3^4 t_4^5 - 
 307800 Z^5 t_3^5 t_4^5 + 49036806000 t_1 t_3^5 t_4^5 \nonumber \\
&- 
 108705602416 t_2 t_3^5 t_4^5 + 50844400500 Z^4 t_3^6 t_4^5 - 
 219447220080 Z^3 t_3^7 t_4^5 \nonumber \\
&- 570111869640 Z^2 t_3^8 t_4^5 + 
 1249790969280 Z t_3^9 t_4^5 + 479633884512 t_3^{10} t_4^5 - 
 803880 Z^4 t_1 t_4^6 \nonumber \\
&+ 45585960 Z^4 t_2 t_4^6 + 
 5400 Z^8 t_3 t_4^6 + 74024790 Z^3 t_1 t_3 t_4^6 - 
 107328900 Z^3 t_2 t_3 t_4^6 \nonumber \\
&+ 126000 Z^7 t_3^2 t_4^6 - 
 225570240 Z^2 t_1 t_3^2 t_4^6 + 
 7964410020 Z^2 t_2 t_3^2 t_4^6 - 108000 Z^6 t_3^3 t_4^6 \nonumber \\
&+ 
 1399923420 Z t_1 t_3^3 t_4^6 + 61840178440 Z t_2 t_3^3 t_4^6 - 
 1415250 Z^5 t_3^4 t_4^6 + 89821989600 t_1 t_3^4 t_4^6 \nonumber \\
&- 
 183359111760 t_2 t_3^4 t_4^6 + 113517812720 Z^4 t_3^5 t_4^6 - 
 379424004920 Z^3 t_3^6 t_4^6 \nonumber \\
&- 1998548319360 Z^2 t_3^7 t_4^6 + 
 1114813850860 Z t_3^8 t_4^6 + 4204075665280 t_3^9 t_4^6 + 
 2025 Z^8 t_4^7 \nonumber \\
&+ 27351660 Z^3 t_1 t_4^7 - 
 87125115 Z^3 t_2 t_4^7 + 158400 Z^7 t_3 t_4^7 - 
 231759360 Z^2 t_1 t_3 t_4^7 \nonumber \\
&+ 5344717200 Z^2 t_2 t_3 t_4^7 + 
 337500 Z^6 t_3^2 t_4^7 + 2400609960 Z t_1 t_3^2 t_4^7 + 
 65461451880 Z t_2 t_3^2 t_4^7 \nonumber \\
&- 3114000 Z^5 t_3^3 t_4^7 + 
 113150026920 t_1 t_3^3 t_4^7 - 234714455760 t_2 t_3^3 t_4^7 + 
 169129871150 Z^4 t_3^4 t_4^7 \nonumber \\
&- 396171632160 Z^3 t_3^5 t_4^7 - 
 5361270379440 Z^2 t_3^6 t_4^7 - 5723934259520 Z t_3^7 t_4^7 \nonumber \\
&+ 
 20067909135750 t_3^8 t_4^7 + 70200 Z^7 t_4^8 - 
 97251840 Z^2 t_1 t_4^8 + 1534267620 Z^2 t_2 t_4^8 \nonumber \\
&+ 
 831600 Z^6 t_3 t_4^8 + 2057206620 Z t_1 t_3 t_4^8 + 
 42875701800 Z t_2 t_3 t_4^8 - 3307950 Z^5 t_3^2 t_4^8 \nonumber \\
&+ 
 103042663920 t_1 t_3^2 t_4^8 - 249632019090 t_2 t_3^2 t_4^8 + 
 160462051440 Z^4 t_3^3 t_4^8 \nonumber \\
&- 159576392790 Z^3 t_3^4 t_4^8 - 
 11200275624000 Z^2 t_3^5 t_4^8 - 29740325063280 Z t_3^6 t_4^8 \nonumber \\
&+ 
 62134558132320 t_3^7 t_4^8 + 535500 Z^6 t_4^9 + 
 712487160 Z t_1 t_4^9 + 13938177080 Z t_2 t_4^9 \nonumber \\
&- 
 1452600 Z^5 t_3 t_4^9 + 68111856120 t_1 t_3 t_4^9 - 
 203506674360 t_2 t_3 t_4^9 + 84058972000 Z^4 t_3^2 t_4^9 \nonumber \\
&+ 
 151452210000 Z^3 t_3^3 t_4^9 - 18131069931360 Z^2 t_3^4 t_4^9 - 
 75535900424000 Z t_3^5 t_4^9 \nonumber \\
&+ 130861824336350 t_3^6 t_4^9 - 
 125550 Z^5 t_4^{10} + 25090865520 t_1 t_4^{10} - 
 86556812352 t_2 t_4^{10} \nonumber \\
&+ 12789015040 Z^4 t_3 t_4^{10} + 
 241285828350 Z^3 t_3^2 t_4^{10} - 
 22060526165760 Z^2 t_3^3 t_4^{10} \nonumber \\
&- 
 125256834868940 Z t_3^4 t_4^{10} + 186992066838888 t_3^5 t_4^{10} - 
 5145601875 Z^4 t_4^{11} \nonumber \\
&+ 114867728760 Z^3 t_3 t_4^{11} - 
 18961639423440 Z^2 t_3^2 t_4^{11} - 
 142926704208480 Z t_3^3 t_4^{11} \nonumber \\
&+ 171096028587180 t_3^4 t_4^{11} + 
 11533193965 Z^3 t_4^{12} - 10237991411520 Z^2 t_3 t_4^{12} \nonumber \\
&- 
 110581088197120 Z t_3^2 t_4^{12} + 78681803681760 t_3^3 t_4^{12} - 
 2609003447640 Z^2 t_4^{13} \nonumber \\
&- 53312566785280 Z t_3 t_4^{13} - 
 16250437730220 t_3^2 t_4^{13} - 12255008872620 Z t_4^{14} \nonumber \\
&\left. -
 46426601177520 t_3 t_4^{14} - 22667569904962 t_4^{15}\right),
\end{align*}
\begin{align*}
y_2=& \, -\frac{2^{12}5^6}{3^8}\left(t_1^2 - 4 t_1 t_2 + 3 t_2^2 - 45 Z^3 t_2 t_3^2 - 
 810 Z t_2 t_3^4 - 1026 t_1 t_3^5 + 2052 t_2 t_3^5 \right. \nonumber \\
&- 
 405 Z^4 t_3^6 - 4050 Z^2 t_3^8 + 10044 t_3^{10} + 
 10 Z^4 t_2 t_4 - 180 Z^3 t_2 t_3 t_4 - 
 100 Z^2 t_2 t_3^2 t_4 \nonumber \\
&- 6480 Z t_2 t_3^3 t_4 - 
 10260 t_1 t_3^4 t_4 + 12220 t_2 t_3^4 t_4 - 
 4860 Z^4 t_3^5 t_4 - 9085 Z^3 t_3^6 t_4 \nonumber \\
&- 
 64800 Z^2 t_3^7 t_4 + 33450 Z t_3^8 t_4 + 200880 t_3^9 t_4 - 
 225 Z^3 t_2 t_4^2 - 400 Z^2 t_2 t_3 t_4^2 \nonumber \\
&- 
 16080 Z t_2 t_3^2 t_4^2 - 30780 t_1 t_3^3 t_4^2 - 
 4840 t_2 t_3^3 t_4^2 - 22485 Z^4 t_3^4 t_4^2 - 
 109020 Z^3 t_3^5 t_4^2 \nonumber \\
&- 511200 Z^2 t_3^6 t_4^2 + 
 535200 Z t_3^7 t_4^2 + 1253640 t_3^8 t_4^2 - 
 300 Z^2 t_2 t_4^3 - 12480 Z t_2 t_3 t_4^3 \nonumber \\
&- 
 20520 t_1 t_3^2 t_4^3 - 163520 t_2 t_3^2 t_4^3 - 
 50280 Z^4 t_3^3 t_4^3 - 550695 Z^3 t_3^4 t_4^3 - 
 2505600 Z^2 t_3^5 t_4^3 \nonumber \\
&+ 3797100 Z t_3^6 t_4^3 + 
 773760 t_3^7 t_4^3 + 2410 Z t_2 t_4^4 + 46710 t_1 t_3 t_4^4 - 
 380460 t_2 t_3 t_4^4 \nonumber \\
&- 55975 Z^4 t_3^2 t_4^4 - 
 1498360 Z^3 t_3^3 t_4^4 - 8058900 Z^2 t_3^4 t_4^4 + 
 15594000 Z t_3^5 t_4^4 \nonumber \\
&- 27933080 t_3^6 t_4^4 + 
 60588 t_1 t_4^5 - 287748 t_2 t_4^5 - 29020 Z^4 t_3 t_4^5 - 
 2341695 Z^3 t_3^2 t_4^5 \nonumber \\
&- 17008800 Z^2 t_3^3 t_4^5 + 
 39864360 Z t_3^4 t_4^5 - 162541344 t_3^5 t_4^5 - 
 6375 Z^4 t_4^6 \nonumber \\
&- 2031420 Z^3 t_3 t_4^6 - 
 22768320 Z^2 t_3^2 t_4^6 + 62921280 Z t_3^3 t_4^6 - 
 465541320 t_3^4 t_4^6 \nonumber \\
&- 778805 Z^3 t_4^7 - 
 17671680 Z^2 t_3 t_4^7 + 56624500 Z t_3^2 t_4^7 - 
 804407360 t_3^3 t_4^7 \nonumber \\
&- 6115770 Z^2 t_4^8 + 
 23108560 Z t_3 t_4^8 - 886614660 t_3^2 t_4^8 + 
 1355790 Z t_4^9 \nonumber \\
&\left. - 625055760 t_3 t_4^9 - 240893344 t_4^{10}\right), \\
y_3=& \,-\frac{2^75^4}{3^5}\left(-Z t_2 - 6 t_1 t_3 + 12 t_2 t_3 + Z^4 t_3^2 - 6 Z^2 t_3^4 - 
 297 t_3^6 - 12 t_1 t_4 + 30 t_2 t_4 \right. \nonumber \\
&+ 4 Z^4 t_3 t_4 + 
 12 Z^3 t_3^2 t_4 - 48 Z^2 t_3^3 t_4 + 15 Z t_3^4 t_4 - 
 3564 t_3^5 t_4 + 3 Z^4 t_4^2 + 48 Z^3 t_3 t_4^2 \nonumber \\
&- 
 132 Z^2 t_3^2 t_4^2 + 120 Z t_3^3 t_4^2 - 22365 t_3^4 t_4^2 + 
 52 Z^3 t_4^3 - 144 Z^2 t_3 t_4^3 + 530 Z t_3^2 t_4^3 \nonumber \\
&- 
 83880 t_3^3 t_4^3 - 54 Z^2 t_4^4 + 1160 Z t_3 t_4^4 -
 166515 t_3^2 t_4^4 + 843 Z t_4^5 - 147084 t_3 t_4^5 \nonumber \\
&\left. - 20311 t_4^6\right) , \\
y_4=& \, 20t_4.
\end{align*}
\end{theorem}

\begin{proposition}
The derivatives $\frac{\partial y_i}{\partial t_j} \in \C[t_1, t_2, t_3, t_4, Z].$
\end{proposition}
\noindent Proof is similar to the one for Proposition \ref{DerivsProp}.

\subsection{$H_4(7)$ example}
The prepotential for $H_4(7)$ is
\begin{align*}
F(t)=& \, t_1t_2t_3+\frac{1}{2}t_1^2t_4-\frac{4096}{135}t_3t_4\left(315t_3^3+32t_4^5\right)+\frac{32768}{1125}t_4^2\left(75t_3^3+2t_4^5\right)Z^2 \\
&-\frac{32768}{225}t_3\left(75t_3^3+2t_4^5\right)Z^3-\frac{16384}{5625}t_4\left(375t_3^3+14t_4^5\right)Z^5+\frac{34816}{225}t_3t_4^4Z^6 \\
&-\frac{116736}{175}t_3^2t_4^2Z^7+\frac{256}{75}\left(220t_3^3+3t_4^5\right)Z^8-\frac{118784}{945}t_3t_4^3Z^9+\frac{44544}{175}t_3^2t_4Z^{10} \\
&-\frac{5632}{1575}t_4^4Z^{11}+\frac{832}{225}t_3t_4^2Z^{12}+\frac{17664}{455}t_3^2Z^{13}-\frac{352}{315}t_4^3Z^{14}+\frac{1568}{225}t_3t_4Z^{15} \\
&+\frac{496}{2975}t_4^2Z^{17}+\frac{496}{945}t_3Z^{18}+\frac{71}{1575}t_4Z^{20}+\frac{16}{7245}Z^{23},
\end{align*}
where
\begin{equation*}
P(t_2, t_3, t_4, Z):=Z^8+\frac{32}{5}t_4Z^5+64t_3Z^3-\frac{64}{5}t_4^2Z^2+128t_3t_4-t_2=0.
\end{equation*}
The Euler vector field is
\begin{equation*}
E(t)=t_1\partial_{t_1}+\frac{4}{5}t_2\partial_{t_2}+\frac{1}{2}t_3\partial_{t_3}+\frac{3}{10}t_4\partial_{t_4},
\end{equation*}
the unity vector field is $e(t)=\partial_{t_1},$ and the charge is $d=\frac{7}{10}.$ The intersection form (\ref{gDef}) is then given by
\begin{align*}
g^{11}(t)=& \, \frac{128}{703125}\left(-1250 Z t_2^2 + 11250 Z^4 t_2 t_3 + 1850000 Z^7 t_3^2 + 
 153000000 Z^2 t_3^3 \right. \nonumber \\
&+ 5375 Z^6 t_2 t_4 + 
 758000 Z t_2 t_3 t_4 + 13878000 Z^4 t_3^2 t_4 - 
 32200 Z^3 t_2 t_4^2 \nonumber \\
&+ 745600 Z^6 t_3 t_4^2 - 
 117344000 Z t_3^2 t_4^2 - 754420 t_2 t_4^3 - 
 1829120 Z^3 t_3 t_4^3 \nonumber \\
&-\left. 819072 Z^5 t_4^4 - 
 121034240 t_3 t_4^4 + 1223424 Z^2 t_4^5\right), \\
g^{12}(t)=& \, \frac{128}{9375}\left(-125 Z^7 t_2 - 11500 Z^2 t_2 t_3 + 16000 Z^5 t_3^2 + 
 5400000 t_3^3 \right. \nonumber \\
&- 450 Z^4 t_2 t_4 + 28400 Z^7 t_3 t_4 + 
 3272000 Z^2 t_3^2 t_4 + 5280 Z t_2 t_4^2 \nonumber \\
&+ 
 202480 Z^4 t_3 t_4^2 +\left. 4448 Z^6 t_4^3 - 1155840 Z t_3 t_4^3 + 
 3584 Z^3 t_4^4 - 256000 t_4^5\right), \\
g^{22}(t)=& \, \frac{512}{1875}\left(-25 Z^5 t_2 - 4250 t_2 t_3 - 40000 Z^3 t_3^2 - 
 295 Z^2 t_2 t_4 + 2480 Z^5 t_3 t_4 \right. \nonumber \\
&+\left. 622000 t_3^2 t_4 + 
 788 Z^7 t_4^2 + 123760 Z^2 t_3 t_4^2 + 5324 Z^4 t_4^3 - 
 20800 Z t_4^4\right), \\
g^{13}(t)=& \, \frac{1}{1875}\left(25 Z^4 t_2 - 4000 Z^7 t_3 - 288000 Z^2 t_3^2 - 560 Z t_2 t_4 - 
 20160 Z^4 t_3 t_4 \right. \nonumber \\
&+\left. 544 Z^6 t_4^2 + 148480 Z t_3 t_4^2 - 
 2048 Z^3 t_4^3 - 51200 t_4^4\right), \\
g^{23}(t)=& \, \frac{1}{75}\left(75 t_1 + 20 Z^2 t_2 - 320 Z^5 t_3 - 38400 t_3^2 - 128 Z^7 t_4 - 
 12160 Z^2 t_3 t_4 \right. \nonumber \\
&-\left. 704 Z^4 t_4^2 + 2560 Z t_4^3\right),
\end{align*}
\begin{align*}
g^{33}(t)=& \, \frac{Z}{600}\left(5 Z^6 + 420 Z t_3 + 35 Z^3 t_4 - 112 t_4^2\right), \\
g^{14}(t)=& \, t_1, \hspace{5mm} g^{24}(t)=\frac{4}{5}t_2, \hspace{5mm} g^{34}(t)=\frac{1}{2}t_3, \hspace{5mm} g^{44}(t)=\frac{3}{10}t_4.
\end{align*}
We have that $\mathrm{deg} \, t_1(x)=\frac{20}{3}, \, \mathrm{deg} \, t_2(x)=\frac{16}{3}, \, \mathrm{deg} \, t_3(x)=\frac{10}{3}, \, \mathrm{deg} \, t_4(x)=2$ and $\mathrm{deg} \, Z(x)=\frac{2}{3}.$

\begin{theorem} \label{theoremYH47}
We have the following relations
\begin{align*}
y_1=& \, -\frac{5^4}{3^{28}7}\left(2^33^45^{18}7 \, Z^7 t_1^3 t_2 - 
 2^33^75^{15}7 \, Z t_1^2 t_2^3+ 3^35^{14}7\!\cdot\!11\!\cdot\!31 \, Z^3 t_1 t_2^4 \right. \nonumber \\
&- 
 5^{13}7\!\cdot\!1831 \, Z^5 t_2^5 - 2^43^65^{20}7 \, t_1^4 t_3 + 
 2^43^45^{18}7^3 Z^2 t_1^3 t_2 t_3 + 
 3^55^{15}2\!\cdot\!7\!\cdot\!967 \, Z^4 t_1^2 t_2^2 t_3 \nonumber \\
&- 
 3^45^{14}2\!\cdot\!7\!\cdot\!2383 \, Z^6 t_1 t_2^3 t_3 + 
 2^25^{13}7\!\cdot\!19121 \, t_2^5 t_3 - 
 2^83^45^{18}7^2 Z^5 t_1^3 t_3^2 \nonumber \\
&- 
 2^63^65^{15}7\!\cdot\!283 \, Z^7 t_1^2 t_2 t_3^2 + 
 2^53^45^{16}7\!\cdot\!11\!\cdot\!17 \, Z t_1 t_2^3 t_3^2 - 
 2^65^{13}3\!\cdot\!7\!\cdot\!31547 \, Z^3 t_2^4 t_3^2 \nonumber \\
&- 
 2^{12}3^55^{18}7\!\cdot\!61 \, t_1^3 t_3^3 - 
 2^{12}3^55^{15}7\!\cdot\!257 \, Z^2 t_1^2 t_2 t_3^3 - 
 2^83^35^{14}7\!\cdot\!47\!\cdot\!421 \, Z^4 t_1 t_2^2 t_3^3 \nonumber \\
&- 
 2^75^{13}7\!\cdot\!131\!\cdot\!2647 \, Z^6 t_2^3 t_3^3 - 
 2^{14}3^55^{15}7\!\cdot\!1153 \, Z^5 t_1^2 t_3^4 - 
 2^{21}3^35^{14}7\!\cdot\!239 \, Z^7 t_1 t_2 t_3^4 \nonumber \\
&+ 
 2^{11}5^{13}3\!\cdot\!7\!\cdot\!181\!\cdot\!6269 \, Z t_2^3 t_3^4 - 
 2^{17}3^65^{15}7\!\cdot\!13\!\cdot\!1291 \, t_1^2 t_3^5 + 
 2^{16}3^55^{14}7\!\cdot\!16447 \, Z^2 t_1 t_2 t_3^5 \nonumber \\
&- 
 2^{13}5^{14}3\!\cdot\!7\!\cdot\!3183137 \, Z^4 t_2^2 t_3^5 - 
 2^{20}3^55^{15}7\!\cdot\!17\!\cdot\!643 \, Z^5 t_1 t_3^6 + 
 2^{18}5^{14}7\!\cdot\!1747\!\cdot\!1789 \, Z^7 t_2 t_3^6 \nonumber \\
&- 
 2^{24}3^55^{15}7\!\cdot\!34649 \, t_1 t_3^7 + 
 2^{23}5^{14}7\!\cdot\!2511151 \, Z^2 t_2 t_3^7 - 
 2^{26}5^{15}7^2112687 \, Z^5 t_3^8 \nonumber \\
&- 
 2^{28}3^35^{15}7\!\cdot\!530807 \, t_3^9 + 
 3^55^{17}7^32 \, Z^4 t_1^3 t_2 t_4 + 
 2^23^75^{14}7\!\cdot\!67 \, Z^6 t_1^2 t_2^2 t_4 \nonumber \\
&- 
 3^45^{13}7\!\cdot\!41 \, t_1 t_2^4 t_4 - 5^{13}2\!\cdot\!7\!\cdot\!4871 \, Z^2 t_2^5 t_4 - 
 2^63^45^{17}7^213 \, Z^7 t_1^3 t_3 t_4 \nonumber \\
&+ 
 2^53^55^{14}7\!\cdot\!23\!\cdot\!199 \, Z t_1^2 t_2^2 t_3 t_4 - 
 2^43^35^{14}7\!\cdot\!1091 \, Z^3 t_1 t_2^3 t_3 t_4 - 
 2^55^{13}7\!\cdot\!11\!\cdot\!3119 \, Z^5 t_2^4 t_3 t_4 \nonumber \\
&- 
 2^{11}3^45^{18}7^3 Z^2 t_1^3 t_3^2 t_4 + 
 2^73^65^{14}7\!\cdot\!10169 \, Z^4 t_1^2 t_2 t_3^2 t_4 - 
 2^73^45^{13}7\!\cdot\!53\!\cdot\!401 \, Z^6 t_1 t_2^2 t_3^2 t_4 \nonumber
\end{align*}
\begin{align*}
&+ 
 2^65^{12}7\!\cdot\!4272533 \, t_2^4 t_3^2 t_4 - 
 2^{12}3^55^{14}7\!\cdot\!37\!\cdot\!367 \, Z^7 t_1^2 t_3^3 t_4 - 
 2^{12}3^35^{13}7\!\cdot\!127\!\cdot\!251 \, Z t_1 t_2^2 t_3^3 t_4 \nonumber \\
&- 
 2^{10}5^{12}7^23\!\cdot\!19\!\cdot\!48157 \, Z^3 t_2^3 t_3^3 t_4 + 
 2^{19}3^55^{15}7\!\cdot\!257 \, Z^2 t_1^2 t_3^4 t_4 - 
 2^{13}3^35^{13}7\!\cdot\!1883993 \, Z^4 t_1 t_2 t_3^4 t_4 \nonumber \\
&+ 
 2^{13}5^{13}7^23\!\cdot\!397\!\cdot\!557 \, Z^6 t_2^2 t_3^4 t_4 - 
 2^{18}3^35^{14}7\!\cdot\!67\!\cdot\!293 \, Z^7 t_1 t_3^5 t_4 - 
 2^{17}5^{13}13^23\!\cdot\!7\!\cdot\!22063 \, Z t_2^2 t_3^5 t_4 \nonumber \\
&- 
 2^{23}3^55^{14}7\!\cdot\!16447 \, Z^2 t_1 t_3^6 t_4 + 
 2^{19}5^{13}73^27\!\cdot\!12347 \, Z^4 t_2 t_3^6 t_4 - 
 2^{24}5^{14}7^23\!\cdot\!31\!\cdot\!6673 \, Z^7 t_3^7 t_4 \nonumber \\
&- 
 2^{30}5^{14}7\!\cdot\!2511151 \, Z^2 t_3^8 t_4 - 
 2^53^75^{16}7^2 Z t_1^3 t_2 t_4^2 + 
 2^53^75^{13}7\!\cdot\!367 \, Z^3 t_1^2 t_2^2 t_4^2 \nonumber \\
&+ 
 2^53^45^{13}7\!\cdot\!2099 \, Z^5 t_1 t_2^3 t_4^2 - 
 2^45^{12}7\!\cdot\!11\!\cdot\!3517 \, Z^7 t_2^4 t_4^2 - 
 2^73^45^{16}7^2173 \, Z^4 t_1^3 t_3 t_4^2 \nonumber \\
&- 
 2^63^55^{13}7\!\cdot\!19\!\cdot\!29\!\cdot\!59 \, Z^6 t_1^2 t_2 t_3 t_4^2 + 
 2^53^45^{12}7\!\cdot\!407717 \, t_1 t_2^3 t_3 t_4^2 - 
 2^53^35^{13}7\!\cdot\!149\!\cdot\!271 \, Z^2 t_2^4 t_3 t_4^2 \nonumber \\
&- 
 2^{11}3^55^{13}7\!\cdot\!19\!\cdot\!3119 \, Z t_1^2 t_2 t_3^2 t_4^2 + 
 2^93^45^{12}7\!\cdot\!61\!\cdot\!48311 \, Z^3 t_1 t_2^2 t_3^2 t_4^2 + 
 2^95^{13}7\!\cdot\!11\!\cdot\!19\!\cdot\!479 \, Z^5 t_2^3 t_3^2 t_4^2 \nonumber \\
&- 
 2^{13}3^55^{13}7\!\cdot\!311\!\cdot\!439 \, Z^4 t_1^2 t_3^3 t_4^2 - 
 2^{14}3^45^{13}7\!\cdot\!61\!\cdot\!4663 \, Z^6 t_1 t_2 t_3^3 t_4^2 + 
 2^{11}5^{12}7\!\cdot\!13\!\cdot\!3501503 \, t_2^3 t_3^3 t_4^2 \nonumber \\
&+ 
 2^{17}3^45^{13}7\!\cdot\!14029 \, Z t_1 t_2 t_3^4 t_4^2 - 
 2^{15}5^{12}3\!\cdot\!7\!\cdot\!11\!\cdot\!13\!\cdot\!104297 \, Z^3 t_2^2 t_3^4 t_4^2 - 
 2^{19}3^35^{13}7\!\cdot\!4814471 \, Z^4 t_1 t_3^5 t_4^2 \nonumber \\
&- 
 2^{18}3^25^{13}7\!\cdot\!3231829 \, Z^6 t_2 t_3^5 t_4^2 + 
 2^{23}5^{13}3\!\cdot\!7\!\cdot\!47\!\cdot\!263303 \, Z t_2 t_3^6 t_4^2 - 
 2^{25}3^25^{13}7\!\cdot\!1951\!\cdot\!29207 \, Z^4 t_3^7 t_4^2 \nonumber \\
&+ 
 2^63^45^{15}7^217 \, Z^6 t_1^3 t_4^3 - 
 2^43^75^{12}7\!\cdot\!11\!\cdot\!79 \, t_1^2 t_2^2 t_4^3 + 
 2^43^35^{12}7\!\cdot\!89\!\cdot\!6449 \, Z^2 t_1 t_2^3 t_4^3 \nonumber \\
&+ 
 2^45^{11}3\!\cdot\!7\!\cdot\!211\!\cdot\!6449 \, Z^4 t_2^4 t_4^3 + 
 2^{12}3^75^{16}7^2 Z t_1^3 t_3 t_4^3 - 
 2^{10}3^55^{12}7\!\cdot\!208253 \, Z^3 t_1^2 t_2 t_3 t_4^3 \nonumber \\
&- 
 2^{15}3^35^{12}7\!\cdot\!1429 \, Z^5 t_1 t_2^2 t_3 t_4^3 - 
 2^85^{11}7\!\cdot\!31\!\cdot\!181\!\cdot\!3659 \, Z^7 t_2^3 t_3 t_4^3 + 
 2^{12}3^55^{14}7\!\cdot\!13\!\cdot\!37\!\cdot\!59 \, Z^6 t_1^2 t_3^2 t_4^3 \nonumber \\
&+ 
 2^{16}3^25^{13}23^2 Z^{10} t_2^2 t_3^2 t_4^3 + 
 2^{12}3^35^{12}7^21178167 \, t_1 t_2^2 t_3^2 t_4^3 - 
 2^{10}5^{11}2107634939 \, Z^2 t_2^3 t_3^2 t_4^3 \nonumber \\
&+ 
 2^{18}3^85^{13}7\!\cdot\!1033 \, Z t_1^2 t_3^3 t_4^3 - 
 2^{17}3^35^{12}7\!\cdot\!13\!\cdot\!23\!\cdot\!4099 \, Z^3 t_1 t_2 t_3^3 t_4^3 + 
 2^{18}5^{11}107\!\cdot\!330509 \, Z^5 t_2^2 t_3^3 t_4^3 \nonumber \\
&- 
 2^{18}3^45^{13}7\!\cdot\!377563 \, Z^6 t_1 t_3^4 t_4^3 + 
 2^{24}3^25^{16}23 \, Z^8 t_2 t_3^4 t_4^3 + 
 2^{16}5^{13}491035057 \, t_2^2 t_3^4 t_4^3 \nonumber \\
&- 
 2^{24}3^35^{13}7\!\cdot\!373\!\cdot\!523 \, Z t_1 t_3^5 t_4^3 - 
 2^{22}5^{12}3\!\cdot\!658847099 \, Z^3 t_2 t_3^5 t_4^3 + 
 2^{24}5^{12}7\!\cdot\!25229\!\cdot\!55439 \, Z^6 t_3^6 t_4^3 \nonumber \\
&- 
 2^{30}5^{13}3\!\cdot\!7\!\cdot\!23\!\cdot\!71\!\cdot\!5791 \, Z t_3^7 t_4^3 - 
 2^{12}3^45^{15}7^2 Z^3 t_1^3 t_4^4 + 
 2^93^{12}5^{12}7 \, Z^5 t_1^2 t_2 t_4^4 \nonumber \\
&+ 
 2^73^35^{11}7\!\cdot\!13\!\cdot\!17\!\cdot\!5981 \, Z^7 t_1 t_2^2 t_4^4 + 
 2^{12}3^25^{12}13 \, Z^9 t_2^3 t_4^4 + 
 2^85^{10}1553\!\cdot\!74317 \, Z t_2^4 t_4^4 \nonumber \\
&- 
 2^{10}3^55^{12}7\!\cdot\!23\!\cdot\!97849 \, t_1^2 t_2 t_3 t_4^4 - 
 2^{16}3^25^{12}23 \, Z^{12} t_2^2 t_3 t_4^4 + 
 2^93^45^{11}7^2644789 \, Z^2 t_1 t_2^2 t_3 t_4^4 \nonumber \\
&- 
 2^{10}5^{10}684923563 \, Z^4 t_2^3 t_3 t_4^4 + 
 2^{16}3^55^{13}7\!\cdot\!11\!\cdot\!23\!\cdot\!157 \, Z^3 t_1^2 t_3^2 t_4^4 + 
 2^{14}3^55^{11}7\!\cdot\!239\!\cdot\!1361 \, Z^5 t_1 t_2 t_3^2 t_4^4 \nonumber \\
&- 
 2^{15}5^{10}3\!\cdot\!7\!\cdot\!105745993 \, Z^7 t_2^2 t_3^2 t_4^4 + 
 2^{24}3^45^{14} Z^{10} t_2 t_3^3 t_4^4 + 
 2^{17}3^55^{13}7^2144073 \, t_1 t_2 t_3^3 t_4^4 \nonumber \\
&- 
 2^{15}5^{10}31\!\cdot\!2275515659 \, Z^2 t_2^2 t_3^3 t_4^4 - 
 2^{25}3^35^{12}7^219\!\cdot\!1987 \, Z^3 t_1 t_3^4 t_4^4 \nonumber \\
&- 
 2^{20}5^{10}7\!\cdot\!17\!\cdot\!683\!\cdot\!247601 \, Z^5 t_2 t_3^4 t_4^4 + 
 2^{30}3^25^{17} Z^8 t_3^5 t_4^4 - 
 2^{22}5^{11}56634970591 \, t_2 t_3^5 t_4^4 \nonumber \\
&+ 
 2^{28}3^25^{11}59\!\cdot\!56802241 \, Z^3 t_3^6 t_4^4 - 
 2^{16}3^45^{15}7\!\cdot\!97 \, t_1^3 t_4^5 + 
 2^{10}3^{12}5^{10}7\!\cdot\!103 \, Z^2 t_1^2 t_2 t_4^5 \nonumber \\
&+ 
 2^{14}3^25^{11} Z^{14} t_2^2 t_4^5 + 
 2^83^35^97\!\cdot\!17\!\cdot\!6436589 \, Z^4 t_1 t_2^2 t_4^5 + 
 2^85^83\!\cdot\!83\!\cdot\!113\!\cdot\!283487 \, Z^6 t_2^3 t_4^5 \nonumber \\
&- 
 2^{15}3^{10}5^{10}7^2367 \, Z^5 t_1^2 t_3 t_4^5 + 
 2^{12}3^35^{10}7\!\cdot\!47\!\cdot\!409\!\cdot\!1801 \, Z^7 t_1 t_2 t_3 t_4^5 - 
 2^{15}3^35^{11}89\!\cdot\!241 \, Z^9 t_2^2 t_3 t_4^5 \nonumber \\
&+ 
 2^{15}5^819^2379\!\cdot\!72043 \, Z t_2^3 t_3 t_4^5 - 
 2^{17}3^55^{12}7\!\cdot\!67\!\cdot\!151\!\cdot\!443 \, t_1^2 t_3^2 t_4^5 + 
 2^{20}3^25^{12}29\!\cdot\!173 \, Z^{12} t_2 t_3^2 t_4^5 \nonumber \\
&+ 
 2^{17}3^35^97\!\cdot\!242346031 \, Z^2 t_1 t_2 t_3^2 t_4^5 - 
 2^{14}5^83\!\cdot\!17\!\cdot\!29\!\cdot\!307\!\cdot\!2364179 \, Z^4 t_2^2 t_3^2 t_4^5 \nonumber \\
&+ 
 2^{21}3^35^97\!\cdot\!151247651 \, Z^5 t_1 t_3^3 t_4^5 + 
 2^{18}5^921157\!\cdot\!4697569 \, Z^7 t_2 t_3^3 t_4^5 + 
 2^{27}3^25^{15}13\!\cdot\!17 \, Z^{10} t_3^4 t_4^5 \nonumber \\
&+ 
 2^{24}3^35^{13}7\!\cdot\!2004073 \, t_1 t_3^4 t_4^5 - 
 2^{22}5^83\!\cdot\!883\!\cdot\!159406111 \, Z^2 t_2 t_3^4 t_4^5 - 
 2^{29}5^83343\!\cdot\!11789651 \, Z^5 t_3^5 t_4^5 \nonumber \\
&- 
 2^{29}5^{11}85909378939 \, t_3^6 t_4^5 + 
 2^{14}3^{14}5^97 \, Z^7 t_1^2 t_4^6 + 
 2^{14}3^25^{10}1523 \, Z^{11} t_2^2 t_4^6 \nonumber \\
&- 
 2^{11}3^35^87\!\cdot\!13\!\cdot\!73\!\cdot\!627131 \, Z t_1 t_2^2 t_4^6 + 
 2^{11}5^76691\!\cdot\!10945717 \, Z^3 t_2^3 t_4^6 - 
 2^{16}3^{10}5^{10}7\!\cdot\!739 \, Z^2 t_1^2 t_3 t_4^6 \nonumber \\
&- 
 2^{19}3^35^{11}73 \, Z^{14} t_2 t_3 t_4^6 + 
 2^{13}3^35^87\!\cdot\!6560672719 \, Z^4 t_1 t_2 t_3 t_4^6 - 
 2^{13}5^83\!\cdot\!19\!\cdot\!107\!\cdot\!271\!\cdot\!4603 \, Z^6 t_2^2 t_3 t_4^6 \nonumber \\
&- 
 2^{21}3^45^87\!\cdot\!195868609 \, Z^7 t_1 t_3^2 t_4^6 + 
 2^{21}3^25^{10}173\!\cdot\!2039 \, Z^9 t_2 t_3^2 t_4^6 - 
 2^{17}5^73\!\cdot\!175949\!\cdot\!534167 \, Z t_2^2 t_3^2 t_4^6 \nonumber \\
&+ 
 2^{26}3^35^{13}73 \, Z^{12} t_3^3 t_4^6 - 
 2^{25}3^35^97\!\cdot\!307854203 \, Z^2 t_1 t_3^3 t_4^6 + 
 2^{19}5^717\!\cdot\!1171\!\cdot\!755025827 \, Z^4 t_2 t_3^3 t_4^6 \nonumber \\
&+ 
 2^{26}5^73\!\cdot\!11\!\cdot\!29\!\cdot\!1522702439 \, Z^7 t_3^4 t_4^6 + 
 2^{28}5^817\!\cdot\!23\!\cdot\!99971\!\cdot\!187073 \, Z^2 t_3^5 t_4^6 + 
 2^{17}3^{14}5^97 \, Z^4 t_1^2 t_4^7 \nonumber \\
&+ 
 2^{12}3^35^77\!\cdot\!17\!\cdot\!61\!\cdot\!545533 \, Z^6 t_1 t_2 t_4^7 + 
 2^{12}3^25^91747\!\cdot\!2179 \, Z^8 t_2^2 t_4^7 + 
 2^{12}5^63\!\cdot\!7\!\cdot\!25411781761 \, t_2^3 t_4^7 \nonumber
\end{align*}
\begin{align*}
&- 
 2^{18}3^25^97\!\cdot\!31\!\cdot\!101\!\cdot\!109 \, Z^{11} t_2 t_3 t_4^7 + 
 2^{18}3^45^77\!\cdot\!3491\!\cdot\!87421 \, Z t_1 t_2 t_3 t_4^7 \nonumber \\
&+ 
 2^{16}5^63\!\cdot\!11\!\cdot\!41\!\cdot\!625672687 \, Z^3 t_2^2 t_3 t_4^7 + 
 2^{22}3^45^{11}73^2 Z^{14} t_3^2 t_4^7 \nonumber \\
&- 
 2^{20}3^35^77^213^2607\!\cdot\!45503 \, Z^4 t_1 t_3^2 t_4^7 - 
 2^{18}3^35^6128237\!\cdot\!6437531 \, Z^6 t_2 t_3^2 t_4^7 \nonumber \\
&- 
 2^{28}3^35^{11}6197 \, Z^9 t_3^3 t_4^7 - 
 2^{24}5^63\!\cdot\!443\!\cdot\!8098119613 \, Z t_2 t_3^3 t_4^7 \nonumber \\
&+ 
 2^{26}5^63\!\cdot\!7\!\cdot\!715783571707 \, Z^4 t_3^4 t_4^7 + 
 2^{19}3^25^87541 \, Z^{13} t_2 t_4^8 \nonumber \\
&- 
 2^{19}3^35^67^217^219\!\cdot\!31\!\cdot\!199 \, Z^3 t_1 t_2 t_4^8 + 
 2^{15}5^63\!\cdot\!23\!\cdot\!29\!\cdot\!219793529 \, Z^5 t_2^2 t_4^8 \nonumber \\
&+ 
 2^{20}3^35^67\!\cdot\!23\!\cdot\!67\!\cdot\!1297\!\cdot\!3301 \, Z^6 t_1 t_3 t_4^8 - 
 2^{20}3^45^928807 \, Z^8 t_2 t_3 t_4^8 \nonumber \\
&+ 
 2^{17}5^53\!\cdot\!7\!\cdot\!2544911109121 \, t_2^2 t_3 t_4^8 + 
 2^{25}3^25^919\!\cdot\!43\!\cdot\!829 \, Z^{11} t_3^2 t_4^8 \nonumber \\
&+ 
 2^{26}3^35^77\!\cdot\!53\!\cdot\!167\!\cdot\!116381 \, Z t_1 t_3^2 t_4^8 - 
 2^{26}3^25^5167\!\cdot\!5651\!\cdot\!157637 \, Z^3 t_2 t_3^2 t_4^8 \nonumber \\
&+ 
 2^{26}3^25^519\!\cdot\!47\!\cdot\!359\!\cdot\!6805069 \, Z^6 t_3^3 t_4^8 + 
 2^{32}5^63\!\cdot\!31\!\cdot\!1249\!\cdot\!8516173 \, Z t_3^4 t_4^8 \nonumber \\
&+ 
 2^{18}3^35^711\!\cdot\!151\!\cdot\!2281 \, Z^{10} t_2 t_4^9 - 
 2^{18}3^35^57^21453\!\cdot\!9140533 \, t_1 t_2 t_4^9 \nonumber \\
&+ 
 2^{16}5^58861954242873 \, Z^2 t_2^2 t_4^9 - 
 2^{23}3^35^873\!\cdot\!7541 \, Z^{13} t_3 t_4^9 \nonumber \\
&- 
 2^{25}3^35^57\!\cdot\!4139\!\cdot\!129119 \, Z^3 t_1 t_3 t_4^9 - 
 2^{23}5^5151\!\cdot\!5680110559 \, Z^5 t_2 t_3 t_4^9 \nonumber \\
&- 
 2^{26}3^25^740185941 \, Z^8 t_3^2 t_4^9 + 
 2^{24}5^417\!\cdot\!23621998776373 \, t_2 t_3^2 t_4^9 \nonumber \\
&+ 
 2^{30}5^411\!\cdot\!764381\!\cdot\!5742169 \, Z^3 t_3^3 t_4^9 - 
 2^{24}3^{14}5^47\!\cdot\!11\!\cdot\!863 \, Z^5 t_1 t_4^{10} \nonumber \\
&+ 
 2^{20}5^317\!\cdot\!233941\!\cdot\!6976421 \, Z^7 t_2 t_4^{10} - 
 2^{25}3^25^67\!\cdot\!1381231 \, Z^{10} t_3 t_4^{10} \nonumber \\
&- 
 2^{25}3^35^57\!\cdot\!7457\!\cdot\!33020201 \, t_1 t_3 t_4^{10} - 
 2^{25}5^358812201936403 \, Z^2 t_2 t_3 t_4^{10} \nonumber \\
&+ 
 2^{29}5^339878573708081 \, Z^5 t_3^2 t_4^{10} + 
 2^{32}5^43909211\!\cdot\!4328957 \, t_3^3 t_4^{10} \nonumber \\
&+ 
 2^{22}3^25^57541^2 Z^{12} t_4^{11} + 
 2^{24}3^{13}5^47^25399 \, Z^2 t_1 t_4^{11} - 
 2^{28}3^25^57\!\cdot\!53\!\cdot\!7541 Z^9 t_4^{12} \nonumber \\
&+ 
 2^{22}5^2225349\!\cdot\!1564621501 \, Z^4 t_2 t_4^{11} - 
 2^{27}5^23\!\cdot\!107\!\cdot\!127\!\cdot\!4920163967 \, Z^7 t_3 t_4^{11} \nonumber \\
&- 
 2^{31}5^313\!\cdot\!1583\!\cdot\!4421\!\cdot\!158759 \, Z^2 t_3^2 t_4^{11} - 
 2^{26}5\!\cdot\!296367312358063 \, Z t_2 t_4^{12} \nonumber \\
&- 
 2^{32}5\!\cdot\!1383659\!\cdot\!123837583 \, Z^4 t_3 t_4^{12} + 
 2^{30}7\!\cdot\!249677\!\cdot\!73045429 \, Z^6 t_4^{13} \nonumber \\
&+ 
 2^{33}5\!\cdot\!296367312358063 \, Z t_3 t_4^{13} - 
 2^{32}11\!\cdot\!43\!\cdot\!104455205831 \, Z^3 t_4^{14} \nonumber \\
&\left. - 2^{50}5^{10}7\!\cdot\!349 \, t_4^{15}\right), \\
y_2=& \, \frac{5^3}{2^33^{18}}\left(2^23^35^{13} t_1^3 - 3^25^{10}17 \, Z^4 t_1 t_2^2 - 
 5^82\!\cdot\!3\!\cdot\!7 \, Z^6 t_2^3 - 2^53^35^{10}11 \, Z^7 t_1 t_2 t_3 \right. \nonumber \\
&+ 
 2^53^35^831 \, Z t_2^3 t_3 + 2^83^45^{12}19 \, t_1^2 t_3^2 - 
 2^93^25^{10}353 \, Z^2 t_1 t_2 t_3^2 - 
 2^63^25^8367 \, Z^4 t_2^2 t_3^2 \nonumber \\
&+ 2^{15}3^25^{10}29 \, Z^5 t_1 t_3^3 - 
 2^{11}5^83\!\cdot\!4327 \, Z^7 t_2 t_3^3 + 2^{14}3^35^{10}1319 \, t_1 t_3^4 \nonumber \\
&- 
 2^{15}5^83\!\cdot\!109\!\cdot\!137 \, Z^2 t_2 t_3^4 - 2^{21}5^{11}3\!\cdot\!11 \, Z^5 t_3^5 + 
 2^{20}3^25^{10}97 \, t_3^6 - 
 2^45^75023 \, Z^3 t_2^3 t_4 \nonumber \\
&+ 2^43^25^911\!\cdot\!13 \, Z t_1 t_2^2 t_4 - 2^73^35^{10}7^2 Z^4 t_1 t_2 t_3 t_4 + 
 2^75^73\!\cdot\!4817 \, Z^6 t_2^2 t_3 t_4 \nonumber \\
&+ 
 2^{15}3^25^953 \, Z^7 t_1 t_3^2 t_4 - 
 2^{10}5^73\!\cdot\!13\!\cdot\!397 \, Z t_2^2 t_3^2 t_4 + 
 2^{16}3^25^{10}353 \, Z^2 t_1 t_3^3 t_4 \nonumber \\
&- 
 2^{13}5^719\!\cdot\!41\!\cdot\!317 \, Z^4 t_2 t_3^3 t_4 + 
 2^{19}5^811\!\cdot\!1723 \, Z^7 t_3^4 t_4 + 
 2^{22}5^83\!\cdot\!109\!\cdot\!137 \, Z^2 t_3^5 t_4 \nonumber \\
&- 
 2^53^25^8547 \, Z^6 t_1 t_2 t_4^2 - 2^43^25^67\!\cdot\!13\!\cdot\!167 \, t_2^3 t_4^2 - 
 2^{10}3^25^811\!\cdot\!103 \, Z t_1 t_2 t_3 t_4^2 \nonumber \\
&- 
 2^95^63\!\cdot\!7\!\cdot\!3907 \, Z^3 t_2^2 t_3 t_4^2 + 
 2^{15}3^25^8881 \, Z^4 t_1 t_3^2 t_4^2 + 
 2^{11}3^25^731883 \, Z^6 t_2 t_3^2 t_4^2 \nonumber \\
&- 
 2^{16}5^843\!\cdot\!461 \, Z t_2 t_3^3 t_4^2 + 
 2^{22}5^7127\!\cdot\!1301 \, Z^4 t_3^4 t_4^2 + 
 2^93^25^7823 \, Z^3 t_1 t_2 t_4^3 \nonumber \\
&- 2^85^63\!\cdot\!7\!\cdot\!37\!\cdot\!53 \, Z^5 t_2^2 t_4^3 - 
 2^{13}3^25^87\!\cdot\!53 \, Z^6 t_1 t_3 t_4^3 - 
 2^{12}5^63\!\cdot\!73\!\cdot\!1877 \, t_2^2 t_3 t_4^3 \nonumber \\
&- 
 2^{17}3^55^811 \, Z t_1 t_3^2 t_4^3 + 
 2^{17}3^25^613\!\cdot\!137 \, Z^3 t_2 t_3^2 t_4^3 - 
 2^{19}3^45^7641 \, Z^6 t_3^3 t_4^3 \nonumber \\
&+ 
 2^{23}5^796601 \, Z t_3^4 t_4^3 + 
 2^83^25^67\!\cdot\!13\!\cdot\!67\!\cdot\!79 \, t_1 t_2 t_4^4 - 2^{10}5^57^210343 \, Z^2 t_2^2 t_4^4 \nonumber \\
&+ 
 2^{14}3^25^619\!\cdot\!443 \, Z^3 t_1 t_3 t_4^4 + 
 2^{13}5^57^229\!\cdot\!409 \, Z^5 t_2 t_3 t_4^4 - 
 2^{14}5^517\!\cdot\!379\!\cdot\!17209 \, t_2 t_3^2 t_4^4 \nonumber
\end{align*}
\begin{align*}
&- 
 2^{20}5^517\!\cdot\!19\!\cdot\!13931 \, Z^3 t_3^3 t_4^4 + 
 2^{13}3^75^579 \, Z^5 t_1 t_4^5 - 2^{12}5^313\!\cdot\!59\!\cdot\!5623 \, Z^7 t_2 t_4^5 \nonumber \\
&+ 
 2^{15}3^25^613\!\cdot\!103\!\cdot\!883 \, t_1 t_3 t_4^5 + 
 2^{15}5^311\!\cdot\!5478043 \, Z^2 t_2 t_3 t_4^5 \nonumber \\
&- 
 2^{19}5^311\!\cdot\!83\!\cdot\!32839 \, Z^5 t_3^2 t_4^5 + 
 2^{21}5^55189507 \, t_3^3 t_4^5 - 2^{14}3^75^559 \, Z^2 t_1 t_4^6 \nonumber \\
&- 
 2^{12}5^2327966773 \, Z^4 t_2 t_4^6 + 2^{18}5^23\!\cdot\!83\!\cdot\!593\!\cdot\!613 \, Z^7 t_3 t_4^6 \nonumber \\
&+ 
 2^{20}5^31249\!\cdot\!3677 \, Z^2 t_3^2 t_4^6 + 2^{17}5\!\cdot\!139571863 \, Z t_2 t_4^7 + 
 2^{19}5\!\cdot\!1099915517 \, Z^4 t_3 t_4^7 \nonumber \\
&- 2^{18}13\!\cdot\!49685821 \, Z^6 t_4^8 - 
 2^{24}5\!\cdot\!139571863 \, Z t_3 t_4^8 + 2^{23}19\!\cdot\!1661677 \, Z^3 t_4^9 \nonumber \\
&\left. + 2^{35}5^717 \, t_4^{10}\right), \\
y_3=& \, -\frac{5^3}{2^23^{11}}\left(3^35^6 t_1 t_2 - 2^25^413 \, Z^2 t_2^2 + 2^65^423 \, Z^5 t_2 t_3 + 
 2^65^4587 \, t_2 t_3^2 \right. \nonumber \\
&+ 2^{13}5^7 Z^3 t_3^3 - 2^55^3 Z^7 t_2 t_4 + 
 2^73^35^511 \, t_1 t_3 t_4 + 2^95^353 \, Z^2 t_2 t_3 t_4 \nonumber \\
&+ 
 2^{12}5^5 Z^5 t_3^2 t_4 + 2^{13}5^423\!\cdot\!131 \, t_3^3 t_4 - 
 2^45^21523 \, Z^4 t_2 t_4^2 + 2^95^33\!\cdot\!73 \, Z^7 t_3 t_4^2 \nonumber \\
&- 
 2^{13}5^329 \, Z^2 t_3^2 t_4^2 + 2^85\!\cdot\!31\!\cdot\!41 \, Z t_2 t_4^3 + 
 2^{11}5\!\cdot\!7\!\cdot\!491 \, Z^4 t_3 t_4^3 - 2^97541 \, Z^6 t_4^4 \nonumber \\
&-\left. 
 2^{15}5\!\cdot\!31\!\cdot\!41 \, Z t_3 t_4^4 + 2^{14}7\!\cdot\!53 \, Z^3 t_4^5 + 2^{21}5^311 \, t_4^6\right), \\
y_4=& \, \frac{40}{3}t_4.
\end{align*}
\end{theorem}

\begin{proposition}
The derivatives $\frac{\partial y_i}{\partial t_j} \in \C[t_1, t_2, t_3, t_4, Z].$
\end{proposition}
\noindent Proof is similar to the one for Proposition \ref{DerivsProp}.

\section{Remarks on almost duality} \label{AlmostDualitySection}

Let us define on a Frobenius manifold $M$ the following tensor fields:
\begin{equation} \label{Above}
\overset{*}{c}_{ijk}:=g_{i\lambda}\overset{*}{c}{\vphantom{c}}^\lambda_{jk}, \hspace{10mm} \overset{*}{c}{\vphantom{c}}^{ijk}:=g^{i\lambda}c^{jk}_\lambda,
\end{equation}
where $\overset{*}{c}{\vphantom{c}}^\lambda_{jk}$ and $c^{jk}_\lambda$ are given by formulas (\ref{DualStructureConstants}) and (\ref{shorthand}), respectively. It can be shown that $\overset{*}{c}_{ijk}(x)=\frac{\partial^3 F_*}{\partial x_i \partial x_j \partial x_k}$ for a function $F_*(x),$ which is the dual prepotential of $M$ \cite{DubrovinAlmost}.

For irreducible polynomial Frobenius manifolds, it was shown in \cite{DubrovinAlmost} that their dual prepotentials (up to rescaling) have the following simple form:
\begin{equation} \label{CoxeterDualPrepotential}
F_*(x)=\sum_{\alpha \in R_+} \frac{(\alpha, \, x)^2}{(\alpha, \, \alpha)} \, \mathrm{log}(\alpha, \, x),
\end{equation}
where $R_+$ is a positive root system for the associated Coxeter group $W.$ Below we give some partial results about dual prepotentials for some algebraic Frobenius manifolds.

\subsection{Two-dimensional examples} \label{2dimSubsection}

A two-dimensional (semisimple) algebraic Frobenius manifold has a prepotential of the form
\begin{equation} \label{2dimPrepotential}
F(t)=\frac{1}{2}t_1^2t_2+\frac{k(2k)^k}{k^2-1}t_2^{k+1},
\end{equation}
with $k \in \Q \setminus \{-1, 0, 1\}$ (see \cite{DubrovinNotes}). This has degrees $d_1=1$ and $d_2=\frac{2}{k},$ and charge $d=\frac{k-2}{k}.$ The choice of the coefficient of $t_2^{k+1}$ in formula (\ref{2dimPrepotential}) is convenient for having a simple relation between coordinates $t_1, \, t_2$ and the flat coordinates of the intersection form $x_1, \, x_2.$ Using formulas (\ref{3rdDerivs}), (\ref{EulerDiag}) and (\ref{gDef}), we find the intersection form:
\begin{equation*}
g^{ij}(t)=\begin{pmatrix}
(2k)^{k+1}t_2^{k-1} & t_1 \\
t_1 & \frac{2}{k}t_2
\end{pmatrix}.
\end{equation*}
Using formulas (\ref{shorthand}), (\ref{DualStructureConstants}) and (\ref{Above}), we get
\begin{align}
\overset{*}{c}_{111}(t)=& \, -4k^{-1}t_1t_2D, &\overset{*}{c}_{112}(t)&=\left(4(2k)^kt_2^k+t_1^2\right)D, \label{DualDerivativesInT1} \\
\overset{*}{c}_{122}(t)=& \, -2(2k)^{k+1}t_1t_2^{k-1}D, &\overset{*}{c}_{222}(t)&=(2k)^kk^2t_2^{k-2}\left(4(2k)^kt_2^k+t_1^2\right)D, \label{DualDerivativesInT2}
\end{align}
where $D=\mathrm{det}\left(g^{ij}(t)\right)^{-2}=\left(4(2k)^kt_2^k-t_1^2\right)^{-2}.$ Similar to the polynomial case $k \in \Z_{\geq 2}$ considered in \cite{DubrovinNotes}, the flat coordinates of the metric $t_1, \, t_2$ are related to the flat coordinates of the intersection form $x_1, \, x_2$ by the following formulas:
\begin{equation} \label{2dimRelations}
t_1=z^k+\overline{z}^k, \hspace{20mm} t_2=\frac{z\overline{z}}{2k},
\end{equation}
where $z:=x_1+ix_2$ and $\overline{z}:=x_1-ix_2.$

Here and in the next two theorems, we assume that when taking powers of $k$ we are working in an open set $U \subseteq \C$ which contains points $z, \, \overline{z}, \, 2k, \, \frac{z\overline{z}}{2k}$ and $2ix_2.$ In the open set $U$ we choose a single branch of the function $f(w)=w^k$ so that we have the relation $f(z)f(\overline{z})=f(2k)f(\frac{z \overline{z}}{2k}).$ For example, we can assume that $U$ does not contain the non-positive imaginary axis which can be achieved for $k>0$ by taking $|\mathrm{Re}(x_1)|, \, |\mathrm{Im}(x_1)|<1$ and $\mathrm{Re}(x_2), \, \mathrm{Im}(x_2)>1.$ Similarly, for $k<0$ we can assume that $U$ does not contain the non-negative imaginary axis and take the same conditions for $x_1$ and $x_2.$

Performing a tensorial transformation of (\ref{DualDerivativesInT1}) and (\ref{DualDerivativesInT2}) with the relations (\ref{2dimRelations}), we get the following third order derivatives of the dual prepotential:
\begin{align}
\hspace{-2.5mm} \overset{*}{c}_{111}(x)=& \, \frac{k \, x_1(x_1^2+3x_2^2)}{(z\overline{z})^2}+\frac{2ki \, x_2^3}{(z\overline{z})^2} \, \frac{\overline{z}^k+z^k}{\overline{z}^k-z^k}, &\overset{*}{c}_{112}(x)&=\frac{k \, x_2(x_2^2-x_1^2)}{(z\overline{z})^2}-\frac{2ki \, x_1x_2^2}{(z\overline{z})^2} \, \frac{\overline{z}^k+z^k}{\overline{z}^k-z^k}, \label{cStar1} \\
\hspace{-2.5mm} \overset{*}{c}_{122}(x)=& \, \frac{k \, x_1(x_1^2-x_2^2)}{(z\overline{z})^2}+\frac{2ki \, x_1^2x_2}{(z\overline{z})^2} \, \frac{\overline{z}^k+z^k}{\overline{z}^k-z^k}, &\overset{*}{c}_{222}(x)&=\frac{k \, x_2(x_2^2+3x_1^2)}{(z\overline{z})^2}-\frac{2ki \, x_1^3}{(z\overline{z})^2} \, \frac{\overline{z}^k+z^k}{\overline{z}^k-z^k}. \label{cStar4}
\end{align}
For the next result we assume that $k=l^{-1}$ with $l \in \Z_{\geq 2}.$ Also, we will make use of the hypergeometric function ${}_2F_1(a, b; c; w)$ which is single-valued for the argument $|w|<1.$ This condition holds for $w=\frac{ix_1+x_2}{2x_2}$ when we use the constraints specified above.

\begin{theorem} \label{1overL}
Let $M$ be a two-dimensional Frobenius manifold with prepotential (\ref{2dimPrepotential}) with $k=l^{-1},$ where $l \in \Z_{\geq 2}.$ Then the dual prepotential of $M$ has the form
\begin{align}
F_*(x)=& \, \frac{x_2^2}{l} \, \mathrm{log} \, x_2+\frac{\overline{z}^2}{4l} \, \mathrm{log}\, \overline{z}+\frac{z^2}{4l} \, \mathrm{log} \, z \nonumber \\
&+\sum_{j=1}^{l-1}\frac{\overline{z}^\frac{j}{l}}{4j}\left(\frac{lx_1+(l-2j)ix_2}{(j-l)z^{\frac{j}{l}-1}}+(2ix_2)^{2-\frac{j}{l}}{}_2F_1\left(\frac{j}{l}, \frac{j}{l}; \frac{j}{l}+1; \frac{ix_1+x_2}{2x_2}\right)\right), \label{DualPrepotential2dim}
\end{align}
where ${}_2F_1(a, b; c; w)$ is the hypergeometric function.
\end{theorem}
\begin{proof}
For $k=l^{-1},$ the third order derivatives of the dual prepotential given by formulas (\ref{cStar1})--(\ref{cStar4}) may be simplified as
\begin{align}
\overset{*}{c}_{111}(x)=& \, \frac{x_1}{lz\overline{z}}-\frac{2x_2^2}{l(z\overline{z})^2}\sum\limits_{j=1}^{l-1} \overline{z}^\frac{j}{l}z^\frac{l-j}{l}, \label{cStar1m1} \\
\overset{*}{c}_{112}(x)=& \, \frac{x_2}{lz\overline{z}}+\frac{2x_1x_2}{l(z\overline{z})^2}\sum\limits_{j=1}^{l-1} \overline{z}^\frac{j}{l}z^\frac{l-j}{l}, \\
\overset{*}{c}_{122}(x)=& \, -\frac{x_1}{lz\overline{z}}-\frac{2x_1^2}{l(z\overline{z})^2}\sum\limits_{j=1}^{l-1} \overline{z}^\frac{j}{l}z^\frac{l-j}{l}, \\
\overset{*}{c}_{222}(x)=& \, \frac{1}{l}\left(\frac{2}{x_2}-\frac{x_2}{z\overline{z}}\right)+\frac{2x_1^3}{lx_2(z\overline{z})^2}\sum\limits_{j=1}^{l-1} \overline{z}^\frac{j}{l}z^\frac{l-j}{l}, \label{cStar4m1}
\end{align}
where we use the identity
\begin{equation*}
\frac{\overline{z}^\frac{1}{l}+z^\frac{1}{l}}{\overline{z}^{\frac{1}{l}}-z^{\frac{1}{l}}}=\frac{\overline{z}+z}{\overline{z}-z}+\frac{2}{\overline{z}-z}\sum\limits_{j=1}^{l-1}\overline{z}^\frac{j}{l}z^\frac{l-j}{l}.
\end{equation*}
Let us define the following functions:
\begin{align*}
A(x):=& \, \frac{x_2^2}{l} \, \mathrm{log} \, x_2+\frac{\overline{z}^2}{4l} \, \mathrm{log} \, \overline{z}+\frac{z^2}{4l} \, \mathrm{log} \, z, \\
B_j(x):=& \, \frac{\overline{z}^\frac{j}{l}\left(lx_1+(l-2j)ix_2\right)}{4j(j-l)z^{\frac{j}{l}-1}},\\
C_j(x):=& \, \frac{\overline{z}^\frac{j}{l}}{4j}(2ix_2)^{2-\frac{j}{l}}{}_2F_1\left(\frac{j}{l}, \frac{j}{l}; \frac{j}{l}+1; \frac{ix_1+x_2}{2x_2}\right),
\end{align*}
for $j=1, \dots, l-1.$ Then, we want to show that
\begin{equation} \label{SumFormula}
\frac{\partial^3}{\partial x_a \partial x_b \partial x_c}\left(A(x)+\sum_{j=1}^{l-1}\left(B_j(x)+C_j(x)\right)\right)=\overset{*}{c}_{abc}(x),
\end{equation}
where $\overset{*}{c}_{abc}(x)$ are given by formulas (\ref{cStar1m1})--(\ref{cStar4m1}). The third order derivatives of $A(x)$ are
\begin{equation} \label{Aderivatives}
\frac{\partial^3 A}{\partial x_1^3}=\frac{x_1}{lz\overline{z}}, \hspace{10mm} \frac{\partial^3 A}{\partial x_1^2 \partial x_2}=\frac{x_2}{lz\overline{z}}, \hspace{10mm} \frac{\partial^3 A}{\partial x_1 \partial x_2^2}=-\frac{x_1}{lz\overline{z}}, \hspace{10mm} \frac{\partial^3 A}{\partial x_2^3}=\frac{1}{l}\left(\frac{2}{x_2}-\frac{x_2}{z\overline{z}}\right).
\end{equation}
Next, we calculate the third order derivatives of $B_j(x)$ for $j=1, \dots, l-1$ to be
\begin{align}
\frac{\partial^3 B_j}{\partial x_1^3}=& \, \frac{4ix_2^3 \, \overline{z}^{\frac{j}{l}-3}b_j}{l^3z^{\frac{j}{l}+2}}, &\frac{\partial^3 B_j}{\partial x_1^2 \partial x_2}=& \, -\frac{4ix_1x_2^2 \, \overline{z}^{\frac{j}{l}-3}b_j}{l^3z^{\frac{j}{l}+2}}, \\
\frac{\partial^3 B_j}{\partial x_1 \partial x_2^2}=& \, \frac{4ix_1^2x_2 \, \overline{z}^{\frac{j}{l}-3}b_j}{l^3z^{\frac{j}{l}+2}}, &\frac{\partial^3 B_j}{\partial x_2^3}=& \, -\frac{4ix_1^3 \, \overline{z}^{\frac{j}{l}-3}b_j}{l^3z^{\frac{j}{l}+2}}, \label{Bderivatives4}
\end{align}
where $b_j=l(l-2j)x_1+i(j^2-jl+l^2)x_2.$ Now, let us consider the first order derivatives of $C_j(x)$ for $j=1, \dots, l-1.$ We get
\begin{align*}
\frac{\partial C_j}{\partial x_1}=& \, \frac{\overline{z}^{\frac{j}{l}-1}}{4l}(2ix_2)^{2-\frac{j}{l}}{}_2F_1\left(\frac{j}{l}, \frac{j}{l}; \frac{j}{l}+1; \frac{ix_1+x_2}{2x_2}\right) \\
&-\frac{\overline{z}^\frac{j}{l}}{4j}(2ix_2)^{1-\frac{j}{l}}{}_2F_1'\left(\frac{j}{l}, \frac{j}{l}; \frac{j}{l}+1; \frac{ix_1+x_2}{2x_2}\right), \\
\frac{\partial C_j}{\partial x_2}=& \, \left(\frac{ix_1+x_2}{j}-\frac{ix_1}{2l}\right)\overline{z}^{\frac{j}{l}-1}(2ix_2)^{1-\frac{j}{l}}{}_2F_1\left(\frac{j}{l}, \frac{j}{l}; \frac{j}{l}+1; \frac{ix_1+x_2}{2x_2}\right) \\
&+\frac{ix_1}{2j}\overline{z}^\frac{j}{l}(2ix_2)^{-\frac{j}{l}}{}_2F_1'\left(\frac{j}{l}, \frac{j}{l}; \frac{j}{l}+1; \frac{ix_1+x_2}{2x_2}\right).
\end{align*}
The hypergeometric function has the properties
\begin{align*}
{}_2F_1'(a, b; c; w)=& \, \frac{c-1}{w}({}_2F_1(a, b; c-1; w)-{}_2F_1(a, b; c; w)), \\
{}_2F_1(a, b; b; w)=& \, (1-w)^{-a},
\end{align*}
for all $a, b, c \in \C.$ Here the branch of 
$(1-w)^{-a}$ is the one which equals 1 at $w=0.$ When $a=\frac{j}{l}$ and $w=\frac{ix_1+x_2}{2x_2}$ we have the relation 
\begin{equation} \label{Ratio}
{}_2F_1\left(\frac{j}{l}, b; b; \frac{ix_1+x_2}{2x_2}\right)=\left(1-\frac{ix_1+x_2}{2x_2}\right)^{-\frac{j}{l}}=\frac{(2ix_2)^\frac{j}{l}}{z^\frac{j}{l}},
\end{equation}
for any $b \in \C,$ where the functions $f(t)=t^\frac{1}{l}$ on the right-hand side of (\ref{Ratio}) are taken on the same branch, which is possible since the open set $U$ contains both $z$ and $2ix_2.$ Hence we have
\begin{equation} \label{HypergeometricDerivative}
{}_2F_1'\left(\frac{j}{l}, \frac{j}{l}; \frac{j}{l}+1; \frac{ix_1+x_2}{2x_2}\right)=\frac{2jx_2}{l(ix_1+x_2)}\left(\frac{(2ix_2)^\frac{j}{l}}{z^\frac{j}{l}}-{}_2F_1\left(\frac{j}{l}, \frac{j}{l}; \frac{j}{l}+1; \frac{ix_1+x_2}{2x_2}\right)\right).
\end{equation}
Substitution of relation (\ref{HypergeometricDerivative}) into the above formulas for the derivatives $\frac{\partial C_j}{\partial x_i}$ gives
\begin{align*}
\frac{\partial C_j}{\partial x_1}=& \, -\frac{x_2^2 \, \overline{z}^{\frac{j}{l}-1}}{lz^\frac{j}{l}}, \\
\frac{\partial C_j}{\partial x_2}=& \, \frac{i}{j}\overline{z}^\frac{j}{l}(2ix_2)^{1-\frac{j}{l}}{}_2F_1\left(\frac{j}{l}, \frac{j}{l}; \frac{j}{l}+1; \frac{ix_1+x_2}{2x_2}\right)+\frac{x_1x_2 \, \overline{z}^{\frac{j}{l}-1}}{lz^\frac{j}{l}}.
\end{align*}
Since $\frac{\partial C_j}{\partial x_1}$ contains no hypergeometric functions, its derivatives are more easily attainable, and we see that
\begin{equation} \label{Cderivatives1}
\frac{\partial^3 C_j}{\partial x_1^3}=-\frac{2x_2^2 \, \overline{z}^{\frac{j}{l}-3}c_j}{l^3z^{\frac{j}{l}+2}}, \hspace{10mm} \frac{\partial^3 C_j}{\partial x_1^2 \partial x_2}=\frac{2x_1x_2 \, \overline{z}^{\frac{j}{l}-3}c_j}{l^3z^{\frac{j}{l}+2}}, \hspace{10mm} \frac{\partial^3 C_j}{\partial x_1 \partial x_2^2}=-\frac{2x_1^2 \, \overline{z}^{\frac{j}{l}-3}c_j}{l^3z^{\frac{j}{l}+2}},
\end{equation}
where $c_j=l^2x_1^2+2il(l-2j)x_1x_2-(l^2-2jl+2j^2)x_2^2.$ On the other hand, $\frac{\partial C_j}{\partial x_2}$ still contains hypergeometric functions. Looking at the second order derivative, and using relation (\ref{HypergeometricDerivative}), we see that
\begin{equation*}
\frac{\partial^2 C_j}{\partial x_2^2}=-\frac{2}{j}\overline{z}^\frac{j}{l}(2ix_2)^{-\frac{j}{l}}{}_2F_1\left(\frac{j}{l}, \frac{j}{l}; \frac{j}{l}+1; \frac{ix_1+x_2}{2x_2}\right)+\frac{x_1 \, \overline{z}^{\frac{j}{l}-2}}{l^2z^{\frac{j}{l}+1}}(3lx_1^2+i(l-2j)x_1x_2+2lx_2^2).
\end{equation*}
Differentiating $C_j(x)$ with respect to $x_2$ for the third time and substituting relation (\ref{HypergeometricDerivative}) into this expression, we get
\begin{equation} \label{Cderivative4}
\frac{\partial^3 C_j}{\partial x_2^3}=\frac{2x_1^3 \, \overline{z}^{\frac{j}{l}-3}}{l^3x_2z^{\frac{j}{l}+2}}(l^2x_1^2+2il(2j-l)x_1x_2-(l^2-2jl+2j^2)x_2^2).
\end{equation}
From relations (\ref{Aderivatives})--(\ref{Bderivatives4}) and (\ref{Cderivatives1})--(\ref{Cderivative4}), one can check directly that formula (\ref{SumFormula}) holds.
\end{proof}

\begin{theorem}
Let $\widetilde{M}$ be a two-dimensional Frobenius manifold with prepotential (\ref{2dimPrepotential}) with $k=-l^{-1},$ where $l \in \Z_{\geq 2}.$ Then the dual prepotential of $\widetilde{M}$ has the form
\begin{equation*}
\widetilde{F_*}(x)=F_*(x)-\frac{x_1^2+x_2^2}{2l} \, \mathrm{log}(x_1^2+x_2^2),
\end{equation*}
where $F_*(x)$ is the function given by formula (\ref{DualPrepotential2dim}).
\end{theorem}
\begin{proof}
Given a two-dimensional Frobenius manifold $M,$ with charge $d \neq 1$ and $\eta_{11}=0$ one can construct a two-dimensional Frobenius manifold $\widetilde{M}$ with charge $\widetilde{d}=2-d$ using a symmetry of the WDVV equations known as an inversion \cite{DubrovinNotes}. The flat coordinates $x$ of the intersection form of $M$ may be expressed in terms of the flat coordinates $\widetilde{x}$ of the intersection form of $\widetilde{M}$ via the following relation:
\begin{equation*}
x_i=\frac{2 \, \widetilde{x}_i}{\left(1-\widetilde{d}\right)\left(\widetilde{x}_1^2+\widetilde{x}_2^2\right)},
\end{equation*}
for $i=1, \, 2.$ Moreover, the dual prepotential $\widetilde{F_*}$ of $\widetilde{M}$ may be expressed as
\begin{equation} \label{DualInversion}
\widetilde{F_*}(\widetilde{x})=\frac{4F_*(x(\widetilde{x}))}{(1-d)^2\left(x_1(\widetilde{x})^2+x_2(\widetilde{x})^2\right)^2},
\end{equation}
where $F_*$ is the dual prepotential of $M$ \cite{Ian}. In two dimensions, semisimple Frobenius manifolds with $d \neq 1$ and $\eta_{11}=0$ are uniquely parametrized, up to isomorphism, by their charge \cite{DubrovinNotes}. A Frobenius manifold with prepotential (\ref{2dimPrepotential}) has charge $d=\frac{k-2}{k}.$ Let $M$ be the Frobenius manifold with prepotential (\ref{2dimPrepotential}) with $k=l^{-1},$ thus $M$ has charge $d=1-2l.$ We know from Theorem \ref{1overL} that this Frobenius manifold has a dual prepotential of the form (\ref{DualPrepotential2dim}). The inversion $\widetilde{M}$ must have charge $\widetilde{d}=2l+1$ and therefore its prepotential must be of the form (\ref{2dimPrepotential}) with $k=-l^{-1}.$ The dual prepotential of $\widetilde{M}$ is given by equation (\ref{DualInversion}) from which the statement follows.
\end{proof}

\subsection{$(H_3)''$ and $D_4(a_1)$} \label{7h3d4}

Recall that for a polynomial Frobenius manifold associated to a Coxeter group $W$ with root system $R=R_W,$ the dual prepotential has the form (\ref{CoxeterDualPrepotential}). Let $\alpha \in R$ and define $\alpha_i=(\alpha, \, e_i),$ then we have the following relations (for generic points on $(\alpha, \, x)=0$):
\begin{equation*}
\left((\alpha, \, x)\frac{\partial^3 F_*}{\partial x_i \partial x_j \partial x_k}\right)\bigg\vert_{(\alpha, \, x)=0}=\frac{2\alpha_i\alpha_j\alpha_k}{(\alpha, \, \alpha)},
\end{equation*}
for all $i, j, k=1, \dots, n.$ Below we give related results for the algebraic Frobenius manifolds $(H_3)''$ and $D_4(a_1).$
\begin{proposition} \label{PropZ}
Let $P^M(x, Z)$ be the polynomial from relation (\ref{DeltaPolyH32}) for $M=(H_3)''$ and let it be the polynomial from relation (\ref{DeltaPolyD4}) for $M=D_4(a_1)$ expressed in the $x$ coordinates. Then for each $\alpha \in R,$ where $R=R_{H_3}$ for $M=(H_3)''$ and $R=R_{D_4}$ for $M=D_4(a_1),$ we have that
\begin{equation*}
P^M(x, Z) \vert_{(\alpha, \, x)=0}=K_\alpha^M (L_\alpha^M)^2,
\end{equation*}
where $K_\alpha^M, L_\alpha^M \in \C[x; Z]$ and $L_\alpha^M$ is linear in $Z.$ $K_\alpha^M$ is cubic in $Z$ for $M=(H_3)''$ and quartic in $Z$ for $M=D_4(a_1).$
\end{proposition}
\noindent To check that the polynomial $P^M(x, Z)$ factorises on the hyperplanes $(\alpha, \, x)=0$ we first substitute the expressions for $y_i(x)$ from relations (\ref{H3y1x})--(\ref{H3delta}), or (\ref{D4y1x})--(\ref{D4y4x}), into the left-hand side of equation (\ref{DeltaPolyH32}), or equation (\ref{DeltaPolyD4}), respectively. We then restrict to the hyperplane $(\alpha, \, x)=0$ and see that the polynomial factorises as claimed.

\begin{proposition}
Let $\alpha \in R,$ where $R=R_{H_3}$ for $M=(H_3)''$ and $R=R_{D_4}$ for $M=D_4(a_1).$ The third order derivatives $\overset{*}{c}_{ijk}(x)$ of the dual prepotential $F_*$ of $M=(H_3)''$ or $M=D_4(a_1)$ satisfy
\begin{equation*}
\left((\alpha, \, x) \, \overset{*}{c}_{ijk}(x)\right)\Big\vert_{(\alpha, \, x)=0}=0
\end{equation*}
if $L_\alpha^M(x, Z)=0.$ If $K_\alpha^M(x, Z)=0$ then we have
\begin{equation*}
\left((\alpha, \, x) \, \overset{*}{c}_{ijk}(x)\right)\Big\vert_{(\alpha, \, x)=0}=\frac{2\alpha_i\alpha_j\alpha_k}{(\alpha, \, \alpha)}.
\end{equation*}
\end{proposition}
\begin{proof}
By formulas (\ref{mug}) and (\ref{Above}) we have
\begin{equation*}
\frac{\partial^3 F_*}{\partial x_i \partial x_j \partial x_k}=\overset{*}{c}_{ijk}(x)=g_{i\lambda}(x)g_{j\mu}(x)g_{k\nu}(x)\overset{*}{c}{\vphantom{c}}^{\lambda\mu\nu}(x)=\overset{*}{c}{\vphantom{c}}^{ijk}(x).
\end{equation*}
Then
\begin{equation} \label{cStarH3}
\frac{\partial^3 F_*}{\partial x_i \partial x_j \partial x_k}=\overset{*}{c}{\vphantom{c}}^{\alpha\beta\gamma}(t)\frac{\partial x_i}{\partial t_\alpha}\frac{\partial x_j}{\partial t_\beta}\frac{\partial x_k}{\partial t_\gamma}=g^{\alpha\delta}(t)c^{\beta\gamma}_\delta(t)\frac{\partial x_i}{\partial y_r}\frac{\partial y_r}{\partial t_\alpha}\frac{\partial x_j}{\partial y_s}\frac{\partial y_s}{\partial t_\beta}\frac{\partial x_k}{\partial y_t}\frac{\partial y_t}{\partial t_\gamma}.
\end{equation}
Now we express the right-hand side of (\ref{cStarH3}) in $x$ coordinates and $Z.$ For the terms $g^{\alpha\delta}(t)$ and $c^{\beta\gamma}_\delta(t)$ we apply Theorem \ref{H32inverse}. The derivatives $\frac{\partial x_i}{\partial y_r}, \, \frac{\partial x_j}{\partial y_s}$ and $\frac{\partial x_k}{\partial y_t}$ can be found by inverting the Jacobi matrix $J=\left(\frac{\partial y_i}{\partial x_j}\right).$ The derivatives $\frac{\partial y_r}{\partial t_\alpha}, \, \frac{\partial y_s}{\partial t_\beta}$ and $\frac{\partial y_t}{\partial t_\gamma}$ can be found by Theorem \ref{theoremYH32}. We then reduce the resulting expression for $\overset{*}{c}{\vphantom{c}}_{ijk}(x)$ as a polynomial in $Z$ modulo the relation (\ref{DeltaPolyH32}) for $M=(H_3)''$, or modulo the relation (\ref{DeltaPolyD4}) for $M=D_4(a_1)$.

Then, for any $\alpha \in R$ we get $(\alpha, \, x) \, \overset{*}{c}_{ijk}(x)$ which can be restricted to $(\alpha, \, x)=0.$ Using Proposition \ref{PropZ} we can then reduce the restricted expression as a polynomial in $Z$ modulo $K_\alpha^M$ or modulo $L_\alpha^M$ depending on which branch of $Z$ we consider on the hyperplane. This leads to the claim.
\end{proof}

\section*{Acknowledgements}

We are very grateful to Theo Douvropoulos for useful discussions and for pointing out the reference \cite{Sekiguchi}. We would like to thank Yassir Dinar for sending us a corrected version of the prepotential for $F_4(a_2)$ as well as for other useful discussions. We are also grateful to Ian Strachan for providing useful comments.

The work of J.W. was supported by EPSRC (Engineering and Physical Sciences Research Council) via a postgraduate scholarship. D.V. acknowledges the financial support of the
project MMNLP (Mathematical Methods in Non Linear Physics) of the INFN, and of the University 2022 grant RM1221815BD9120E.

\end{document}